\documentclass[12pt,amstex]{amsart}
\usepackage{amsfonts}
\usepackage{amsfonts}

\usepackage{epsfig}
\usepackage{amsmath}
\usepackage{amssymb}
\usepackage{amscd}
\usepackage{graphicx}

\usepackage{pstricks}

\usepackage{tocvsec2}

\input xy
\xyoption{all}

\newtheorem{thm}{Theorem}[section]
\newtheorem{lem}[thm]{Lemma}

\newtheorem{prop}[thm]{Proposition}
\newtheorem{ques}[thm]{Question}
\newtheorem{cor}[thm]{Corollary}

\theoremstyle{definition}
\newtheorem{de}[thm]{Definition}

\theoremstyle{remark}
\newtheorem{rem}[thm]{Remark}

\numberwithin{equation}{section}

\def \N {\mathbb N}
\def \C {\mathbb C}
\def \Z {\mathbb Z}
\def \R {\mathbb R}

\def \T {\mathbb{T}}
\def \G {\mathbb{G}}

\def \F {\mathcal{F}}

\def \X {\mathcal{X}}

\def \RP {{\bf RP}}
\def \M {{\bf M}}
\def \NN {\mathcal{N}}

\def \a {\alpha }
\def \b {\beta}
\def \ep {\epsilon}
\def \d {\delta}
\def \D {\Delta}
\def \ll {\lambda}

\def \lra{\longrightarrow}
\def \g {\mathfrak{g}}
\def \exp {\text{exp}}
\def \ad {\text{ad}\ }

\begin{document}
\title{Nil Bohr$_0$-sets, Poincar\'e recurrence and generalized polynomials}

\author{Wen Huang}
\author{Song Shao}
\author{Xiangdong Ye}

\address{Department of Mathematics, University of Science and Technology of China,
Hefei, Anhui, 230026, P.R. China.}

\email{wenh@mail.ustc.edu.cn}\email{songshao@ustc.edu.cn}
\email{yexd@ustc.edu.cn}


\thanks{Huang is supported by  Fok Ying Tung Education Foundation,
Shao is supported by NNSF of China (10871186) and Program for New
Century Excellent Talents in University, and Huang+Ye are supported
by NNSF of China (11071231).}

\date{August 24, 2011}

\begin{abstract}

The problem which can be viewed as the higher order version of an
old question concerning Bohr sets is investigated: for any $d\in \N$
does the collection of  $\{n\in \Z: S\cap (S-n)\cap\ldots\cap
(S-dn)\neq \emptyset\}$ with $S$ syndetic coincide with that of
Nil$_d$ Bohr$_0$-sets?

In this paper it is proved that Nil$_d$ Bohr$_0$-sets could be
characterized via generalized polynomials, and applying this result
one side of the problem could be answered affirmatively: for any
Nil$_d$ Bohr$_0$-set $A$, there exists a syndetic set $S$ such that
$A\supset \{n\in \Z: S\cap (S-n)\cap\ldots\cap (S-dn)\neq
\emptyset\}.$ Note that other side of the problem can be deduced
from some result by Bergelson-Host-Kra if modulo a set with zero
density. As applications it is shown that the two collections
coincide dynamically, i.e. both of them can be used to characterize
higher order almost automorphic points.

\end{abstract}

\maketitle

\markboth{Nil-Bohr$_0$ sets}{W. Huang, S. Shao and X.D. Ye}


\section{Introduction}
Combinatorial number theory attracts a lot of attention. In such a
theory, problems concerning Bohr sets are extensively studied, and
have a long history which at least could be traced back to the work
of Veech in 1968 \cite{V68}. Bohr sets are fundamentally abelian in
nature. Nowadays it has become apparent that a higher order
non-abelian Fourier analysis plays a role both in combinatorial
number theory and ergodic theory. Related to this, a higher-order
version of Bohr sets, namely Nil$_d$ Bohr$_0$-sets, was introduced
in \cite{HK10}. For the recent results obtained by Katznelson,
Bergelson-Furstenberg-Weiss and Host-Kra see \cite{Kaz, BFW,HK10}.

\subsection{Nil-Bohr sets}
There are several equivalent definitions for Bohr sets. Here is the
one easy to understand: a subset $A\subseteq \Z$ is a {\em Bohr set}
if there exist $m\in \N$, $\a\in \T^m$, and an open set $U\subseteq
\T^m$ such that $\{n\in \Z: n\a \in U\}$ is contained in $A$; the
set $A$ is a {\em Bohr$_0$-set} if additionally $0\in U$.

It is not hard to see that if $(X,T)$ is a minimal equicontinuous
system, $x\in X$ and $U$ is a neighborhood of $x$, then
$N(x,U)=:\{n\in\Z:T^nx\in U\}$ contains $S-S=:\{a-b: a,b\in S\}$
with $S$ syndetic, i.e. with a bounded gap. An old question
concerning Bohr sets is

\medskip

\noindent {\bf Problem A-I:} \ {Let $S$ be a syndetic set of $\Z$,
is  $S-S$ a Bohr$_0$-set?}

\medskip

That is, are the common differences of arithmetic progressions with
length $2$ appeared in a syndetic set a Bohr$_0$ set? Veech showed
that it is at least ``almost'' true \cite{V68}. That is, given a
syndetic set $S\subseteq \Z$, there is some subset $N$ with density
zero such that $(S-S)\Delta N$ is a Bohr$_0$-set.

\medskip

A subset $A\subseteq \Z$ is a {\em Nil$_d$ Bohr$_0$-set} if there
exist a $d$-step nilsystem $(X,T)$, $x_0\in X$ and an open set
$U\subseteq X$ containing $x_0$ such that $N(x_0,U)=:\{n\in \Z: T^n
x_0\in U\}$ is contained in $A$. Denote by $\F_{d,0}$ the
family\footnote{ A collection $\F$ of subsets of $\Z$ (or $\N$) is
{\em a family} if it is hereditary upward, i.e. $F_1 \subseteq F_2$
and $F_1 \in \F$ imply $F_2 \in \F$. Any nonempty collection
$\mathcal{A}$ of subsets of $\Z$ generates a family $\F(\mathcal{A})
:=\{F \subseteq \Z:F \supset A$ for some $A \in \mathcal{A}\}$.}
consisting of all Nil$_d$ Bohr$_0$-sets. We can now formulate a
higher order form of Problem A-I. We note that $\{n\in \Z: S\cap
(S-n)\cap\ldots\cap (S-dn)\neq \emptyset\}$ can be viewed as the
common differences of arithmetic progressions with length $d+1$
appeared in the subset $S$. In fact, $S\cap (S-n)\cap\ldots\cap
(S-dn)\neq \emptyset$ if and only if there is $m\in S$ with
$m,m+n,\ldots,m+dn\in S$.

\medskip

\noindent {\bf Problem B-I:} [Higher order form of Problem A-I] Let
$d\in\N$.

{\em \begin{enumerate} \item For any Nil$_d$ Bohr$_0$-set $A$, is it
true that there is a syndetic subset $S$ of $\Z$ with
$A\supset\{n\in \Z: S\cap (S-n)\cap\ldots\cap (S-dn)\neq
\emptyset\}$?

\item For any syndetic set $S$, is $\{n\in \Z:
S\cap (S-n)\cap\ldots\cap (S-dn)\neq \emptyset\}$ a Nil$_d$
Bohr$_0$-set?\end{enumerate}}

\medskip

\subsection{Dynamical version of the higher order Bohr problem}
Sometimes combinatorial questions can be translated into dynamical
ones by the Furstenberg correspondence principle, see Section
\ref{FCP}. Using this principle, it can be shown that Problem A-I is
equivalent to the following version:

\medskip

\noindent {\bf Problem A-II:} {\it For any minimal system $(X,T)$
and any nonempty open set $U\subset X$, is the set $\{n\in \Z: U\cap
T^{-n}U\neq \emptyset \}$ a Bohr$_0$-set?}

\medskip

Similarly, Problem B-I has its dynamical version:

\medskip

\noindent {\bf Problem B-II:} [Dynamical version of Problem B-I] Let
$d\in\N$. {\em
\begin{enumerate} \item  For any Nil$_d$ Bohr$_0$-set $A$, it is
true that there are a minimal system $(X,T)$ and a non-empty open
subset $U$ of $X$ with
$$A\supset \{n\in\Z:
U\cap T^{-n}U\cap \ldots \cap T^{-dn}U\not=\emptyset\}?$$

\item For any minimal system $(X,T)$ and any open non-empty $U\subset X$,
is it true that $\{n\in\Z: U\cap T^{-n}U\cap \ldots \cap
T^{-dn}U\not=\emptyset\}$ a Nil$_d$ Bohr$_0$-set?
\end{enumerate}}

\medskip

It follows from some result by Bergelson-Host-Kra in \cite{BHK05}
that Problem B-II(2) has a positive answer if ignoring a set with
zero density. In fact, the authors \cite{BHK05} showed: Let
$(X,\X,\mu, T )$ be an ergodic system and $d\in \N$, then for all
$A\in \X$ with $\mu(A)>0$ the set $I=\{n\in \Z: \mu(A\cap
T^{-n}A\cap \ldots \cap T^{-dn}A)>0\}$ is almost a Nil$_d$
Bohr$_0$-set, i.e. there is some subset $N$ with density zero such
that $I \Delta N$ is a Nil$_d$ Bohr$_0$-set.

\subsection{Main results}

We will show that Problem B-II(1) has an affirmative answer. Namely,
we will show

\medskip

\noindent {\bf Theorem A}: {\em Let $d\in\N$.  If $A\subseteq\Z$ is
a Nil$_d$ Bohr$_0$-set, then there exist a minimal $d$-step
nilsystem $(X,T)$ and a nonempty open set $U$ of $X$ with
\begin{equation*}
    A\supset \{n\in\Z: U\cap T^{-n}U\cap \ldots\cap T^{-dn}U\neq \emptyset \}.
\end{equation*}}
As we said before for $d=1$ Theorem A can be easily proved. To show
Theorem~ A in the general case, we need to investigate the
properties of $\F_{d,0}$. It is interesting that in the process to
do this, generalized polynomials (see \S  \ref{section-GP} for a
definition) appear naturally. Generalized polynomials have been
studied extensively, see for example the nice paper by Bergelson and
Leibman \cite{BL07} and references therein. After finishing this
paper we even find that it also plays an important role in the
recent work by Green, Tao and Ziegler \cite{GTZ}. In fact the
special generalized polynomials defined in this paper are closely
related to the nilcharacters defined there.

Let $\F_{GP_d}$ (resp. $\F_{SGP_d}$) be the family generated by the
sets of forms
$$\bigcap_{i=1}^k\{n\in\Z:P_i(n)(\text{mod}\ \Z)\in (-\ep_i,\ep_i)
\},$$ where $k\in\N$, $P_1,\ldots,P_k$ are generalized polynomials
of degree $\le d$ (resp. special generalized polynomials), and
$\ep_i>0$. For the precise definitions see \S\ref{section-GP}.

The following theorem illustrates the relation between Nil$_d$
Bohr$_0$-sets and the sets defined above using generalized
polynomials.

\medskip

\noindent {\bf Theorem B}: {\em Let $d\in\N$. Then
$\F_{d,0}=\F_{GP_d}$.}

\medskip
To prove Theorem B we first figure out a subclass of generalized
polynomials (called special generalized polynomials) and show that
$\F_{GP_d}=\F_{SGP_d}$. When $d=1$, we have $\F_{1,0}=\F_{SGP_1}$.
This is the result of Katznelson \cite{Kaz}, since $\F_{SGP_1}$ is
generated by sets of forms $\cap_{i=1}^k \{n\in\Z: na_i(\text{mod}\
\Z)\in (-\ep_i,\ep_i)\}$ with $k\in\N$, $a_i\in\R$ and $\ep_i>0$.

Theorem A follows from Theorem B and the following result:

\medskip

\noindent {\bf Theorem C}: {\em Let $d\in\N$.  If $A\in \F_{GP_d}$,
then there exist a minimal $d$-step nilsystem $(X,T)$ and a nonempty
open set $U$ such that
$$A\supset \{n\in\Z: U\cap T^{-n}U\cap \ldots\cap T^{-dn}U\neq \emptyset \}.$$}

The proof of Theorem B is divided into two parts, namely

\medskip

\noindent {\bf Theorem B(1)}: {\em $\F_{d,0}\subset\F_{GP_d}$} and

\medskip

\noindent {\bf Theorem B(2)}: {\em $\F_{d,0}\supset\F_{GP_d}$.}

\medskip

The proof of Theorem B(1) is a theoretical argument using nilpotent
Lie group theory; and the proofs of Theorem B(2) and Theorem C are
very complicated construction and computation where nilpotent matrix
Lie group is used.

\begin{rem} Our definition of generalized polynomials is slight different from the ones
defined in \cite{BL07}. In fact we need to specialize the degree of
the generalized polynomials which is not needed in \cite{BL07}.
Moreover, our Theorem B can be compared with Theorem A of Bergelson
and Leibman proved in \cite{BL07}.
\end{rem}

In \cite{F, F81} Furstenberg introduced the notion of Poincar\'e
recurrence sets and Birkhoff recurrence sets. Here is a
generalization of the above notion. Let $d\in \N$. We say that $S
\subset \Z$ is a set of {\em $d$-recurrence } if for every measure
preserving dynamical system $(X,\X,\mu,T)$ and for every $A\in \X$
with $\mu (A)>0$, there exists $n \in S$  such that
$$\mu(A\cap T^{-n}A\cap \ldots \cap T^{-dn}A)>0.$$

We say that $S\subseteq \Z$ is a set of {\em $d$-topological
recurrence} if for every minimal system $(X, T)$ and for every
nonempty open subset $U$ of $X$, there exists $n\in S$  such that
$$U\cap T^{-n}U\cap \ldots \cap T^{-dn}U\neq \emptyset.$$

\begin{rem} The above definitions are slightly different from the
ones introduced in \cite{FLW}, namely we do not require $n\not=0$.
The main reason we define in this way is that for each $A\in
\F_{d,0}$, $0\in A$. Thus $\{0\}\cup C\in \F^*_{d,0}$ for each
$C\subset \Z$.
\end{rem}

Let $\F_{Poi_d}$ (resp. $\F_{Bir_d}$) be the family generated by the
collection of all sets of $d$-recurrence (resp. sets of
$d$-topological recurrence). It is obvious by the above definition
that $\F_{Poi_d}\subset \F_{Bir_d}$. Moreover, it is known that for
each $d\in\N$, $\F_{Poi_d}\supsetneqq \F_{Poi_{d+1}}$ and
$\F_{Bir_d}\supsetneqq \F_{Bir_{d+1}}$ \cite{FLW}. Now we state a
problem which is related to Problem B-II.

\medskip

\noindent {\bf Problem B-III:} {\em Is it true that $\F_{Bir_d} =
\F^*_{d,0}$?}

\medskip

\noindent where $\F_{d,0}^*$ is the dual family of $\F_{d,0}$, i.e.
the collection of sets intersecting every Nil$_d$ Bohr$_0$ set.
\medskip

An immediate corollary of Theorem A is:

\medskip

\noindent {\bf Corollary D}: {\em Let $d\in\N$. Then
$$\F_{Poi_d}\subset\F_{Bir_d}\subset \F^*_{d,0}.$$ }
\medskip

Note that $\F_{Poi_1}\neq \F_{Bir_1}$ \cite{K}. Though we can not
prove $\F_{Bir_d}=\F^*_{d,0}$, we will show that the two collections
coincide ``dynamically'', i.e. both of them can be used to
characterize higher order almost automorphic points, see \S
\ref{section-8}.

\subsection{Organization of the paper}

We organize the paper as follows: In Section \ref{section-pre}, we
give some basic definitions, and particularly we show the
equivalence of the Problems I, II and III. In Section
\ref{section-nilsystem} we recall basic facts related to nilpotent
Lie groups and nilmanifolds, and study the properties of the metric
on nilpotent matrix Lie groups. In Section \ref{section-GP}, we
introduce the notions related to generalized polynomials and special
generalized polynomials, and give the basic properties. In the next
three sections we show the main results. And in the last section, we
state the applications of our main results, which will appear in a
forthcoming article \cite{HSY} by the same authors.

\subsection{Thanks} We thank Bergelson, Frantzikinakis and Glasner for useful
comments.

\section{Preliminaries}\label{section-pre}

In this section we introduce some basic notions related to dynamical
systems,explain how Bergelson-Host-Kra's result is related to
Problem B-II and show the equivalence of the Problems I, II and III.

\subsection{Measurable and topological dynamics}


A (measurable) {\em system} is a quadruple $(X,\X, \mu, T)$, where
$(X,\X,\mu )$ is a Lebesgue probability space and $T : X \rightarrow
X$ is an invertible measure preserving transformation.

\medskip


A {\em topological dynamical system}, referred to more succinctly as
just a {\em system}, is a pair $(X, T)$, where $X$ is a compact
metric space and $T : X \rightarrow  X$ is a homeomorphism. We use
$\rho(\cdot, \cdot)$ to denote the metric on $X$.

\subsection{Families and filters}
Since many statements of the paper are better stated using the
notion of a family, we now give the definition. See \cite{A} for
more details.

\subsubsection{Furstenberg families}

We say that a collection $\F$ of subsets of $\Z$  is  {\em
a family} if it is hereditary upward, i.e. $F_1 \subseteq F_2$ and
$F_1 \in \F$ imply $F_2 \in \F$. A family $\F$ is called {\em
proper} if it is neither empty nor the entire power set of $\Z$, or,
equivalently if $\Z \in \F$ and $\emptyset \not\in \F$. Any nonempty
collection $\mathcal{A}$ of subsets  of $\Z$ generates a family
$\F(\mathcal{A}) :=\{F \subseteq \Z:F \supset A$ for some $A \in
\mathcal{A}\}$.

For a family $\F$ its {\em dual} is the family $\F^{\ast}
:=\{F\subseteq \Z  : F \cap F' \neq \emptyset \ \text{for  all} \ F'
\in \F \}$. It is not hard to see that $\F^*=\{F\subset\Z:
\Z\setminus F\not\in \F\}$, from which we have that if $\F$ is a
family then $(\F^*)^*=\F.$

\subsubsection{Filter and Ramsey property}

If a family $\F$ is closed under finite intersections and is proper,
then it is called a {\em filter}.

A family $\F$ has the {\em Ramsey property} if $A=A_1\cup A_2\in \F$
then $A_1\in \F$ or $A_2\in \F$. It is well known that a proper family has
the Ramsey property if and only if its dual $\F^*$ is a filter
\cite{F}.


A subset $S$ of $\Z$ is {\it syndetic} if it has a bounded gap, i.e.
there is $N\in \N$ such that $\{i,i+1,\cdots,i+N\} \cap S \neq
\emptyset$ for every $i \in {\Z}$.
A subset $S$ is an {\it $IP$-set}, if there is a subsequence
$\{p_i\}$ of $\mathbb{Z}$ such that
$$S\supset \{p_{i_1}+\ldots+p_{i_n}:i_1<\ldots<i_n,n\in\N\}.$$
It is known that the family of all $IP^*$-sets is a filter and each
$IP^*$-set is syndetic \cite{F}.


The {\it upper Banach density} and {\it lower Banach density} of $S$
are
$$BD^*(S)=\limsup_{|I|\to \infty}\frac{|S\cap I|}{|I|},\ \text{and}\
BD_*(S)=\liminf_{|I|\to \infty}\frac{|S\cap I|}{|I|},$$ where $I$
ranges over intervals of $\mathbb{Z}$, while the {\it upper density}
of $S$ and the {\it lower density} of $S$ are
$$D^*(S)=\limsup_{n\to \infty}\frac{|S\cap [-n,n]|}{2n+1},\ \text{and}\ D_*(S)=\liminf_{n\to \infty}\frac{|S\cap [-n,n]|}{2n+1}.$$
If $D^*(S)=D_*(S)$, then we say the {\it density} of $S$ is
$D(S)=D^*(S)=D_*(S)$.


\subsection{A Bergelson-Host-Kra' Theorem and a consequence}
In this subsection we explain how Bergelson-Host-Kra's result is
related to Problem B-II. First we need some definitions.

\begin{de} Let $k\ge 1$ be an integer and let $X = G/\Gamma$ be a
$d$-step nilmanifold. Let $\phi$ be a continuous real (or complex)
valued function on $X$ and let $a\in G$ and $b\in X$. The sequence
$\{\phi(a^n\cdot b)\}$ is called a basic $d$-step nilsequence. A
$d$-step nilsequence is a uniform limit of basic $d$-step
nilsequences.
\end{de}

For the definition of nilmanifolds see Section
\ref{section-nilsystem}.

\begin{de} Let $\{a_n : n\in \Z\}$ be a bounded sequence. We say
that $a_n$ tends to zero in uniform density, and we write
$\text{UD-Lim}\ a_n = 0,$ if $$\lim_{N\lra+\infty}
\sup_{M\in\Z}\sum_{n=M}^{M+N-1}|a_n| = 0.$$
\end{de}

Equivalently, $\text{UD-Lim}\ a_n = 0$ if and only if for any $\ep>
0,$ the set $\{n\in \Z : |a_n| >\ep\}$ has upper Banach density
zero. Now we state their result.

\begin{thm}[Bergelson-Host-Kra]\cite[Theorem 1.9]{BHK05}
Let $(X,\X,\mu, T )$ be an ergodic system, let $f \in L^\infty(\mu)$
and let $d\ge 1$ be an integer. The sequence $\{I_f(d,n)\}$ is the
sum of a sequence tending to zero in uniform density and a $d$-step
nilsequence, where
\begin{equation}\label{}
  I_f(d,n)=\int f(x)f(T^nx)\ldots f(T^{dn}x)\ d\mu(x).
\end{equation}
\end{thm}

Especially, for any $A\in \X$
\begin{equation}\label{noteasy}
\{I_{1_A}(d,n)\}=\{\mu(A\cap T^{-n}A\cap \ldots \cap
T^{-dn}A)\}=F_d+N, \end{equation} where $F_d$ is a $d$-step
nilsequence and $N$ tending to zero in uniform density. Regard $F_d$
as a function $F_d: \Z\rightarrow \C$. By \cite{HM} there is a
$d$-step nilsystem $(Z, S)$, $x_0\in Z$ and a continuous function
$\phi\in C(Z)$ such that
\begin{equation*}
    F_d(n)=\phi(S^nx_0).
\end{equation*}
We claim that $\phi(x_0)>0$ if $\mu(A)>0$. Assume that contrary that
$\phi(x_0)\le 0$. By \cite{FK} or \cite[Theorem 6.15]{BM20} there is
$c>0$ such that
$$\{n\in \Z: \mu(A\cap T^{-n}A\cap \ldots \cap T^{-dn}A)>c\}$$ is an
$IP^*$-set. On the other hand there is a small neighborhood $V$ of
$x_0$ such that $\phi(x)< \frac{1}{2}c$ for each $x\in V$ by the
continuity of $\phi$. It is known that $N(x_0,V)$ is an $IP^*$-set
\cite{F} since $(Z,S)$ is distal \cite[Ch 4, Theorem 3]{AGH} or
\cite{L}. This contradicts to (\ref{noteasy}) by the facts that the
family of $IP^*$-sets is a filter, each $IP^*$-set is syndetic and
$N(n)$ tends to zero in uniform density. That is, we have shown that
$\phi(x_0)>0$ if $\mu(A)>0$.

\medskip
Hence if $\mu(A)>0$ then for each $\ep>0$, $\{n\in\Z:
\phi(S^nx_0)>\phi(x_0)-\frac{1}{2}\ep\}$ is a Nil$_d$-Bohr$_0$ set.
Since $\{n\in\Z: |N(n)|>\frac{1}{2}\ep\}$ has zero upper Banach
density we have the following corollary

\begin{cor}\label{d-recurrence}
Let $(X,\X,\mu, T )$ be an ergodic system and $d\in \N$. Then for
all $A\in \X$ with $\mu(A)>0$ and $\ep>0$, the set $$I=\{n\in \Z:
\mu(A\cap T^{-n}A\cap \ldots \cap T^{-dn}A)>\phi(x_0)-\ep\}$$ is an
almost Nil$_d$ Bohr$_0$-set, i.e. there is some subset $M$ with
$BD^*(M)=0$ such that $I \Delta M$ is a Nil$_d$ Bohr$_0$-set.
\end{cor}

It follows that problem B-II(2) has a positive answer ignoring a set
with zero density, since for a minimal system $(X,T)$, each
invariant measure of $(X,T)$ is fully supported.

\subsection{Furstenberg correspondence principle}\label{FCP}

Let $\F(\Z)$ denote the collection of finite non-empty subsets of
$\Z$. It is well known that

\begin{thm}[Topological case]\label{topocase}

\begin{enumerate}

\item Let $E\subseteq \Z$ be a syndetic set. Then there
exist a minimal system $(X,T)$ and a non-empty open set $ U\subseteq
X$ such that
\begin{equation*}
\{\a\in \F(\Z): \bigcap_{n\in \a}T^{-n}U\neq \emptyset\}\subseteq
\{\a\in \F(\Z): \bigcap_{n\in \a}(E-n)\neq \emptyset \}.
\end{equation*}

\item For any minimal system $(X,T)$ and any open non-empty set $U$, there is a syndetic
set $E$ such that
\begin{equation*}
\{\a\in \F(\Z): \bigcap_{n\in \a}(E-n)\neq \emptyset \}\subseteq
\{\a\in \F(\Z): \bigcap_{n\in \a}T^{-n}U\neq \emptyset\}.
\end{equation*}

\end{enumerate}

\end{thm}

\begin{thm}[Measurable case]
\begin{enumerate}
  \item Let $E\subseteq \Z$ with $BD^*(E)>0$. Then there
  exists a measurable system $(X,\X,\mu,T)$ and $ A\in \X$ with $\mu(A)=BD^*(E)$ such that
  for all $\a\in \F(\Z)$
  \begin{equation*}
    BD^*(\bigcap_{n\in \a}(E-n))\ge \mu(\bigcap_{n\in \a}T^{-n}A).
  \end{equation*}

  \item Let $(X,\X,\mu,T)$ be a measurable system and $ A\in \X$ with $\mu(A)>0$.
  There is a set $E$ with $D^*(E)\ge \mu(A)$ such that
  \begin{equation*}
    \{\a\in \F(\Z): \bigcap_{n\in \a}(E-n)\neq \emptyset
    \}\subseteq \{\a\in \F(\Z): \mu(\bigcap_{n\in \a}T^{-n}U
    )>0\}.
  \end{equation*}

\end{enumerate}

\end{thm}

\subsection{Equivalence} In this subsection we explain why Problems
B-I,II,III are equivalent. Let $\F$ be the family generated by all
sets of forms $\{n\in\Z: U\cap T^{-n}U\cap \ldots \cap
T^{-dn}U\not=\emptyset\},$ with $(X,T)$ a minimal system, $U$ a
non-empty open subset of $X$. Then it is clear from the definition
that
$$\F_{Bir_d}=\F^*.$$
\begin{prop} For any $d\in\N$ the following statements are equivalent.
\begin{enumerate}
\item For any Nil$_d$ Bohr$_0$-set $A$, there is a syndetic subset $S$ of $\Z$ with
$A\supset\{n\in \Z: S\cap (S-n)\cap\ldots\cap (S-dn)\neq
\emptyset\}$.
\item  For any Nil$_d$ Bohr$_0$-set $A$, there are a minimal system $(X,T)$ and a
non-empty open subset $U$ of $X$ with $A\supset \{n\in\Z: U\cap
T^{-n}U\cap \ldots \cap T^{-dn}U\not=\emptyset\}.$
\item $\F_{Bir_d}\subset \F^*_{d,0}$.
\end{enumerate}
\end{prop}
\begin{proof} Let $d\in\N$ be fixed. (1)$\Rightarrow$(2). Let $A$ be a Nil$_d$
Bohr$_0$-set, then there is a syndetic subset $S$ of $\Z$ with
$A\supset\{n\in \Z: S\cap (S-n)\cap\ldots\cap (S-dn)\neq
\emptyset\}$. For such $S$ using Theorem \ref{topocase}, we get that
there exist a minimal system $(X,T)$ and a non-empty open set $
U\subseteq X$ such that $\{n\in \Z: S\cap (S-n)\cap\ldots\cap
(S-dn)\neq \emptyset\}\supset\{n\in\Z: U\cap T^{-n}U\cap \ldots \cap
T^{-dn}U\not=\emptyset\}.$ Thus $A\supset\{n\in\Z: U\cap T^{-n}U\cap
\ldots \cap T^{-dn}U\not=\emptyset\}.$ (2)$\Rightarrow$(1) follows
similarly by the above argument. (2)$\Rightarrow$(3) follows by the
definition. (3)$\Rightarrow$(2). Since $\F_{Bir_d}\subset
\F^*_{d,0}$ and $\F_{Bir_d}=\F^*$, we have that $\F^*\subset
\F_{d,0}^*$ which implies that $\F\supset \F_{d,0}$.
\end{proof}

\begin{prop} For any $d\in\N$ the following statements are equivalent.
\begin{enumerate}
\item For any syndetic set $S$, $\{n\in \Z:
S\cap (S-n)\cap\ldots\cap (S-dn)\neq \emptyset\}$ is a Nil$_d$
Bohr$_0$-set.

\item For any minimal system $(X,T)$, and any open non-empty $U\subset
X$, $\{n\in\Z: U\cap T^{-n}U\cap \ldots \cap
T^{-dn}U\not=\emptyset\}$ is a Nil$_d$ Bohr$_0$-set.

\item $\F_{Bir_d}\supset \F^*_{d,0}$.
\end{enumerate}
\end{prop}
\begin{proof} Let $d\in\N$ be fixed. (1)$\Rightarrow$(2). Let
$(X,T)$ be a minimal system and $U$ be a non-empty open set of $X$.
By Theorem \ref{topocase}, there is a syndetic set $S$ such that
$$\{n\in
\Z: S\cap (S-n)\cap\ldots\cap (S-dn)\neq \emptyset\}\subset\{n\in\Z:
U\cap T^{-n}U\cap \ldots \cap T^{-dn}U\not=\emptyset\}.$$ By (1),
$\{n\in \Z: S\cap (S-n)\cap\ldots\cap (S-dn)\neq \emptyset\}$ is a
Nil$_d$ Bohr$_0$-set, and so is $\{n\in\Z: U\cap T^{-n}U\cap \ldots
\cap T^{-dn}U\not=\emptyset\}.$ Similarly, we have
(2)$\Rightarrow$(1). (2)$\Rightarrow$(3) follows by the definition.
(3)$\Rightarrow$(2). By Corollary D we have $\F_{d,0}^*=\F^*$, i.e.
$\F_{d,0}=\F$.
\end{proof}




\subsection{Other related problems} We remark that similar
problems can be formulated replacing syndetic sets by sets with
positive Banach density, minimal systems by ergodic systems and open
non-empty sets by positive measurable sets.

\section{Nilsystems}\label{section-nilsystem}

In this section we recall some basic facts concerning nilpotent Lie
groups and nilmanifolds. Since in the proofs of our main results we
need to use the metric in the nilpotent matrix Lie group, we state
some its basic properties. Note that we follow Green and Tao
\cite{GT} to define such a metric.

\subsection{Nilmanifolds and nilsystems}

\subsubsection{Nilpotent groups}

Let $G$ be a group. For $g, h\in G$, we write $[g, h] =
ghg^{-1}h^{-1}$ for the commutator of $g$ and $h$ and we write
$[A,B]$ for the subgroup spanned by $\{[a, b] : a \in A, b\in B\}$.
The commutator subgroups $G_j$, $j\ge 1$, are defined inductively by
setting $G_1 = G$ and $G_{j+1} = [G_j ,G]$. Let $d \ge 1$ be an
integer. We say that $G$ is {\em $d$-step nilpotent} if $G_{d+1}$ is
the trivial subgroup.

\subsubsection{Nilmanifolds}

Let $G$ be a $d$-step nilpotent Lie group and $\Gamma$ a discrete
cocompact subgroup of $G$, i.e. a uniform subgroup of $G$. The
compact manifold $X = G/\Gamma$ is called a {\em $d$-step
nilmanifold}. The group $G$ acts on $X$ by left translations and we
write this action as $(g, x)\mapsto gx$. The Haar measure $\mu$ of
$X$ is the unique probability measure on $X$ invariant under this
action. Let $\tau\in G$ and $T$ be the transformation $x\mapsto \tau
x$ of $X$, i.e the nilrotation induced by $\tau\in G$. Then $(X, T,
\mu)$ is called a basic {\em $d$-step nilsystem}. See \cite{CG,M}
for the details.

\subsubsection{$d$-step nilsystem and system of order $d$}

We also make use of inverse limits of nilsystems and so we recall
the definition of an inverse limit of systems (restricting ourselves
to the case of sequential inverse limits). If $(X_i,T_i)_{i\in \N}$
are systems with $diam(X_i)\le M<\infty$ and $\phi_i:
X_{i+1}\rightarrow X_i$ are factor maps, the {\em inverse limit} of
the systems is defined to be the compact subset of $\prod_{i\in
\N}X_i$ given by $\{ (x_i)_{i\in \N }: \phi_i(x_{i+1}) = x_i,
i\in\N\}$, which is denoted by $\displaystyle
\lim_{\longleftarrow}\{X_i\}_{i\in \N}$. It is a compact metric
space endowed with the distance $\rho(x, y) =\sum_{i\in \N} 1/2^i
\rho_i(x_i, y_i )$. We note that the maps $\{T_i\}$ induce a
transformation $T$ on the inverse limit.

\begin{de}\label{de-nilsystem} [Host-Kra-Maass] \cite{HKM}
A system $(X,T)$ is called a {\em $d$-step nilsystem}, if it is an inverse limit of basic
$d$-step nilsystems. A system $(X,T)$ is called
a {\em system of order $d$}, if  it is a minimal $d$-step nilsystem, equivalently it is an inverse limit of basic
$d$-step minimal nilsystems.
\end{de}

Recall that a subset $A\subseteq \Z$ is a {\em Nil$_d$ Bohr$_0$-set}
if there exist a $d$-step nilsystem $(X,T)$, $x_0\in X$ and an open
set $U\subseteq X$ containing $x_0$ such that $N(x_0,U)$ is
contained in $A$. As each basic $d$-step nilsystem is distal, so is
a $d$-step nilsystem. Hence by Definition \ref{de-nilsystem}, it is
not hard to see that a subset $A\subseteq \Z$ is a {\em Nil$_d$
Bohr$_0$-set} if and only if there exist a basic $d$-step (minimal)
nilsystem $(X,T)$ (or a {\em system $(X,T)$ of order $d$}), $x_0\in
X$ and an open set $U\subseteq X$ containing $x_0$ such that
$N(x_0,U)$ is contained in $A$. Note that here we need the facts
that the product of finitely many of $d$-step nilmanifolds is a
$d$-step nilmanifold,
and the orbit closure of any point in a basic $d$-step nilsystem is
a $d$-step nilmanifold \cite[Theorem 2.21]{L}.

\subsection{Reduction}\label{reduction}
Let $X=G/\Gamma$ be a nilmanifold. Then there exists a connected,
simply connected nilpotent Lie group $\widehat{G}$ and
$\widehat{\Gamma}\subseteq \widehat{G}$ a co-compact subgroup such
that $X$ with the action of $G$ is isomorphic to a submanifold
$\widetilde{X}$ of $\widehat{X}=\widehat{G}/\widehat{\Gamma}$
representing the action of $G$ in $\widehat{G}$. See \cite{L} for
more details.

Thus a subset $A\subseteq \Z$ is a {\em Nil$_d$ Bohr$_0$-set} if and only if
there exist a basic $d$-step nilsystem $(G/\Gamma,T)$ with $G$ is a connected,
simply connected nilpotent Lie group and
$\Gamma$ a co-compact subgroup of $G$, $x_0\in X$ and an open set
$U\subseteq X$ containing $x_0$ such that $N(x_0,U)$ is contained in $A$.

\subsection{Nilpotent Lie group and Mal'cev basis}

\subsubsection{}

We will make use of the Lie algebra $\mathfrak{g}$ of a $d$-step
nilpotent Lie group $G$ together with the exponential map $\exp: \g
\lra G$. When $G$ is a connected, simply-connected $d$-step
nilpotent Lie group the exponential map is a diffeomorphism
\cite{CG, M}. In particular, we have a logarithm map $\log: G\lra
\g$. Let $$\exp(X*Y)=\exp(X)\exp(Y), \ X,Y\in \g.$$

\subsubsection{Campbell-Baker-Hausdorff formula}

The following Campbell-Baker-Hausdorff formula (CBH formula) will be
used frequently
\begin{equation*}
\begin{aligned}
X*Y&=\sum_{n>0}\frac{(-1)^{n+1}}{n}\sum_{p_i+q_i>0,1\le i\le
n}\frac{(\sum_{i=1}^n(p_i+q_i))^{-1}}{p_1!q_1!\ldots p_n!q_n!}\\
&\times (\ad X)^{p_1}(\ad Y)^{q_1}\ldots (\ad X)^{p_n}(\ad
Y)^{q_n-1}Y,
\end{aligned}
\end{equation*}
where $(\ad X)Y=[X,Y]$. (If $q_n=0$, the term in the sum is $\ldots
(\ad X)^{p_n-1}X$; of course if $q_n>1$, or if $q_n=0$ and $p_n>1$,
then the term if zero.) The low order nonzero terms are well known,
\begin{equation*}
\begin{aligned}
X*Y=& X+Y+\frac{1}{2}[X,Y]+\frac{1}{12}[X,[X,Y]]-\frac{1}{12}[Y,[X,Y]]\\
&-\frac{1}{48}[Y,[X,[X,Y]]]-\frac{1}{48}[X,[Y,[X,Y]]]\\
&+ (\text{ commutators in five or more terms}).
\end{aligned}
\end{equation*}

\subsubsection{}

We assume $\mathfrak{g}$ is the Lie algebra of $G$ over $\R$, and
$\exp: \g \lra G$ is the exponential map. The descending central
series of $\g$ is defined inductively by
\begin{equation*}
   \g^{(1)}=\g; \ \g^{(n+1)}=[\g,\g^{(n)}]={\rm span}\{[X,Y]:
   X\in \g, Y\in \g^{(n)}\}.
\end{equation*}
Since $\g$ is a $d$-step nilpotent Lie algebra, we have
$$\g=\g^{(1)}\supsetneq\g^{(2)}\supsetneq \ldots\supsetneq \g^{(d)} \supsetneq\g^{(d+1)}=\{0\}.$$
We note that $$[\g^{(i)},\g^{(j)}]\subset \g^{(i+j)}, \forall i,
j\in \N.$$ In particular, each $\g^{(k)}$ is an ideal in $\g$.

\subsubsection{Mal'cev Base}

\begin{de}(Mal'cev base)
Let $G/\Gamma$ be an $m$-dimensional nilmanifold (i.e. $G$ is a
$d$-step nilpotent Lie group and $\Gamma$ is a discrete uniform
subgroup of $G$) and let $G=G_1\supset \ldots\supset G_d\supset
G_{d+1}=\{e\}$ be the lower central series filtration. A basis
$\mathcal {X}= \{X_1, \ldots ,X_m\}$ for the Lie algebra $\g$ over
$\R$ is called a {\em Mal'cev basis} for $G/\Gamma$ if the following
four conditions are satisfied:

\begin{enumerate}
\item  For each $j = 0,\ldots,m-1$ the subspace $\eta_j := \text{Span}(X_{j+1}, \ldots
,X_m)$ is a Lie algebra ideal in $\g$, and hence $H_j := \exp\
\eta_j$ is a normal Lie subgroup of $G$.

\item For every $0< i<d$ we have $G_i = H_{l_{i-1}+1}$. Thus
$0=l_0<l_1<\ldots<l_{d-1}\le m-2$.

\item Each $g\in G$ can be written uniquely as
$\exp(t_1X_1) \exp(t_2X_2)\ldots \exp(t_mX_m)$, for $t_i\in \R$.

\item $\Gamma$ consists precisely of those elements which, when written in
the above form, have all $t_i\in \Z$.
\end{enumerate}
\end{de}
Note that such a basis exists when $G$ is a connected, simply
connected $d$-step nilpotent Lie group \cite{CG, GT, M}.

\subsection{Metrics on nilmanifolds}
For a connected, simply connected $d$-step nilpotent Lie group $G$,
we can use a Mal'cev basis $X$ to put a metric structure on $G$ and
on $G/\Gamma$.

\begin{de}[Metrics on G and $G/\Gamma$]\cite{GT}\label{gtao}
Let $G/\Gamma$ be a nilmanifold with a Mal'cev basis $\mathcal X$,
where $G$ is a connected, simply connected Lie group and $\Gamma$ is
a discrete uniform subgroup of $G$. Let $\phi:G\lra \R^m$ with
$$g=\exp(t_1X_1)\ldots\exp(t_mX_m)\mapsto (t_1,\ldots,t_m).$$ We
define $\rho = \rho_X: G \times G\lra \R$ to be the largest metric
such that $\rho(x, y)\le |\phi(xy^{-1})|$ for all $x, y\in G$, where
$|\cdot|$ denotes the $\ell^\infty$-norm on $\R^m$. More explicitly,
we have
$$\rho (x, y) = \inf\Big\{\sum_{i=1}^n\min\{|\phi(x_{i-1}x_i^{-1})|,
|\phi(x_ix_{i-1}^{-1})|\}: x_0,\ldots, x_n\in G;
x_0=x,x_n=y\Big\}.$$ This descends to a metric on $G/\Gamma$ by
setting
$$\rho (x\Gamma,y\Gamma):=\inf\{d(x', y'): x', y'\in G; x'=x(\text{mod}\ \Gamma);
y'=y(\text{mod}\ \Gamma)\}.$$ It turns out that this is indeed a
metric on $G/\Gamma$, see \cite{GT}. Since $\rho$ is
right-invariant, we also have
$$\rho (x\Gamma, y\Gamma) = \inf_{\gamma\in \Gamma}\rho (x, y\gamma).$$
\end{de}

\subsection{Base points}
The following proposition should be well known.

\begin{prop}\label{replace}
Let $X=G/\Gamma$ be a nilmanifold, $T$ be a nilrotation induced by
$a\in G$. Let $x\in G$ and $U$ be an open neighborhood $x\Gamma$ in
$X$. Then there are a uniform subgroup $\Gamma_x\subset G$ and an
open neighborhood $V\subset G/\Gamma_x$ of $e\Gamma_x$ such that
$$N_T(x\Gamma,U)=N_{T'}(e\Gamma_x,V),$$
where $T'$ is a nilrotation induced by $a\in G$ in $X'=G/\Gamma_x$.
\end{prop}
\begin{proof} Let $\Gamma_x=x\Gamma x^{-1}$. Then $\Gamma_x$ is also a uniform subgroup of $G$.

Put $V=Ux^{-1}$, where we view $U$ as the collections of equivalence
classes. It is easy to see that $V\subset G/\Gamma_x$ is open, which
contains $e\Gamma_x$. Let $n\in N_T(x\Gamma,U)$ then $a^nx\Gamma\in
U$ which implies that $a^nx\Gamma x^{-1}\in Ux^{-1}=V$, i.e. $n\in
N_{T'}(e\Gamma_x,V)$. The other direction follows similarly.
\end{proof}

\subsection{Nilpotent Matrix Lie Group}

\subsubsection{}

Let $M_{d+1}(\mathbb{R})$ denote the space of all $(d+1)\times
(d+1)$-matrices with real entries. For $A=(A_{ij})_{1\le i,j\le
d+1}\in M_{d+1}(\mathbb{R})$, we define
\begin{equation}\label{Mat-eq-1}
\|A\|=\left( \sum \limits_{i,j=1}^{d+1}
|A_{ij}|^2\right)^{\frac{1}{2}}.
\end{equation}
Then $\|\cdot\|$ is a norm on $M_{d+1}(\mathbb{R})$ and the norm
satisfies the inequalities
\begin{align*}
\|A+B\|\le \|A\|+\|B\| \text{ and } \|AB\|\le \|A\| \|B\|
\end{align*}
for $A,B\in M_{d+1}(\mathbb{R})$.

\subsubsection{}

Let ${\bf a}=(a_i^k)_{1\le k\le d, 1\le i\le d-k+1}\in
\mathbb{R}^{d(d+1)/2}$. Then corresponding to ${\bf a}$ we define
$\M({\bf a})$ with
{\footnotesize
$$
\M({\bf a})=\left(
  \begin{array}{cccccccc}
    1 & a_1^1 & a_1^2 & a_1^3  &\ldots & a_1^{d-1} &a_1^d \\
    0 & 1     & a_2^1 & a_2^2   &\ldots &  a_2^{d-2}&a_2^{d-1}\\
    0 & 0     &1      & a_3^1   &\ldots &  a_3^{d-3} &a_3^{d-2}\\
   \vdots & \vdots & \vdots &  \vdots & \vdots & \vdots &\vdots\\
    0 & 0    &0       &0           & \ldots & a_{d-1}^1 &a_{d-1}^2 \\
    0 & 0    &0       &0            & \ldots & 1 &a_{d}^1 \\
    0 & 0    &0       &0            & \ldots & 0  & 1
  \end{array}
\right).
$$}

\subsubsection{}

Let $\mathbb{G}_d$ be the (full) upper triangular group
$$\G_d=\{\M({\bf a}):a_i^k\in \mathbb{R}, 1\le k\le d, 1\le i\le d-k+1\}.$$
The group $\G_d$ is a $d$-step nilpotent group, and it is clear that
for $A\in \G_d$ there exists a unique ${\bf c}=(c_i^k)_{1\le k\le d,
1\le i\le d-k+1}\in \mathbb{R}^{d(d+1)/2}$ such that $A=\M({\bf
c})$. Let
$$\Gamma=\{\M({\bf h}): h_i^k\in \mathbb{Z}, 1\le k\le d, 1\le i\le d-k+1\}.$$
Then $\Gamma$ is a uniform subgroup of $\G_d$.

\medskip

\subsubsection{}

Let ${\bf a}=(a_i^k)_{1\le k\le d, 1\le i\le d-k+1}\in
\mathbb{R}^{d(d+1)/2}$ and ${\bf b}=(b_i^k)_{1\le k\le d, 1\le i\le
d-k+1}\in \mathbb{R}^{d(d+1)/2}$. If ${\bf c}=(c_i^k)_{1\le k\le d,
1\le i\le d-k+1}\in \mathbb{R}^{d(d+1)/2}$ such that $\M({\bf
c})=\M({\bf a})\M({\bf b})$, then
\begin{equation}\label{sec8-4-eq-1}
c_i^k=\sum \limits_{j=0}^k a_i^{k-j}b_{i+k-j}^j=a_i^k+(\sum
\limits_{j=1}^{k-1}a_i^{k-j}b_{i+k-j}^j)+b_i^k
\end{equation}
for $1\le k \le d$ and $1\le i \le d-k+1$, where we assume
$a_1^0=a_2^0=\ldots=a_d^0=1$ and $b_1^0=b_2^0=\ldots=b_d^0=1$.

\subsubsection{}

Now we endow a compatible metric $\rho$ on $\G_d$ and $\G_d/\Gamma$.

\begin{de}[Metric on $\G_d$]
Let $I$ be the $(d+1)\times (d+1)$ identity matrix. Define $\rho :
\G_d \times \G_d\lra \R$ such that
$$\rho(A, B) = \inf\{\sum_{i=1}^n\min\{\|A_{i-1}A_{i}^{-1}-I\|,
\|A_{i}A_{i-1}^{-1}-I\|\}: A_0,\ldots, A_n\in G; A_0=A,A_n=B\}.$$
\end{de}

\begin{lem}\label{lem-matrix}
For any $A,B\in \G_d$,
$$\rho(A, B)\le
\|AB^{-1}-I\|\le \|A-B\|\|B^{-1}\| \text{ and }$$
$$\frac{\|A-B\|}{\|B\|}\le \|AB^{-1}-I\|\le e^{\rho(A,B)+\rho(A,B)^2+\cdots+\rho(A,B)^d}-1.$$
\end{lem}

\begin{proof}
It is clear that it is sufficient to prove
\begin{equation}\label{matrix-key}
\|AB^{-1}-I\|\le e^{\rho(A,B)+\rho(A,B)^2+\cdots+\rho(A,B)^d}-1
\end{equation}
for $A,B\in \G_d$. The others are obvious. Let $A,B\in \G_d$. For
any $\epsilon>0$, there exist  $A_0,\ldots, A_n\in \G_d;
A_0=A,A_n=B$ such that
\begin{align*}
\sum_{i=1}^n\min\{\|A_{i-1}A_{i}^{-1}-I\|,
\|A_{i}A_{i-1}^{-1}-I\|\}\le \rho(A,B)+\epsilon.
\end{align*}
 Let
$t_i=\min\{\|A_{i-1}A_{i}^{-1}-I\|, \|A_{i}A_{i-1}^{-1}-I\|\}$ for
$i=1,2,\cdots,n$.
Note that for $C\in \G_d$, $C^{-1}=I+\sum_{i=1}^d (I-C)^i$. Hence
\begin{equation*}
\|C^{-1}-I\|\le \sum_{i=1}^d \|C-I\|^i.
\end{equation*}
Moreover, if we set $t=\min\{ \|C-I\|,\|C^{-1}-I\|\}$, then
\begin{equation}\label{eq-mat-es-eq-1}
\|C^{-1}-I\|\le t(1+t+t^2+\cdots+t^{d-1}).
\end{equation}
Then by \eqref{eq-mat-es-eq-1},
\begin{align*}
\|A_{i-1}A_{i}^{-1}-I\|&\le t_i(1+t_i+t_i^2+\cdots+t_i^{d-1})\\
&\le t_i(1+(\rho(A,B)+\epsilon)+\cdots+(\rho(A,B)+\epsilon)^{d-1})
\end{align*}
for $i=1,2,\cdots,n$. Thus
\begin{align*}
\sum \limits_{i=1}^n \|A_{i-1}A_{i}^{-1}-I\|&\le \sum
\limits_{i=1}^n
t_i(1+(\rho(A,B)+\epsilon)+\cdots+(\rho(A,B)+\epsilon)^{d-1}) \\
&\le (\rho(A,B)+\epsilon)+\cdots+(\rho(A,B)+\epsilon)^{d}.
\end{align*}
 Now
\begin{align*}
1+\|AB^{-1}-I\|&=1+\|A_0A_n^{-1}-I\|=1+\|A_0A_1^{-1}A_1A_n^{-1}-I\|\\
&=1+\|(A_0A_1^{-1}-I)(A_1A_n^{-1}-I)+(A_0A_1^{-1}-I)+(A_1A_n^{-1}-I)\|\\
&\le
1+\|(A_0A_1^{-1}-I)\|\|(A_1A_n^{-1}-I)\|+\|(A_0A_1^{-1}-I)\|+\|(A_1A_n^{-1}-I)\|\\
&=(1+\|A_0A_1^{-1}-I\|)(1+\|A_1A_n^{-1}-I\|).
\end{align*}
Continuing this process we get that
\begin{align*}
1+\|AB^{-1}-I\|&\le (1+\|A_0A_1^{-1}-I\|)(1+\|A_1A_2^{-1}-I\|)\cdots
(1+\|A_{n-1}A_n^{-1}-I\|)\\
&\le
e^{\|A_0A_1^{-1}-I\|+\|A_1A_2^{-1}-I\|+\cdots+\|A_{n-1}A_n^{-1}-I\|}\\
&=e^{(\rho(A,B)+\epsilon)+\cdots+(\rho(A,B)+\epsilon)^{d}}
\end{align*}
This implies $\|AB^{-1}-I\|\le
e^{(\rho(A,B)+\epsilon)+\cdots+(\rho(A,B)+\epsilon)^{d}}-1$. Let
$\epsilon\searrow 0$, we get \eqref{matrix-key}. This ends the proof
of the lemma.
\end{proof}

\begin{prop}[Metrics on $\G_d$ and $\G_d/\Gamma$] \label{matix-metric}
$\rho$ is a right-invariant metric on $\G_d$.  This descends to a
metric on $\G_d/\Gamma$ by setting
$$\rho(A\Gamma,B\Gamma):=\inf\{\rho(A\gamma, B\gamma'): \gamma, \gamma'\in \Gamma\}.$$
Since $\rho$ is right-invariant, we also have
$$\rho(A\Gamma, B\Gamma) = \inf_{\gamma\in \Gamma}\rho(A, B\gamma).$$
\end{prop}

\begin{proof} Firstly, it is clear that $\rho:\G_d\times \G_d\lra \R$ is
a right-invariant, non-negative function and for $A,B,C\in \G_d$,
$$\rho(A,B)=\rho(B,A) \text{ and }\rho(A,C)\le \rho(A,B)+\rho(B,C).$$
By Lemma \ref{lem-matrix} $\rho(A,B)=0$ if and only if $\|A-B\|=0$,
i.e., $A=B$. Thus $\rho$ is a right-invariant metric on $\G_d$.
Moreover by Lemma \ref{lem-matrix}, we know that the metric $\rho$
is equivalent to the metric induced by the norm $\|\cdot\|$ on
$\G_d$. Thus, $\rho$ is a compatible metric with topology of $\G_d$.

Next we are going to show that this descends to a metric on
$\G_d/\Gamma$ by setting
$$\rho(A\Gamma,B\Gamma):=\inf\{\rho(A\gamma, B\gamma'): \gamma, \gamma'\in \Gamma\}.$$
Since $\rho$ is a right-invariant metric on $\G_d$, it is sufficient
to show that if $\rho(A\Gamma,B\Gamma)=0$ then $A\Gamma=B\Gamma$.
Suppose $\rho(A\Gamma,B\Gamma)=0$. Since $\rho$ is right-invariant,
$\inf_{\gamma\in \Gamma}\rho(A\gamma,B)=0$. Moreover we can find
$\gamma_i\in \Gamma$ such that
$\|B\|(e^{\rho(A\gamma_i,B)+\cdots+\rho(A\gamma_i,B)^d}-1)<\frac{1}{2^i(1+\|A^{-1}\|)}$
for each $i\in \mathbb{N}$. By Lemma \ref{lem-matrix}, we have
$$\|A\gamma_i-B\|\le
\|B\|(e^{\rho(A\gamma_i,B)+\cdots+\rho(A\gamma_i,B)^d}-1)<\frac{1}{2^i(1+\|A^{-1}\|)}$$
for $i\in \mathbb{N}$. Thus for all $i\le j\in \mathbb{N}$,
\begin{align*}
\|\gamma_i-\gamma_j\|&=\|A^{-1}(A(\gamma_i-B)-A(\gamma_j-B))\|\\
&\le
\|A^{-1}\|(\|A\gamma_i-B\|+\|A\gamma_j-B\|)\\
&<\|A^{-1}\|(\frac{1}{2^i(1+\|A^{-1}\|)}+\frac{1}{2^i(1+\|A^{-1}\|)})<1.
\end{align*}
Since $\gamma_i,\gamma_j\in \Gamma$, this implies
$\gamma_i=\gamma_j$ for $i,j\in \mathbb{N}$. Thus
$$\|A\gamma_1-B\|=\|A\gamma_j-B\|<\frac{1}{2^j(1+\|A^{-1}\|)}$$
for any $j\in \mathbb{N}$. Hence $\|A\gamma_1-B\|=0$. So
$A\gamma_1=B$ and  $A\Gamma=B\Gamma$. This ends the proof of the
proposition.
\end{proof}

\section{Generalized polynomials}\label{section-GP}

In this section we introduce the notions and basic properties of
(special) generalized polynomials. It will be used in the following
sections.

\subsection{Definitions}

\subsubsection{}

For a real number $a\in\R$, let $||a||=\inf\{|a-n|:n\in\Z\}$ and
$$\lceil{a}\rceil=\min\{m\in\Z: |a-m|=||a||\}.$$

When studying $\F_{d,0}$ we find that the generalized polynomials
appear naturally. Here is the precise definition. Note that we use
$f(n)$ or $f$ to denote the generalized polynomials.

\subsubsection{Generalized polynomials}

\begin{de} Let $d\in\N$. We define the {\it generalized polynomials} of
degree $\le d$ (denoted by GP$_d$) by induction. For $d=1$, GP$_1$
is the collection of functions from $\Z$ to $\R$ containing $h_a$,
$a\in \R$ with $h_a(n)=an$ for each $n\in \Z$ which is closed under
taking $\lceil\ \rceil$, multiplying by a constant and the finite
sums.

Assume that GP$_i$ is defined for $i<d$. Then GP$_d$ is the
collection of functions from $\Z$ to $\R$ containing GP$_i$ with
$i<d$, functions of the forms $$a_0n^{p_0}{\lceil{f_1(n)}
\rceil}\ldots {\lceil{f_k(n)} \rceil}$$ (with $a_0\in\R,$ $p_0\ge
0$, $k\ge 0$, $f_l\in$ GP$_{p_l}$ and $\sum_{l=0}^kp_l=d$), which is
closed under taking $\lceil\ \rceil$, multiplying by a constant and
the finite sums. Let GP$=\cup_{i=1}^\infty$GP$_i$.
\end{de}

For example, $a_1\lceil{a_2\lceil{a_3n}}\rceil \rceil+b_1n\in$
GP$_1$, and
$a_1\lceil{a_2n^2}\rceil+b_1\lceil{b_2\lceil{b_3n}}\rceil
\rceil+c_1n^2+c_2n\in$ GP$_2$, where $a_i,b_i,c_i\in\R$. Note that
if $f\in$ GP then $f(0)=0$.

\subsubsection{Special generalized polynomials}

Since generalized polynomials are very complicated, we will specify
a subclass of them, called  the {\it special generalized
polynomials} which will be used in our proofs of the main results.
To do this, we need some notions.

For $a\in \mathbb{R}$, we define $L(a)=a$. For $a_1,a_2\in
\mathbb{R}$ we define $L(a_1,a_2)=a_1\lceil L(a_2)\rceil$.
Inductively, for $a_1,a_2,\cdots,a_{\ell}\in \mathbb{R}$ ($\ell\ge
2$) we define
\begin{equation}\label{eq-de-L}
L(a_1,a_2,\cdots,a_{\ell})=a_1\lceil
L(a_2,a_3,\cdots,a_\ell)\rceil.
\end{equation}

For example, $L(a_1,a_2,a_3)=a_1\lceil a_2\lceil a_3\rceil\rceil$.

\medskip

We give now the precise definition of special generalized
polynomials.

\begin{de} For $d\in\N$ we define {\it special generalized polynomials of degree} $\le d$,
denoted by SGP$_d$ as follows.  SGP$_d$ is the collection of
generalized polynomials of the forms
$L(n^{j_1}a_1,\cdots,n^{j_\ell}a_\ell)$, where $1\le \ell \le d,
a_1,\cdots,a_\ell\in \mathbb{R}, j_1,\cdots,j_\ell\in \mathbb{N}
\text{ with } \sum_{t=1}^\ell j_t\le d.$
\end{de}

Thus SGP$_1=\{an : a\in\R\},$ SGP$_2=\{an^2,
bn\lceil{cn}\rceil,en:a,b,c,e\in\R\}$ and SGP$_3$=SGP$_2\cup
\{an^3,an\lceil bn^2\rceil, an^2\lceil bn\rceil, an\lceil bn\lceil
cn\rceil\rceil: a,b,c\in\R\}$.

\subsubsection{$\F_{GP_d}$ and $\F_{SGP_d}$}

Let $\F_{GP_d}$  be the family generated by the sets of forms
$$\bigcap_{i=1}^k\{n\in\Z:P_i(n)\ (\text{mod}\ \Z)\in (-\ep_i,\ep_i)
\},$$ where $k\in\N$, $P_i\in GP_d$, and $\ep_i>0$, $1\le i\le k$.
Note that $P_i(n)\ (\text{mod}\ \Z)\in (-\ep_i,\ep_i)$ if and only
if $||P_i(n)||<\ep_i.$

Let $\F_{SGP_d}$ be the family generated by the sets of forms
$$\bigcap_{i=1}^k\{n\in\Z:P_i(n)\ (\text{mod}\ \Z)\in (-\ep_i,\ep_i)
\},$$ where $k\in\N$, $P_i\in SGP_d$, and $\ep_i>0$, $1\le i\le k$.
Note that from the definition both $\F_{GP_d}$ and $\F_{SGP_d}$ are
filters; and $\F_{SGP_d}\subset \F_{GP_d}.$

\subsection{Basic properties of generalized polynomials}

\subsubsection{}

The following lemmas lead the way to simplify the generalized
polynomials. Note that for $f\in$ GP we let $f^*=-\lceil{f}\rceil.$
\begin{lem}\label{simple2} Let $c\in\R$ and $f_1, \ldots, f_k\in GP$ with $k\in \N$. Then
$$c\lceil{f_1}\rceil\ldots \lceil{f_k}\rceil=c(-1)^k\prod_{i=1}^k(f_i-\lceil{f_i}\rceil)
-c(-1)^k\sum_{i_1,\ldots,i_k\in
\{1,*\}\atop{(i_1,\ldots,i_k)\not=(*,\ldots,*)}}f_{1}^{i_1}\ldots
f_{k}^{i_k}.$$ Particularly if $k=2$ we get that
$$c\lceil{f_1}\rceil\lceil{f_2}\rceil=cf_1\lceil{f_2}\rceil-cf_1f_2+
cf_2\lceil{f_1}\rceil+c(f_1-\lceil{f_1}\rceil)(f_2-\lceil{f_2}\rceil).$$
\end{lem}
\begin{proof}
Expanding $\prod_{i=1}^k(f_i-\lceil{f_i}\rceil)$ we get that
$$\prod_{i=1}^k(f_i-\lceil{f_i}\rceil)=\sum_{i_1,\ldots,i_k\in \{1,*\}}f_{1}^{i_1}\ldots f_{k}^{i_k}.$$
So we have $$c\lceil{f_1}\rceil\ldots
\lceil{f_k}\rceil=c(-1)^k\prod_{i=1}^k(f_i-\lceil{f_i}\rceil)-c(-1)^k\sum_{i_1,\ldots,i_k\in
\{1,*\}\atop{(i_1,\ldots,i_k)\not=(*,\ldots,*)}}f_{1}^{i_1}\ldots
f_{k}^{i_k}.$$
\end{proof}

Let $c=1$ in Lemma \ref{simple2} we have

\begin{lem}\label{simple3} Let $f_1,f_2, \ldots,f_k\in GP$. Then
$$f_1\lceil{f_2}\rceil\ldots\lceil{f_k}\rceil =(-1)^{k-1}\prod_{i=1}^k(f_i-\lceil{f_i}\rceil)
+(-1)^k\sum_{i_1,\ldots,i_k\in
\{1,*\}\atop{(i_1,\ldots,i_k)\not=(1,*,\ldots,*)}}f_{1}^{i_1}\ldots
f_{k}^{i_k}.$$

Particularly if $k=2$ we have

$$f_1\lceil{f_2}\rceil =\lceil{f_1}\rceil\lceil{f_2}\rceil +f_1f_2-
f_2\lceil{f_1}\rceil-(f_1-\lceil{f_1}\rceil)(f_2-\lceil{f_2}\rceil).$$
\end{lem}

Let $k=1$ in Lemma \ref{simple2} we have

\begin{lem}\label{simple1} Let $c\in\R$ and $f\in GP$. Then $c\lceil{f}\rceil=cf-c(f-\lceil{f}\rceil)$.
\end{lem}


\subsubsection{}

In the next subsection we will show that $\F_{GP_d}=\F_{SGP_d}.$ To
do this we use induction. To make the proof clearer, first we give
some results under the assumption
\begin{equation}\label{firsta}
\F_{GP_{d-1}}\subset\F_{SGP_{d-1}}.
\end{equation}
\begin{de}\label{de-ws} Let $r\in \mathbb{N}$ with $r\ge 2$. We define
$$\mathcal{SW}_r=\{ \prod \limits_{i=1}^\ell (w_i(n)-\lceil w_i(n) \rceil):
\ell\ge 2, r_i\ge 1, w_i(n)\in GP_{r_i} \text{ and } \sum
\limits_{i=1}^\ell r_i\le r\}$$ and
$$\mathcal{W}_r=\mathbb{R}-\text{Span}\{\mathcal{SW}_r\}, $$
that is,
$$\mathcal{W}_r=\{ \sum \limits_{j=1}^\ell a_jp_j(n): \ell\ge 1, a_j\in \mathbb{R}, p_j(n)\in
\mathcal{SW}_r \text{ for each }j=1,2,\cdots,\ell\}.$$
\end{de}

\begin{lem}\label{lem1-1}
Under the assumption \eqref{firsta}, one has for any $p(n)\in
\mathcal{W}_d$ and $\epsilon>0$,
$$\{ n\in \mathbb{Z}:p(n) \ (\text{mod}\ \Z)\in (-\epsilon,\epsilon)\}\in \F_{SGP_{d-1}}.$$
\end{lem}

\begin{proof}
Since $\F_{SGP_d}$ is a filter, it is sufficient to show that for
any $p(n)=aq(n)$ and $\frac{1}{2}>\delta>0$ with $q(n)\in
\mathcal{SW}_d$ and $a\in \mathbb{R}$,
$$\{ n\in \mathbb{Z}:p(n)\ (\text{mod}\ \Z) \in (-\delta,\delta)\}\in \F_{SGP_{d-1}}.$$
Note that as $q(n)\in \mathcal{SW}_d$, there exist $\ell\ge 2,
r_i\ge 1, w_i(n)\in GP_{r_i} \text{ and } \sum \limits_{i=1}^\ell
r_i\le d$ such that $q(n)=\prod \limits_{i=1}^\ell (w_i(n)-\lceil
w_i(n) \rceil)$. Since $\ell\ge 2$, one has $r_1\le d-1$ and so
$w_1(n)\in GP_{d-1}$. By the assumption \eqref{firsta}, $\{n\in
\mathbb{Z}: w_1(n)\ (\text{mod}\ \Z)\in
(-\frac{\delta}{1+|a|},\frac{\delta}{1+|a|}) \}\in \F_{SGP_{d-1}}$.
By the inequality $|q(n)|\le |a||w_1(n)-\lceil w_1(n) \rceil|$ for
$n\in \mathbb{Z}$, we get that
\begin{align*}
\{ n\in \mathbb{Z}:p(n)\ (\text{mod}\ \Z) \in (-\delta,\delta)\}
&\supset
\{n\in \mathbb{Z}: |w_1(n)-\lceil w_1(n) \rceil|\in (-\frac{\delta}{1+|a|},\frac{\delta}{1+|a|})\}\\
&=\{n\in \mathbb{Z}: w_1(n)\ (\text{mod}\ \Z)\in
(-\frac{\delta}{1+|a|},\frac{\delta}{1+|a|})\}.
\end{align*}
Thus $\{ n\in \mathbb{Z}:p(n)\ (\text{mod}\ \Z) \in
(-\delta,\delta)\} \in \F_{SGP_{d-1}}$ since $\{n\in \mathbb{Z}:
w_1(n)\ (\text{mod}\ \Z)\in
(-\frac{\delta}{1+|a|},\frac{\delta}{1+|a|})\}\in \F_{SGP_{d-1}}$.
\end{proof}

\begin{de}
Let $r\in \mathbb{N}$ with $r\ge 2$. For $q_1(n),q_2(n)\in GP_{r}$
we define
$$q_1(n)\simeq_r q_2(n)$$
if there exist $h_1(n)\in GP_{r-1}$ and $h_2(n)\in \mathcal{W}_r$
such that
$$q_2(n)=q_1(n)+h_1(n)+h_2(n) \ (\text{mod} \ \Z)$$ for all $n\in \mathbb{Z}$.
\end{de}

\begin{lem} \label{lem8-13} Let $p(n)\in GP_r$ and $q(n)\in GP_t$, $r,t\in \mathbb{N}$. Then
\begin{enumerate}
\item $p(n)\lceil q(n)\rceil \simeq_{r+t} (p(n)-\lceil p(n) \rceil) q(n)$.

\item if $q_1(n),q_2(n),\cdots,q_k(n)\in GP_t$ such that $q(n)= \sum_{i=1}^k q_i(n)$, then
$$p(n)\lceil q(n)\rceil \simeq_{r+t} \sum_{i=1}^k p(n)\lceil q_i(n)\rceil.$$
\end{enumerate}
\end{lem}
\begin{proof} (1) follows from Lemma \ref{simple3} and (2) follows from (1).
\end{proof}

\begin{de} For $r\in \mathbb{N}$, we define
$${GP_r'}=\{p\in GP_r: \{ n\in \mathbb{Z}: p(n)\ (\text{mod}\ \Z)\in
(-\epsilon,\epsilon)\}\in \F_{SGP_r}\text{ for any }\epsilon>0\}.$$
\end{de}

\begin{prop}\label{useful} It is clear that
\begin{enumerate}
\item For $p(n)\in GP_r$, $p(n)\in GP'_r$ if and only if $-p(n)\in
GP'_r$.

\item If $p_1(n),p_2(n),\cdots,p_k(n)\in GP'_r$ then
$$p(n)=p_1(n)+p_2(n)+\cdots+p_k(n)\in GP'_r.$$

\item $\F_{GP_d}\subset\F_{SGP_d}$ if and only if $GP'_d=GP_d$.
\end{enumerate}
\end{prop}
\begin{proof} (1) can be verified directly. (2) follows from the fact
that for each $\ep>0$, $\{n\in\Z: p(n)\ (\text{mod}\ \Z)\in
(-\epsilon,\epsilon)\}\supset \cap_{i=1}^k \{n\in\Z: p_i(n)\
(\text{mod}\ \Z)\in (-\ep/k,\ep/k)\}$. (3) follows from the
definition of $GP'_d$.
\end{proof}

\begin{lem}\label{lem-8-equi} Let $p(n),q(n)\in GP_d$ with $p(n)\simeq_d q(n)$. Under the assumption \eqref{firsta},
$p(n)\in GP'_d$ if and only if $q(n)\in GP'_d$.
%
\end{lem}
\begin{proof} It follows from Lemma \ref{lem1-1} and the fact that $\F_{SGP_d}$ is a filter.
\end{proof}






\subsection{$\F_{GP_d}=\F_{SGP_d}$}


\begin{thm}\label{gpsgp}$\F_{GP_d}=\F_{SGP_d}.$
\end{thm}
\begin{proof} It is easy to see that $\F_{SGP_d}\subset \F_{GP_d}.$
So it remains to show $\F_{GP_d}\subset \F_{SGP_d}.$ That is, if
$A\in\F_{GP_d}$ then there is $A'\in\F_{SGP_d}$ with $A\supset A'$.
We will use induction to show the proposition.

Assume first $d=1$. In this case we let $GP_1(0)=\{g_a:a\in\R\}$,
where $g_a(n)=an$ for each $n\in\Z$. Inductively if
$GP_1(0),\ldots,GP_1(k)$ have been defined then $f\in GP_1(k+1)$ if
and only if $f\in GP_1\setminus  (\bigcup_{j=0}^kGP_1(j))$ and there
are $k+1$ $\lceil {\ \ } \rceil$ in $f$. It is clear that
$GP_1=\cup_{k=1}^\infty GP_1(k)$. If $f\in GP_1(0)$ then it is clear
that $f\in GP'_1$. Assume that $GP_1(0), \ldots, GP_1(k)\subset
GP'_1$ for some $k\in \mathbb{Z}_+$.

Let $f\in GP_1(k+1)$. We are going to show that $f\in GP_1'$. If
$f=f_1+f_2$ with $f_1,f_2\in \bigcup_{i=0}^k GP_1(i)$, then by the
above assumption and Proposition \ref{useful} we conclude that $f\in
GP_1'$. The remaining case is $f=c\lceil f_1 \rceil +f_2$ with $c\in
\mathbb{R}\setminus \{0\}$, $f_1\in GP_1(k)$,  and $f_2\in GP_1(0)$.
By Proposition \ref{useful} and the fact $GP_1(0)\subseteq GP_1'$,
$f\in GP_1'$ if and only if $c \lceil{f_1}\rceil \in GP_1'$. So it
remains to show $c\lceil f_1 \rceil\in GP'_1$. By Lemma
\ref{simple1} we have $c\lceil f_1\rceil=cf_1-c(f_1-\lceil f_1
\rceil)$. It is clear that $cf_1\in GP_1(k)\subset GP_1'$ since
$f_1\in GP_1(k)\subset GP_1'$. For any $\ep>0$ since
\begin{equation*}
\{n\in\Z:||-c(f_1(n)-\lceil{f_1(n)}\rceil)||<\ep\}\supset
\Big\{n\in\Z: ||f_1(n)||<\frac{\ep}{1+|c|}\Big\},
\end{equation*}
it implies that $-c(f_1-\lceil{f_1}\rceil)\in GP_1'$. By
Proposition \ref{useful} again we conclude that
$c\lceil{f_1}\rceil\in GP'_1$. Hence $f\in GP_1'$.
Thus $GP_1\subseteq GP_1'$ and we
are done for the case $d=1$ by Proposition \ref{useful} (3).

\medskip
Assume that we have proved $\F_{GP_{d-1}}\subset \F_{SGP_{d-1}}\
d\ge 2$, i.e. the assumption (\ref{firsta}) holds.
We define $GP_d(k)$ with $k=0,1,2,\ldots.$ First $f\in GP_d(0)$ if
and only if there is no $\lceil {\  \ } \rceil$ in $f$, i.e. $f$ is
the usual polynomial of degree $\le d$.  Inductively if
$GP_d(0),\ldots,GP_d(k)$ have been defined then $f\in GP_{k+1}$ if
and only if $f\in GP_d\setminus  (\bigcup_{j=0}^kGP_d(j))$ and there are $k+1$
$\lceil {\ \ } \rceil$ in $f$. It is clear that
$GP_d=\cup_{k=0}^\infty GP_d(k)$. We now show $GP_d(k)\subseteq GP_d'$ by induction on $k$.

Let $f$ be a usual polynomial of degree $\le d$. Then
$f(n)=a_0n^d+f_1(n)\simeq_d a_0n^d$ with $f_1\in GP_{d-1}$. By Lemma \ref{lem-8-equi},
$f\in GP_d'$ since $a_0n^d\in \text{SGP}_d\subset GP_d'$. This shows $GP_d(0)\subset GP_d'$.
Now assume that for some
$k\in \mathbb{Z}_+$ we have proved
\begin{equation}\label{aye2}
\bigcup_{i=0}^k GP_d(i)\subseteq GP_d'.
\end{equation}
Let $f\in GP_d(k+1)$. We are going to show that $f\in GP_d'$. If
$f=f_1+f_2$ with $f_1,f_2\in \bigcup_{i=0}^k GP_d(i)$, then by the
assumption (\ref{aye2}) and Proposition \ref{useful} (2) we conclude
that $f\in GP_d'$. The remaining case is that $f$ can be expressed
as the sum of a function in $GP_d(0)$ and a function $g\in
GP_d(k+1)$ having the form of
\begin{enumerate}
\item $g=c\lceil{f_1}\rceil\ldots \lceil{f_l}\rceil$ with $c\neq 0$, $l\ge 1$ or

\item $g=g_1(n)\lceil{g_2(n)}\rceil\ldots \lceil{g_l(n)}\rceil$ for any $n\in\Z$ with
$g_1(n)\in SGP_r$ and $r<d$.
\end{enumerate}
Since $GP_d(0)\subset GP_d'$, $f\in GP_d'$ if and only if $g\in GP_d'$ by Proposition \ref{useful}.
It remains to show that $g\in GP_d'$. There are two cases.

\medskip
\noindent Case (1): $g=c\lceil{f_1}\rceil\ldots \lceil{f_l}\rceil$ with $c\neq 0$, $l\ge 1$.

If $l=1$, then $g=c\lceil{f_1}\rceil$ with $f_1\in GP_d(k)$.
By Lemma \ref{simple1} we have
$c\lceil f_1\rceil=cf_1-c(f_1-\lceil f_1 \rceil)$. It is
clear that $cf_1\in GP_d(k)\subset GP_d'$ since $f_1\in GP_d(k)\subset GP_d'$. For any $\ep>0$ since
\begin{equation*}
\{n\in\Z:||-c(f_1(n)-\lceil{f_1(n)}\rceil)||<\ep\}\supset
\Big\{n\in\Z: ||f_1(n)||<\frac{\ep}{1+|c|}\Big\},
\end{equation*}
it implies that $-c(f_1-\lceil{f_1}\rceil)\in GP_d'$. By
Proposition \ref{useful} again we conclude that
$g=c\lceil{f_1}\rceil\in GP'_d$.

If $l\ge 2$, using Lemmas \ref{simple2} and \ref{lem1-1} we get that
$$c\lceil{f_1}\rceil\ldots \lceil{f_l}\rceil\simeq_d -c(-1)^l\sum_{i_1,\ldots,i_l\in
\{1,*\}\atop{(i_1,\ldots,i_l)\not=(*,\ldots,*)}}f_{1}^{i_1}\ldots
f_{l}^{i_l}.$$
Since each term of the right side is in $GP_d(k)$,  $g\in GP_d'$  by Lemma \ref{lem-8-equi},
the assumption
(\ref{aye2}) and Proposition \ref{useful} (2).

\medskip

\noindent Case (2): $g=g_1(n)\lceil{g_2(n)}\rceil\ldots \lceil{g_l(n)}\rceil$ for any $n\in\Z$ with
$g_1\in SGP_r$ and $1\le r<d$.

In this case using Lemmas \ref{simple3} and
\ref{lem1-1} we get that
$$g_1\lceil{g_2}\rceil\ldots\lceil{g_l}\rceil\simeq_d (-1)^l\sum_{i_1,\ldots,i_l\in
\{1,*\}\atop{(i_1,\ldots,i_l)\not=(1,*,\ldots,*),(*,*,\ldots,*)}}g_{1}^{i_1}\ldots
g_{l}^{i_l}.$$ Assume $i_1,\ldots,i_l\in \{1,*\}$ with
$(i_1,\ldots,i_l)\not=(1,*,\ldots,*),(*,*,\ldots,*)$. If there are at least two $1$ appearing in $(i_1,i_2,\cdots,i_l)$, then $(-1)^\ell
g_{1}^{i_1}\ldots g_{l}^{i_\ell}\in \bigcup_{i=0}^k GP_d(i)$. Hence
$$(-1)^\ell g_{1}^{i_1}\ldots g_{l}^{i_\ell}\in GP_d'$$ by
the assumption (\ref{aye2}). The remaining situation is that $i_1=*$
and there is exact one 1 appearing in $(i_2, \ldots, i_l)$. In this
case,
$(-1)^\ell g_{1}^{i_1}\ldots g_{l}^{i_\ell}\in GP_d(k+1)$ is the
finite sum of the forms $a_1n^{t_1}\lceil{h_1(n)}\rceil \ldots
\lceil{h_{l_1'}(n)}\rceil$ with $t_1\ge 1$ and $h_1(n)=g_1(n)$; or
the forms $c\lceil{h_l}\rceil \ldots \lceil{h_{l_1}}\rceil$ or terms
in $GP_d'$.

If the term has the form $a_1n^{t_1}\lceil{h_1(n)}\rceil \ldots
\lceil{h_{l_1'}(n)}\rceil$ with $t_1\ge 1$ and $h_1(n)=g_1(n)$, we
let
$g_1^{(1)}(n)=a_1n^{t_1}\lceil{h_1(n)}\rceil=a_1n^{t_1}\lceil{g_1(n)}\rceil\in
SGP_{r_1}$. It is clear $d\ge r_1>r$. If $r_1=d$, then
$a_1n^{t_1}\lceil{h_1(n)}\rceil \ldots
\lceil{h_{l_1'}(n)}\rceil=g_1^1(n)\in GP_d'$ since $SGP_d\subset GP_d'$. If $r_1<d$,
then we write
$$a_1n^{t_1}\lceil{h_1(n)}\rceil\ldots
\lceil{h_{l_1'}(n)}\rceil=g_1^{(1)}(n)\lceil{g_2^{(1)}(n)}\rceil\ldots\lceil{g_{l_1}^{(1)}(n)}\rceil.$$
By using  Case (1) we conclude that

$g\simeq_d$ finite sum of the forms
$g_1^{(1)}(n)\lceil{g_2^{(1)}(n)}\rceil\ldots\lceil{g_{l_1}^{(1)}(n)}\rceil$
and terms in $GP_d'$.

\medskip
Repeating the above process finitely many time (at most $d$-times)
we get that $g\simeq_d$ finite sum of terms in $GP_d'$. Thus $g\in GP_d'$ by Lemma \ref{lem-8-equi}
and Proposition \ref{useful} (2). The proof is now finished.
\end{proof}

\medskip




\section{Proof of Theorem B(1)} \label{section-proof1}

In this section, we will prove Theorem B(1), i.e. we will show that
if $A\in \F_{d,0}$ then there are $k\in\N$, $P_i\in GP_d$ ($1\le
i\le k$) and $\ep_i>0$ such that
$$A\supset \bigcap_{i=1}^k\{n\in\Z:P_i(n)\ (\text{mod}\ \Z)\in (-\ep_i,\ep_i) \}.$$

We remark that by Section \ref{reduction}, it is sufficient to
consider the case when the group $G$ is a connected,
simply-connected $d$-step nilpotent Lie group.

\subsection{Notations}

Let $X=G/\Gamma$ with $G$ a connected, simply-connected $d$-step
nilpotent Lie group, $\Gamma$ a uniform subgroup. Let $T:X\lra X$ be
the nilrotation induced by $a\in G$.

\medskip
We assume $\mathfrak{g}$ is the Lie algebra of $G$ over $\R$, and
$\exp: \g \lra G$ is the exponential map. Consider
$$\g=\g^{(1)}\supsetneq\g^{(2)}\supsetneq \ldots\supsetneq \g^{(d)} \supsetneq\g^{(d+1)}=\{0\}.$$
We note that $$[\g^{(i)},\g^{(j)}]\subset \g^{(i+j)}.$$ There is a
Mal'cev basis $\mathcal {X}= \{X_1, \ldots ,X_m\}$ for $\g$ with

\begin{enumerate}
\item  For each $j = 0,\ldots, m-1$ the subspace $\eta_j:=\text{Span}(X_{j+1}, \ldots
,X_m)$ is a Lie algebra ideal in $\g$, and hence $H_j := \exp\
\eta_j$ is a normal Lie subgroup of $G$.

\item For every $0<i<d$ we have $G_i = H_{l_{i-1}+1}$.

\item Each $g\in G$ can be written uniquely as
$\exp(t_1X_1) \exp(t_2X_2)\ldots \exp(t_mX_m)$, for $t_i\in \R$.

\item $\Gamma$ consists precisely of those elements which, when written in
the above form, have all $t_i\in \Z$,
\end{enumerate}
where $G=G_1$, $G_{i+1}=[G_i,G]$ with $G_{d+1}=\{e\}$. Thus, there
are $0=l_0<l_1<\ldots<l_{d-1}<l_d=m$ such that
$\text{Span}-\{X_{l_i+1},\ldots,X_{m}\}=\g^{(i+1)}$ for
$i=0,1,\ldots,d-1$.

\begin{de}
Define $o(0)=0$ and $o(i)=j$ if $l_{j-1}+1\le i\le l_j, 2\le j\le
d-1$.
\end{de}


\subsection{Some lemmas}

We need several lemmas. Note that if
$$\exp(t_1X_1)\ldots\exp(t_mX_m)=\exp(u_1X_1+\ldots+u_mX_m)$$ it is
known that \cite{CG, M} each $t_i$ is a polynomial of $u_1,\ldots,
u_m$ and each $u_i$ is a polynomial of $t_1,\ldots, t_m$. For our
purpose we need to know the precise degree of the polynomials.

\begin{lem}\label{d-step} Let $\{X_1,\ldots,X_m\}$ be a Mal'cev bases for
$G/\Gamma$. Assume that
$$\exp(t_1X_1)\ldots\exp(t_mX_m)=\exp(u_1X_1+\ldots+u_mX_m).$$
Then we have

\begin{enumerate}

\item
$u_i=t_i$ for $1\le i\le l_1$ and if $l_{j-1}+1\le i\le l_j,\ 2\le
j\le d$ then
$$u_i=t_i+\sum_{k_1o(1)+\ldots+ k_{m}o(m)\le o(i),\atop
{k_0\le m-2,\ 0\le k_1,\ldots,k_{m}\le m}} c_{k_1,\ldots,
k_m,i}t_{1}^{k_1}\ldots t_{m}^{k_m},$$ where $k_0$ is the number of
$0's$ appearing in $\{k_1,\ldots, k_m\}$.

\item $t_i=u_i$ for $1\le i\le l_1$ and if $l_{j-1}+1\le i\le l_j,\ 2\le j\le d$ then
$$t_i=u_i+\sum_{k_1o(1)+\ldots+ k_{m}o(m)\le o(i),\atop
{k_0\le m-2,\ 0\le k_1,\ldots,k_{m}\le m}}d_{k_1,\ldots,
k_m,i}u_1^{k_1}\ldots u_m^{k_m},$$ where $k_0$ is the number of
$0's$ appearing in $\{k_1,\ldots, k_m\}$.

\end{enumerate}
\end{lem}
\begin{proof}
(1). It is easy to see that if $m=1$ then $d=1$ and (1) holds. So we
may assume that $m\ge 2$. For $s\in
\{0,1,\ldots,m\}^{\{1,\ldots,m\}}$, let $\{i_1<\ldots<i_n\}$ be the
collection of $p's$ with $s(p)\not=0$. Let
$$X_{s}=[X_{s(i_1)}, [X_{s(i_2)}, \ldots, [X_{s(i_{n-1})},X_{s(i_n)}]]].$$ For
each $0\le p\le m$ let $k_p(s)$ be the number of $p's$ appearing in
$s$ (as usual, the cardinality of the empty set is defined as $0$).
Using the CBH formula $m-1$ times and the condition
$\g^{(d+1)}=\{0\}$ it is easy to see that $(t_1X_1)*\ldots*(t_mX_m)$
is the sum of $\sum_{i=1}^mt_iX_i$ and the terms $$constant \times
t_{q_1}\ldots t_{q_n}[X_{q_1}, [X_{q_2}, \ldots,
[X_{q_{n-1}},X_{q_n}]]],\ m\ge n\ge 2,$$ i.e.
$\exp(t_1X_1)\ldots\exp(t_mX_m)$ can be written as
$$\exp(\sum_{j=1}^mt_jX_j+\sum_{s\in \{0,1,\ldots,m\}^{\{1,\ldots,m\}}
\atop{k_0(s)\le m-2}}c'_{s}t_1^{k_1(s)}\ldots t_m^{k_m(s)}X_s).$$
Note that $X_{s}\subset \g^{(\sum_{j=1}^m k_j(s)o(j))}.$ Let
$X_s=\sum_{j=1}^m c'_{s,j}X_j$. Thus, $c'_{s,1}, \ldots, c_{s,i}'=0$
if $\sum_{j=1}^m k_j(s)o(j)>o(i)$. Thus, $u_i=t_i$ for $1\le i\le
l_1$ and if $l_{j-1}+1\le i\le l_j,\ 2\le j\le d$ then the
coefficient of $X_i$ is
$$u_i=t_i+\sum_{k_1o(1)+\ldots+ k_{m}o(m)\le o(i),\atop
{k_0\le m-2,\ 0\le k_1,\ldots,k_{m}\le m}} c_{k_1,\ldots,
k_m,i}t_{1}^{k_1}\ldots t_{m}^{k_m}.$$

Note that when ${k_1o(1)+\ldots+ k_{m}o(m)}\le o(i)$ and $k_0\le
m-2$, we have that $k_i=k_{i+1}=\ldots=k_m=0$ and some other
restrictions. For example, when $l_1+1\le i\le l_2$,
$t_{1}^{k_1}\ldots t_{m}^{k_m}=t_{i_1}t_{i_2}$ with $1\le i_1,i_2\le
l_1$; and when $l_2+1\le i\le l_3$, $t_{1}^{k_1}\ldots
t_{m}^{k_m}=t_{i_1}t_{i_2}t_{i_3}$ with $1\le i_1,i_2,i_3\le l_1$ or
$t_{i_1}t_{i_2}$ with $1\le i_1\le l_1$ and $l_1+1\le i_2\le l_2$.

\medskip

(2) It is easy to see that $t_i=u_i$ for $1\le i\le l_1$.
If $d=1$ (2) holds, and thus we assume that $d\ge 2$. We show (2) by
induction. We assume that
\begin{equation}\label{tp-eq}
t_p=u_p+\sum_{k_1'o(1)+\ldots+ k_{m}'o(m)\le o(p),\atop
{k_0'\le m-2,\ 0\le k_1',\ldots,k_{m}'\le m}}d_{k_1',\ldots,
k_m',p}u_1^{k_1'}\ldots u_m^{k_m'},
\end{equation}
 for all $p$ with $l_1+1\le
p\le i.$

Since
$$u_{i+1}=t_{i+1}+\sum_{k_1o(1)+\ldots+ k_{m}o(m)\le o(i+1),\atop
{k_0\le m-2,\ 0\le k_1,\ldots,k_{m}\le m}} c_{k_1,\ldots,
k_m,i+1}t_{1}^{k_1}\ldots t_{m}^{k_m},$$ we have that

$$t_{i+1}=u_{i+1}-\sum_{k_1o(1)+\ldots+ k_{m}o(m)\le o(i+1),\atop
{k_0\le m-2,\ 0\le k_1,\ldots,k_{m}\le m}} c_{k_1,\ldots,
k_m,i+1}t_{1}^{k_1}\ldots t_{m}^{k_m}.$$

Since $o(i+1)\le o(i)+1$ and $k_0\le m-2$ we have that if
$k_1o(1)+\ldots+ k_{m}o(m)\le o(i+1)$ then $k_po(p)\le o(i)$ for
each $1\le p\le m$, which implies that $k_{i+1},\ldots,k_m=0$. By
the induction each $t_p$ ($1\le p\le i$) is a polynomial of
$u_1,\ldots,u_m$ of degree at most $\sum_{j=1}^mk_j'\le o(p)$(see
Equation \eqref{tp-eq}) thus
$$\sum \limits_{k_1o(1)+\ldots+ k_{m}o(m)\le o(i+1),\atop {k_0\le m-2,\ 0\le
k_1,\ldots,k_{m}\le m}} c_{k_1,\ldots, k_m,i+1}t_{1}^{k_1}\ldots
t_{m}^{k_m}$$ is a polynomial of $u_1,\ldots,u_m$ of degree at most
$\sum_{p=1}^mk_p\le o(i+1).$  Rearranging the coefficients we
get (2). Note that $k_0\le m-2$ is satisfied automatically.
\end{proof}

\begin{lem}\label{product}
Assume that
$$x=\exp(x_1X_1+\cdots+x_mX_m)\ \text{ and}\
y=\exp(y_1X_1) \ldots \exp(y_mX_m).$$ Then
$$xy^{-1}=\exp(\sum_{i=1}^{l_1}(x_i-y_i)X_i+\sum_{i=l_1+1}^m((x_i-y_i)+P_{i,1}(\{y_p\})
+P_{i,2}(\{x_p\},\{y_p\}))X_i),$$ where $P_{i,1}(\{y_p\}),
P_{i,2}(\{x_p\},\{y_p\})$ are polynomials of degree at most $o(i).$
\end{lem}
\begin{proof}
By Lemma \ref{d-step} we have
\begin{equation*}
\begin{aligned}xy^{-1}=\exp(X)\exp(Y)
\end{aligned}
\end{equation*}
where $X=\sum_{i=1}^{m}x_iX_i$ and
$$Y=-\sum_{i=1}^{m}y_iX_i-\sum_{i=l_1+1}^m(\sum_{k_1'o(1)+\ldots+
k_{m}'o(m)\le o(i),\atop {k_0'\le m-2,0\le k_1',\ldots,k_{m}'\le m}}
c_{k_1',\ldots, k_m',i}y_{1}^{k_1'}\ldots y_{m}^{k_m'})X_i.$$ Using
the CBH formula we get that
\begin{equation*}
\begin{aligned}
xy^{-1}&=\exp(X*Y)=\exp(X+Y+\frac{1}{2}[X,Y]+\frac{1}{12}[X,[X,Y]]+\cdots)\\
&= \exp(\sum_{i=1}^m(x_i-y_i)X_i+\sum_{i=l_1+1}^m (P_{i,1}(\{y_p\})+P_{i,2}(\{x_p\},\{y_p\}))X_i)\\
&=\exp(\sum_{i=1}^{l_1}(x_i-y_i)X_i+\sum_{i=l_1+1}^m((x_i-y_i)+P_{i,1}(\{y_p\})+
P_{i,2}(\{x_p\},\{y_p\}))X_i)
\end{aligned}
\end{equation*}
where
$$P_{i,1}(\{y_p\})=-\sum_{k_1'o(1)+\ldots+
k_{m}'o(m)\le o(i),\atop {k_0'\le m-2,0\le k_1',\ldots,k_{m}'\le m}}
c_{k_1',\ldots, k_m',i}y_{1}^{k_1'}\ldots y_{m}^{k_m'},$$ and
$$P_{i,2}(\{x_p\},\{y_p\})=\sum_{\sum_{j=1}^m (k_j+k_j')o(j)\le o(i),\atop{ {0\le k_1,\ldots,k_{m},
k_1',\ldots,k_{m}'\le m}\atop {k_0\le m-1,k_0'\le
m-1}}}e_{k_1,\ldots,k_m}^{k_1',\ldots,k_m'}x_{1}^{k_1}\ldots
x_{m}^{k_m}y_{1}^{k_1'}\ldots y_{m}^{k_m'}.$$

Note that the reason $P_{i,2}$ has the above form follows from the
fact that $[\g^{(i)},\g^{(j)}]\subset \g^{(i+j)}$,
$\g^{(d+1)}=\{0\}$ and a discussion similar to the one used in Lemma
\ref{d-step}.
\end{proof}

\medskip

\subsection{Proof of Theorem B(1)}
Let $X=G/\Gamma$ with $G$ a connected, simply-connected $d$-step
nilpotent Lie group, $\Gamma$ a uniform subgroup. Let $T:X\lra X$ be
the nilrotation induced by $a\in G$. Assume that $A\supset
N(x\Gamma,U)$ with $x\in X$, $x\Gamma\in U$ and $U\subset G/\Gamma$
open.  By Proposition \ref{replace} we may assume that $U$ is an
open neighborhood of $e\Gamma$ in $G/\Gamma$, where $e$ is the unit
element of $G$, i.e. $A\supset N(e\Gamma,U)$.

Assume that $a=\exp(a_1X_1+\cdots+a_mX_m)$, where $a_1,\cdots a_m\in
\mathbb{R}$. Then $$a^n=\exp(na_1X_1+\ldots+na_mX_m)$$ for any $n\in
\mathbb{Z}$. For $h=\exp(h_1X_1)\ldots \exp(h_mX_m)$, where
$h_1,\cdots,h_m\in \mathbb{R}$, write \begin{equation*}
\begin{aligned}
a^nh^{-1}=\exp(p_1X_1+\ldots+p_mX_m)=\exp(w_1X_1)\ldots\exp(w_mX_m).
\end{aligned}
\end{equation*}
Then by Lemma \ref{product} we have $p_i=na_i-h_i$ for $1\le i\le l_1$
and if $l_{j-1}+1\le i\le l_j,\ 2\le j\le d$ then
\begin{equation}\label{eq-5.2}
\begin{aligned}
p_i &=na_i-h_i+P_{i,1}(\{h_p\})+\sum_{\sum_{p=1}^m (k_p+k_p')o(p)\le
o(i),\atop {{0\le k_1,\ldots,k_{m}, k_1',\ldots,k_{m}'\le
m}\atop{k_0\le m-1,k_0'\le
m-1}}}e_{k_p,k_p'}n^{k_1+\ldots+k_m}h_{1}^{k_1'}\ldots h_{m}^{k_m'},
\end{aligned}
\end{equation}
where $P_{i,1}(\{h_p\})=-\sum\limits_{k_1'o(1)+\ldots+ k_{m}'o(m)\le
o(i),\atop {k_0'\le m-2,0\le k_1',\ldots,k_{m}'\le m}}
c_{k_1',\ldots, k_m',i}h_{1}^{k_1'}\ldots h_{m}^{k_m'}$ and
$e_{k_p,k_p'}=e_{k_1,\ldots,k_m}^{k_1',\ldots,k_m'}a_{1}^{k_1}\ldots
a_{m}^{k_m}$.

Changing the exponential coordinates to Mal'sev coordinates (Lemma
\ref{d-step}), we get that $w_i=na_i-h_i$ for $1\le i\le l_1$ and if
$l_{j-1}+1\le i\le l_j,\ 2\le j\le d$ then
\begin{equation*}
w_i=p_i+\sum_{k_1o(1)+\ldots+ k_{m}o(m)\le o(i),\atop {k_0\le m-2,\
0\le k_1,\ldots,k_{m}\le m}}d_{k_1,\ldots, k_m,i}p_1^{k_1}\ldots
p_m^{k_m},
\end{equation*}
in this case using \eqref{eq-5.2} it is not hard to see that $w_i$
is the sum of $-h_i$ and $Q_i=Q_i(n,h_1,\ldots,h_{i-1})$ such that
$Q_i$ is the sum of terms
$$c(k, k_1,\ldots,k_{i-1})n^kh_1^{k_1}\ldots h_{i-1}^{k_{i-1}}$$
with $k+k_1o(1)+\ldots+k_{i-1}o(i-1)\le o(i)$ (see the argument of Lemma
\ref{d-step}(2)). Note that if $k=0$ then $k_0\le m-2$, and if
$k_1=\ldots=k_m=0$ then $k\ge 1$.

For a given $n\in \mathbb{Z}$, let $h_i(n)=\lceil{na_i}\rceil$ if
$1\le i\le l_1$, and let
$h_i(n)=\lceil{Q_i(n,{h_1(n)},\ldots,h_{i-1}(n))}\rceil$ if
$l_{j-1}+1\le i\le l_j,\ 2\le j\le d$. Again a similar argument as
in the proof of Lemma \ref{d-step}(2) shows that $h_i(n)$ is well
defined and is a generalized polynomial of degree of most $o(i)\le
d$. For example, if $l_1+1\le i\le l_2$ then
$$p_i=na_i-h_i+\sum_{1\le i_1<i_2\le
l_1}c(i_1,i_2,i)h_{i_1}h_{i_2}+\sum_{1\le j_1\le l_1}c(j_1,i)nh_{j_1}.$$
So
\begin{equation*}
\begin{aligned}w_i=na_i-h_i+\sum_{1\le i_1<i_2\le
l_1}c(i_1,i_2,i)h_{i_1}h_{i_2}&+\sum_{1\le j_1\le
l_1}c(j_1,i)nh_{j_1}+\sum_{1\le i_1<i_2\le
l_1}d(i_1,i_2,i)p_{i_1}p_{i_2}\\
=na_i-h_i+\sum_{1\le i_1<i_2\le
l_1}c(i_1,i_2,i)h_{i_1}h_{i_2}&+\sum_{1\le j_1\le
l_1}c(j_1,i)nh_{j_1}\\
&+\sum_{1\le i_1<i_2\le
l_1}d(i_1,i_2,i)(na_{i_1}-h_{i_1})(na_{i_2}-h_{i_2}).
\end{aligned}
\end{equation*}
Thus if we let $h_i(n)=\lceil{na_i}\rceil$, $1\le i\le l_1$ then if
$l_1+1\le i\le l_2$
\begin{equation*}
\begin{aligned}
h_i(n)=\lceil na_i&+\sum_{1\le i_1<i_2\le
l_1}c(i_1,i_2,i)\lceil na_{i_1}\rceil \lceil na_{i_2}\rceil+\sum_{1\le j_1\le l_1}c(j_1,i)n\lceil na_{j_1}\rceil\\
&+\sum_{1\le i_1<i_2\le l_1}d(i_1,i_2,i)(na_{i_1}-\lceil
na_{i_1}\rceil)(na_{i_2}-\lceil na_{i_2}\rceil) \rceil.
\end{aligned}
\end{equation*}
That is, $$h_i(n)=\lceil na_i+n^2a_i'+\sum_{1\le i_1<i_2\le
l_1}c'(i_1,i_2,i)\lceil na_{i_1}\rceil \lceil
na_{i_2}\rceil+\sum_{1\le j_1\le l_1}c'(j_1,i)n\lceil
na_{j_1}\rceil\rceil
$$ is a generalized polynomial of degree at most $2$ in $n$.

\medskip
Next we let
 $w_i(n)=na_i-h_i(n)=na_i-\lceil{na_i}\rceil$ for $1\le i\le l_1$ and
  if
$l_{j-1}+1\le i\le l_j,\ 2\le j\le d$, let
\begin{align*}
w_i(n)&=Q_i(n,{h_1(n)},\ldots,h_{i-1}(n))-h_i(n)\\
&=Q_i(n,{h_1(n)},\ldots,h_{i-1}(n))-\lceil{Q_i(n,{h_1(n)},\ldots,h_{i-1}(n))}\rceil.
\end{align*}
Since $Q_i(n,h_1(n),\ldots,h_{i-1}(n))$  is the sum of terms
$$c(k, k_1,\ldots,k_{i-1})n^kh_1^{k_1}(n)\ldots h_{i-1}^{k_{i-1}}(n)$$
with $k+k_1o(1)+\ldots+k_mo(m)\le o(i)$ and $h_i(n)$ is  a
generalized polynomial of degree of most $o(i)\le d$, we have
$w_i(n)$ is a generalized polynomial of degree of most $o(i)\le d$.


\medskip
Let $h(n)=\exp(h_1(n)X_1)\cdots \exp(h_m(n)X_m)$. Then $h(n)\in
\Gamma$ and $a^nh(n)^{-1}=\exp(w_1(n)X_1)\cdots \exp(w_m(n)X_m)$.
Choose $0<\ep<<\frac{1}{2}$ such that
$$\{g\Gamma: \rho (g\Gamma,e\Gamma)<\ep\}\subset U.$$ Then
$$A\supset N(e\Gamma,U)\supset \{n\in\Z:\rho (a^n\Gamma,e\Gamma)<\ep\}.$$

We get that (see Definition \ref{gtao})
\begin{equation*}
\rho(a^n\Gamma,e\Gamma)\le \rho(a^nh(n)^{-1},e)\le \max_{1\le k\le
m}\{||w_k(n)||\}.
\end{equation*}

So if $n\in \bigcap_{i=1}^m\{n\in\Z:w_i(n)\ (\text{mod}\ \Z)\in
(-\ep,\ep) \}$ then $\rho(a^n\Gamma,e\Gamma)<\ep$ which implies that
$n\in N(e\Gamma,U)\subset A$. That is,
$$A\supset \bigcap_{i=1}^m\{n\in\Z:w_i(n)\ (\text{mod}\ \Z)\in (-\ep,\ep) \}.$$
This ends the proof of Theorem B(1).

\section{Proof of Theorem B(2)}
\label{section-proof2}

In this section, we aim to prove Theorem B(2), i.e.
$\F_{d,0}\supset\F_{GP_d}$.
To do this first we make some preparations, then derive some results
under the inductive assumption, and finally give the proof. Note
that in the construction the nilpotent matrix Lie group is used.


More precisely, to show $\F_{d,0}\supset\F_{GP_d}$ we need only to
prove $\F_{d,0}\supset\F_{SGP_d}$ by Theorem B. To do this, for a
given $F\in \F_{SGP_d}$ we need to find a $d$-step nilsystem
$(X,T)$, $x_0\in X$ and a neighborhood $U$ of $x_0$ such that
$F\supset N(x_0,U)$. In the process doing this, we find that it is
convenient to consider a finite sum of specially generalized
polynomials $P(n;\alpha_1,\ldots,\alpha_r)$ (defined in
(\ref{key-poly})) instead of considering a single specially
generalized polynomial. We can prove that $\F_{d,0}\supset\F_{GP_d}$
if and only if $\{ n\in \mathbb{Z}:
||P(n;\alpha_1,\cdots,\alpha_d)||<\epsilon\}\in \mathcal{F}_{d,0}$
for any $\alpha_1,\cdots,\alpha_d\in \mathbb{R}$ and $\epsilon>0$
(Theorem \ref{lem-trans}). We choose $(X,T)$ as the closure of the
orbit of $\Gamma$ in $\G_d/\Gamma$ (the nilrotation is induced by a
matrix $A\in \G_d$), and consider the most right-corner entry
$z_1^d(n)$ in $A^nB_n$ with $B_n\in\Gamma$. We finish the proof by
showing that $P(n;\alpha_1,\cdots,\alpha_d)\simeq_d z_1^d(n)$ and
$\{n\in\Z: ||z_1^d(n)||<\ep\}\in \F_{d,0}$ for any $\ep>0$.

\subsection{Some preparations}

For a matrix $A$ in $\G_d$ we now give a precise formula of $A^n$.
\begin{lem} \label{lem-8-ite} Let ${\bf x}=(x_i^k)_{1\le k\le d, 1\le i\le
d-k+1}\in \mathbb{R}^{d(d+1)/2}$. For $n\in \mathbb{N}$, assume that
${\bf x}(n)=(x_i^k(n))_{1\le k\le d, 1\le i\le d-k+1}\in
\mathbb{R}^{d(d+1)/2}$ satisfies $\M({\bf x}(n))=\M({\bf x})^n$,
then
\begin{equation}\label{8-11-0}
x_i^k(n)=\tbinom{n}{1}P_1({\bf x};i,k)+\tbinom{n}{2}P_2({\bf
x};i,k)+\cdots+\tbinom{n}{k} P_k({\bf x};i,k)
\end{equation}
for $1\le k\le d$ and $1\le i\le d-k+1$, where
$\tbinom{n}{k}=\frac{n(n-1)\cdots (n-k+1)}{k!}$ for $n,k\in
\mathbb{N}$ and
$$P_\ell({\bf x};i,k)=\sum\limits_{(s_1,s_2,\cdots,s_{\ell})\in \{1,2,\cdots,k\}^\ell
\atop s_1+s_2+\cdots+s_\ell=k}x_i^{s_1}x_{i+s_1}^{s_2}
x_{i+s_1+s_2}^{s_3}\cdots x_{i+s_1+s_2+\cdots+s_{\ell-1}}^{s_\ell}$$
for $1\le k\le d$, $1\le i\le d-k+1$ and $1\le \ell \le k$.
\end{lem}
\begin{proof} Let $x_i^0=1$ and $x_i^0(m)=1$ for $1\le i\le d$ and
$m\in \mathbb{N}$. By \eqref{sec8-4-eq-1}, it is not hard to see
that
\begin{equation}\label{sec8-4-eq-2}
x_i^k(m+1)=\sum \limits_{j=0}^k x_i^{k-j}(m)\cdot x_{i+k-j}^j
\end{equation}
for $1\le k\le d$, $1\le i\le d-k+1$ and $m\in \mathbb{N}$.

Now we do induction for $k$. When $k=1$, $x_i^1(1)=x_i^1$ and
$x_i^1(m+1)=x_i^1(m)+x_i^1$ for $m\in \mathbb{N}$ by
\eqref{sec8-4-eq-2}. Hence $x_i^1(n)=nx_i^1=\tbinom{n}{1}P_1({\bf
x};i,1)$. That is, \eqref{8-11-0} holds for each $1\le i\le d$ and
$n\in \mathbb{N}$ if $k=1$.

Assume that $1\le \ell \le d-1$, and \eqref{8-11-0} holds for each
$1\le k\le \ell$, $1\le i\le d-k+1$ and $n\in \mathbb{N}$. For
$k=\ell+1$, we make induction on $n$. When $n=1$ it is clear
$$x_i^k(1)=x_i^k=\tbinom{1}{1}P_1({\bf x};i,k)+\tbinom{1}{2}P_2({\bf
x};i,k)+\cdots+\tbinom{1}{k} P_k({\bf x};i,k)$$ for $1\le i\le
d-k+1$. That is, \eqref{8-11-0} holds for $k=\ell+1$, $1\le i\le
d-k+1$ and $n=1$. Assume for $n=m\ge 1$,  \eqref{8-11-0} holds for
$k=\ell+1$, $1\le i\le d-k+1$ and $n=m$. For $n=m+1$, by
\eqref{sec8-4-eq-2}
\begin{align*}
x_i^k(n)&=x_i^k(m)+\Big(\sum \limits_{j=1}^{k-1} x_i^{k-j}(m)\cdot x_{i+k-j}^j\Big)+x_i^k\\
&=x_i^k(m)+\Big(\sum \limits_{j=1}^{k-1} (\sum \limits_{r=1}^{k-j}\tbinom{m}{r}P_r({\bf x};i,k-j))\cdot x_{i+k-j}^j)\Big)+x_i^k\\
&=x_i^k(m)+\Big(\sum \limits_{r=1}^{k-1} (\sum \limits_{j=1}^{k-r} P_r({\bf x};i,k-j)x_{i+k-j}^{j})\tbinom{m}{r}\Big)+x_i^k\\
&=x_i^k(m)+\Big(\sum \limits_{r=1}^{k-1} (\sum \limits_{j=r}^{k-1}
P_r({\bf x};i,j)x_{i+j}^{k-j})\tbinom{m}{r}\Big)+x_i^k
\end{align*}
for $1\le i\le d-k+1$. Note that

$$\sum \limits_{j=r}^{k-1} P_r({\bf x};i,j)x_{i+j}^{k-j}=\sum
\limits_{j=r}^{k-1} \sum \limits_{(s_1,\cdots,s_r)\in
\{1,2,\cdots,k-1\}^r \atop
s_1+\cdots+s_r=j}x_i^{s_1}x_{i+s_1}^{s_2}\cdots
x_{i+s_1+\cdots+s_{r-1}}^{s_r} x_{i+j}^{k-j}$$ which is equal to
$$\sum \limits_{(s_1,\cdots,s_r,s_{r+1})\in \{1,2,\cdots,k-1\}^{r+1}
\atop s_1+s_2+\cdots+s_r+s_{r+1}=k}x_i^{s_1}x_{i+s_1}^{s_2}\cdots
x_{i+s_1+\cdots+s_{r-1}}^{s_r}
x_{i+s_1+\cdots+s_{r-1}+s_r}^{s_{r+1}}=P_{r+1}({\bf x};i,k)$$ for
$1\le r\le k-1$ and $1\le i\le d-k+1$. Collecting terms we have

\begin{align*}
x_i^k(n)&=x_i^k(m)+\Big(\sum \limits_{r=1}^{k-1} P_{r+1}({\bf x};i,k) \tbinom{m}{r}\Big)+x_i^k\\
&=x_i^k(m)+\Big(\sum \limits_{r=2}^{k} P_{r}({\bf x};i,k) \tbinom{m}{r-1}\Big)+P_1({\bf x};i,k)\\
&=\Big(\sum \limits_{r=1}^m P_r({\bf x};i,k)\tbinom{m}{r}\Big)+
\Big(\sum \limits_{r=2}^{k} P_{r} ({\bf x};i,k)
\tbinom{m}{r-1}\Big)+P_1({\bf x};i,k).
\end{align*}

Rearranging the order we get
\begin{align*}
x_i^k(n)&=(m+1)P_1({\bf x};i,k)+\sum_{r=2}^k \Big(\tbinom{m}{r}+\tbinom{m}{r-1}\Big)P_r({\bf x};i,k)\\
&=\sum \limits_{r=1}^k \tbinom{m+1}{r}P_r({\bf x};i,k)=\sum
\limits_{r=1}^k \tbinom{n}{r}P_r({\bf x};i,k)
\end{align*}
for $1\le i\le d-k+1$. This ends the proof of the lemma.
\end{proof}

\begin{rem} \label{rem-8-ite}
By the above lemma, we have $$P_1({\bf x};i,k)=x_i^k \text{ and
 }P_k({\bf x};i,k)=x_i^1x_{i+1}^1\cdots x_{i+k-1}^1$$ for $1\le k\le
d$ and $1\le i\le d-k+1$.
\end{rem}

\subsection{Consequences under the inductive assumption}
We will use induction to show Theorem B(2). To make the proof
clearer, we derive some results under the following inductive
assumption.
\begin{equation}\label{as-d-1}
\mathcal{F}_{d-1,0}\supset \mathcal{F}_{GP_{d-1}},
\end{equation}
where $d\in\mathbb{N}$ with $d\ge 2$. For that purpose, we need more
notions and lemmas. The proof of Lemma \ref{lem-8-11} is similar to
the one of Lemma \ref{lem1-1}, where $\mathcal{W}_d$ is defined in
Definition \ref{de-ws}.

\begin{lem}\label{lem-8-11}
Under the assumption \eqref{as-d-1}, one has for any $p(n)\in
\mathcal{W}_d$ and $\epsilon>0$,
$$\{n\in \mathbb{Z}:p(n)\ (\text{mod}\ \Z) \in (-\epsilon,\epsilon)\}\in \mathcal{F}_{d-1,0}.$$
\end{lem}

\begin{de}
For $r\in \mathbb{N}$, we define
$${\widetilde {GP}}_r=\{ p(n)\in GP_r: \{ n\in \mathbb{Z}: p(n)\ (\text{mod}\ \Z)\in (-\epsilon,\epsilon)\}\in
\mathcal{F}_{r,0} \text{ for any }\epsilon>0\}.$$
\end{de}
\begin{rem}\label{rem-8-15} It is clear that for $p(n)\in GP_r$,
$p(n)\in {\widetilde {GP}}_r$ if and only if $-p(n)\in {\widetilde
{GP}}_r$. Since $\mathcal{F}_{r,0}$ is a filter, if
$p_1(n),p_2(n),\cdots,p_k(n)\in {\widetilde {GP}}_r$ then
$$p_1(n)+p_2(n)+\cdots+p_k(n)\in {\widetilde {GP}}_r.$$ Moreover by the definition
of ${\widetilde {GP}}_d$, we know that $\mathcal{F}_{d,0}\supset
\mathcal{F}_{GP_d}$ if and only if ${\widetilde {GP}}_d=GP_d$.
\end{rem}

\begin{lem}\label{lem-8-equi-new} Let $p(n),q(n)\in GP_d$ with $p(n)\simeq_d q(n)$. Under the assumption \eqref{as-d-1},
$p(n)\in {\widetilde {GP}}_d$ if and only if $q(n)\in {\widetilde
{GP}}_d$.
%
\end{lem}
\begin{proof} This follows from Lemma \ref{lem-8-11} and the fact that $\mathcal{F}_{d,0}$ is a filter.
\end{proof}

For $\alpha_1,\alpha_2,\ldots, \alpha_r\in \mathbb{R}, r\in \N$, we
define
\begin{align}\label{key-poly}
   &P(n;\alpha_1,\alpha_2,\cdots,\alpha_r) \nonumber\\
=&\sum \limits_{\ell=1}^r \sum \limits_{j_1,\cdots, j_\ell\in
\mathbb{N}\atop j_1+\cdots+j_\ell=r}(-1)^{\ell-1} L\Big(
\frac{n^{j_1}}{j_1!}\prod \limits_{r_1=1}^{j_1}\alpha_{r_1},\,
\frac{n^{j_2}}{j_2!}\prod \limits_{r_2=1}^{j_2}\alpha_{j_1+r_2},\,
\cdots,\, \frac{n^{j_\ell}}{j_\ell!}\prod
\limits_{r_\ell=1}^{j_\ell}\alpha_{\sum
\limits_{t=1}^{\ell-1}j_t+r_\ell} \Big)
\end{align}
where the definition of $L$ is given in \eqref{eq-de-L}.

\begin{thm} \label{lem-trans}
Under the assumption \eqref{as-d-1}, the following properties are
equivalent:
\begin{enumerate}
\item  $\mathcal{F}_{d,0}\supset \mathcal{F}_{GP_d}$.

\item $P(n;\alpha_1,\alpha_2,\cdots,\alpha_d)\in {\widetilde {GP}}_d$ for any
$\alpha_1,\alpha_2,\cdots,\alpha_d\in \mathbb{R}$, that is
$$\{ n\in \mathbb{Z}: P(n;\alpha_1,\alpha_2,\cdots,\alpha_d)\ (\text{mod}\ \Z)
\in (-\epsilon, \epsilon) \}\in \mathcal{F}_{d,0}$$ for any
$\alpha_1,\alpha_2,\cdots,\alpha_d\in \mathbb{R}$ and $\epsilon>0$.

\item $\text{SGP}_d\subset {\widetilde {GP}}_d$.
\end{enumerate}
\end{thm}
\begin{proof} $(1)\Rightarrow(2)$. Assume $\mathcal{F}_{d,0}\supset \mathcal{F}_{GP_d}$.
By the definition of ${\widetilde {GP}}_d$, we know that
$\mathcal{F}_{d,0}\supset \mathcal{F}_{GP_d}$ if and only if
${\widetilde {GP}}_d=GP_d$. Particularly
$P(n;\alpha_1,\alpha_2,\cdots,\alpha_d)\in {\widetilde {GP}}_d$ for
any $\alpha_1,\alpha_2,\cdots,\alpha_d\in \mathbb{R}$.

\medskip
$(3)\Rightarrow (1)$. Assume that $\text{SGP}_d\subset {\widetilde
{GP}}_d$. Then $\mathcal{F}_{d,0}\supseteq \mathcal{F}_{SGP_d}$.
Moveover $\mathcal{F}_{d,0}\supset\F_{SGP_d}=\F_{GP_d}$ by Theorem
\ref{gpsgp}.

\medskip
$(2)\Rightarrow (3)$.  Assume that
$P(n;\alpha_1,\alpha_2,\cdots,\alpha_d)\in {\widetilde {GP}}_d$ for
any $\alpha_1,\alpha_2,\cdots,\alpha_d\in \mathbb{R}$. We define
$$\Sigma_d=\{ (j_1,j_2,\cdots,j_\ell): \ell\in \{1,2,\cdots,d\}, j_1,j_2,\cdots,
j_\ell\in \mathbb{N} \text{ and }\sum \limits_{t=1}^\ell j_t=d\}. $$
For $(j_1,j_2,\cdots,j_\ell),(r_1,r_2,\cdots,r_s)\in \Sigma_d$, we
say $(j_1,j_2,\cdots,j_\ell)>(r_1,r_2,\cdots,r_s)$ if there exists
$1\le t\le \ell$ such that $j_t>r_s$ and $j_i=r_i$ for $i<t$.
Clearly $(\Sigma_d,>)$ is a totally ordered set with the maximal
element $(d)$ and the minimal element $(1,1,\cdots,1)$.

For ${\bf j}=(j_1,j_2,\cdots,j_\ell)\in \Sigma_d$, put
$$\mathcal{L}({\bf j})=\{L(n^{j_1}a_1,\cdots,n^{j_\ell}a_\ell):a_1,\cdots,a_\ell\in \mathbb{R}\}.$$
Now, we have

\medskip
\noindent{\bf Claim:} $\mathcal{L}({\bf s})\subseteq {\widetilde
{GP}}_d$ for each ${\bf s}\in \Sigma_d$.

\begin{proof} We do induction for ${\bf s}$ under the order $>$.
First, consider the case when ${\bf s}=(d)$. Given  $a_1\in
\mathbb{R}$, we take $\alpha_1=1,\alpha_2=2,\cdots\alpha_{d-1}=d-1$
and $\alpha_d=d a_1$. Then for any $1\le j_1\le d-1$,
$\frac{n^{j_1}}{j_1!}\prod \limits_{t=1}^{j_1} \alpha_t \in
\mathbb{Z}$ for $n\in \mathbb{Z}$. Thus
$$P(n;\alpha_1,\alpha_2,\cdots,\alpha_d)=L(\frac{n^d}{d!}\prod
\limits_{t=1}^{d} \alpha_t)=L(n^d a_1) \ (\text{mod}\ \Z)$$ for any
$n\in \mathbb{Z}$. Hence $L(n^d a_1)\in {\widetilde {GP}}_d$ since
$P(n;\alpha_1,\alpha_2,\cdots,\alpha_d)\in {\widetilde {GP}}_d$.
Since $a_1$ is arbitrary, we conclude that $\mathcal{L}((d))\subset
{\widetilde {GP}}_d$.

Assume that for any ${\bf s}>{\bf i}=(i_1,\cdots,i_k)\in \Sigma_d$,
we have $\mathcal{L}({\bf s})\subset {\widetilde {GP}}_d$. Now
consider the case when  ${\bf s}={\bf i}=(i_1,\cdots,i_k)$. There
are two cases.

The first case is $k=d$, $i_1=i_2=\cdots =i_d=1$. Given
$a_1,a_2,\cdots,a_d \in \R$, by the assumption we have that for any
$(j_1,j_2,\cdots,j_{\ell})>{\bf i}$,
$\mathcal{L}((j_1,j_2,\cdots,j_{\ell}))\subset {\widetilde {GP}}_d$.
Thus
$$\sum \limits_{\ell=1}^{d-1} \sum \limits_{j_1,\cdots, j_\ell\in
\mathbb{N}\atop j_1+\cdots+j_\ell=r}(-1)^{\ell-1} L\Big(
\frac{n^{j_1}}{j_1!}\prod \limits_{r_1=1}^{j_1}a_{r_1},\,
\frac{n^{j_2}}{j_2!}\prod \limits_{r_2=1}^{j_2}a_{j_1+r_2},\,
\cdots,\, \frac{n^{j_\ell}}{j_\ell!}\prod
\limits_{r_\ell=1}^{j_\ell}a_{\sum_{t=1}^{\ell-1}j_t+r_\ell} \Big)$$
belongs to ${\widetilde {GP}}_d$ by the Remark \ref{rem-8-15}. This
implies  $$P(n;a_1,a_2,\cdots,a_d)-L(n a_1,n a_2,\cdots,n a_d)\in
{\widetilde {GP}}_d$$ by \eqref{key-poly}. Combining this with
$P(n;a_1,a_2,\cdots,a_d)\in {\widetilde {GP}}_d$, we have $$L(n
a_1,n a_2,\cdots,n a_d)\in {\widetilde {GP}}_d$$ by  Remark
\ref{rem-8-15}. Since $a_1,a_2,\cdots,a_d \in \R$ are arbitrary, we
get $\mathcal{L}({\bf i})\subset {\widetilde {GP}}_d$.

The second case is ${\bf i}>(1,1,\cdots,1)$. Given
$a_1,a_2,\cdots,a_k\in \mathbb{R}$, for $r=1,2,\cdots,k$, we put
$\alpha_{\sum_{t=1}^{r-1}i_t+h}=h$ for $1\le h\le i_r-1$ and
$\alpha_{\sum_{t=1}^{r-1}i_t+i_r}=i_r a_r$.

By the assumption, for $(j_1,j_2,\cdots,j_{\ell})>{\bf i}$,
 $$L\Big(
\frac{n^{j_1}}{j_1!}\prod \limits_{r_1=1}^{j_1}\alpha_{r_1},\,
\frac{n^{j_2}}{j_2!}\prod \limits_{r_2=1}^{j_2}\alpha_{j_1+r_2},\,
\cdots,\, \frac{n^{j_\ell}}{j_\ell!}\prod
\limits_{r_\ell=1}^{j_\ell}\alpha_{\sum_{t=1}^{\ell-1}j_t+r_\ell}
\Big)\in {\widetilde {GP}}_d.$$

For $(j_1,j_2,\cdots,j_\ell)<{\bf i}$, there exists $1\le u\le k$
such that $j_t=i_t$ for $1\le t\le u-1$ and $i_{u}>j_u$. Then
\begin{equation}\label{8-interger}
\frac{n^{j_u}}{j_u!}\prod
\limits_{r_u=1}^{j_u}\alpha_{\sum_{t=1}^{u-1}j_t+r_u}=n^{j_u}.
\end{equation}
When $u=1$,  by \eqref{8-interger}, $$L\Big(
\frac{n^{j_1}}{j_1!}\prod \limits_{r_1=1}^{j_1}\a_{r_1},\cdots,\,
\cdots,\, \frac{n^{j_\ell}}{j_\ell!}\prod
\limits_{r_\ell=1}^{j_\ell}\a_{\sum_{t=1}^{\ell-1}j_t+r_\ell}
\Big) \in \mathbb{Z}$$ for any $n\in \mathbb{Z}$. Hence
$$L\Big(
\frac{n^{j_1}}{j_1!}\prod \limits_{r_1=1}^{j_1}\a_{r_1},\cdots,\,
\cdots,\, \frac{n^{j_\ell}}{j_\ell!}\prod
\limits_{r_\ell=1}^{j_\ell}\a_{\sum_{t=1}^{\ell-1}j_t+r_\ell}
\Big)\in \widetilde{GP}_d.$$ When $u>1$, write
$\beta_v=\prod_{r_v=1}^{j_v}\alpha_{\sum_{t=1}^{v-1}j_t+r_v}$ for
$v=1,2,\cdots,\ell$. Then $\beta_u=1$ and $$\lceil
L(n^{j_u}\beta_u,n^{j_{u+1}}\beta_{u+1}\, \cdots,\,
n^{j_\ell}\beta_\ell)\rceil=L(n^{j_u}\beta_u,n^{j_{u+1}}\beta_{u+1}\,
\cdots,\, n^{j_\ell}\beta_\ell).$$ Moreover, $$ L\Big(
\frac{n^{j_1}}{j_1!}\prod
\limits_{r_1=1}^{j_1}\a_{r_1},\cdots,\frac{n^{j_u}}{j_u!}\prod
\limits_{r_u=1}^{j_u}\alpha_{\sum_{t=1}^{u-1}j_t+r_\ell},\,
\cdots,\, \frac{n^{j_\ell}}{j_\ell!}\prod
\limits_{r_\ell=1}^{j_\ell}\a_{\sum_{t=1}^{\ell-1}j_t+r_\ell}
\Big)$$ is equal to $$L\left(
n^{j_1}\beta_1,\cdots,n^{j_u}\beta_u,\, \cdots,\,
n^{j_\ell}\beta_\ell \right)=L\left(
n^{j_1}\beta_1,\cdots,n^{j_{u-1}}\beta_{u-1}\lceil L(n^{j_u}\beta_u,
\cdots,\, n^{j_\ell}\beta_\ell)\rceil \right)$$ which is equal to
\begin{align*} &L\left( n^{j_1}\beta_1,\cdots,n^{j_{u-1}}\beta_{u-1}
L(n^{j_u}\beta_u,n^{j_{u+1}}\beta_{u+1}\,
\cdots,\, n^{j_\ell}\beta_\ell)\right)\\
&=L\left( n^{j_1}\beta_1,\cdots,n^{j_{u-1}+j_u}\beta_{u-1}\beta_u
\lceil L(n^{j_{u+1}}\beta_{u+1}\, \cdots,\,
n^{j_\ell}\beta_\ell)\rceil \right)\\
&=L\left( n^{j_1}\beta_1,\cdots,n^{j_{u-1}+j_u}\beta_{u-1}\beta_u,
n^{j_{u+1}}\beta_{u+1}\, \cdots,\, n^{j_\ell}\beta_\ell \right) \in
{\widetilde {GP}}_d
\end{align*}
since $(j_1,\cdots,j_{u-2},j_{u-1}+j_u,j_{u+1},\cdots,j_\ell)>{\bf
i}$.

Summing up for any ${\bf j}=(j_1,\cdots,j_\ell)\in \Sigma_d$ with
${\bf j}\neq {\bf i}$, we have
$$L\Big( \frac{n^{j_1}}{j_1!}\prod
\limits_{r_1=1}^{j_1}\a_{r_1},\cdots,\frac{n^{j_u}}{j_u!}\prod
\limits_{r_u=1}^{j_u}\alpha_{\sum_{t=1}^{u-1}j_t+r_\ell},\,
\cdots,\, \frac{n^{j_\ell}}{j_\ell!}\prod
\limits_{r_\ell=1}^{j_\ell}\a_{\sum_{t=1}^{\ell-1}j_t+r_\ell} \Big)
\in {\widetilde {GP}}_d.$$ Combining this with
$P(n;\alpha_1,\cdots,\alpha_d)\in {\widetilde {GP}}_d$, we have
\begin{align*}
&\hskip0.6cm L\Big( n^{i_1}a_1,\, n^{i_2}a_2,\, \cdots,\,
n^{i_k}a_k \Big)\\
&= L\Big( \frac{n^{i_1}}{i_1!}\prod
\limits_{r_1=1}^{i_1}\alpha_{r_1},\, \frac{n^{i_2}}{i_2!}\prod
\limits_{r_2=1}^{i_2}\alpha_{i_1+r_2},\, \cdots,\,
\frac{n^{i_k}}{i_k!}\prod
\limits_{r_k=1}^{i_k}\alpha_{\sum_{t=1}^{k-1}i_t+r_k} \Big)\in
{\widetilde {GP}}_d
\end{align*}
by \eqref{key-poly} and Remark \eqref{rem-8-15}. Since
$a_1,\cdots,a_k \in \R$ are arbitrary, $\mathcal{L}({\bf i})\subset
{\widetilde {GP}}_d$.
\end{proof}

Finally, since $\text{SGP}_d=\bigcup_{{\bf j}\in \Sigma_d}
\mathcal{L}({\bf j})$, we have $\text{SGP}_d\subset {\widetilde
{GP}}_d$ by the above Claim.
\end{proof}


\subsection{Proof of Theorem B(2)}

We are now ready to give the proof of the Theorem B(2). As we said
before, we will use induction to show Theorem B(2). Firstly, for
$d=1$, since $\mathcal{F}_{GP_1}=\mathcal{F}_{SGP_1}$ and
$\mathcal{F}_{1,0}$ is a filter, it is sufficient to show for any
$a\in \mathbb{R}$ and $\epsilon>0$,
$$\{n\in \mathbb{Z}: an \ (\text{mod}\ \Z)\in (-\epsilon,\epsilon)\}\in \mathcal{F}_{1,0}.$$
This is obvious since the rotation on the unit circle is a 1-step
nilsystem.

\medskip
Now we assume that $\mathcal{F}_{d-1,0}\supset
\mathcal{F}_{GP_{d-1}}$, i.e. the the assumption \eqref{as-d-1}
holds. By Theorem \ref{lem-trans}, to show $\mathcal{F}_{d,0}\supset
\mathcal{F}_{GP_d}$, it remains to prove that
$P(n;\alpha_1,\alpha_2,\cdots,\alpha_d)\in {\widetilde {GP}}_d$ for
any $\alpha_1,\alpha_2,\ldots,\alpha_d\in \mathbb{R}$, that is
$$\{ n\in \mathbb{Z}: P(n;\alpha_1,\alpha_2,\cdots,\alpha_d)\ (\text{mod}\ \Z)
\in (-\epsilon, \epsilon) \}\in \mathcal{F}_{d,0}$$ for any
$\alpha_1,\alpha_2,\cdots,\alpha_d\in \mathbb{R}$ and $\epsilon>0$.

Let $\alpha_1,\alpha_2,\cdots,\alpha_d\in \mathbb{R}$ and choose
${\bf x}=(x_i^k)_{1\le k\le d, 1\le i\le d-k+1}\in
\mathbb{R}^{d(d+1)/2}$ with $x_i^1=\alpha_i$ for $i=1,2,\cdots,d$
and $x_i^k=0$ for $2\le k\le d$ and $1\le i\le d-k+1$. Then
{\footnotesize
$$
A=\M({\bf x})=\left(
  \begin{array}{cccccccc}
    1 & \alpha_1 & 0                &\ldots & 0 &0\\
    0 & 1     & \alpha_2            &\ldots &  0&0\\
    \vdots & \vdots  &  \vdots & \vdots &\vdots\\
    0 & 0    &0                   & \ldots & \alpha_{d-1} &0 \\
    0 & 0    &0                   & \ldots & 1 &\alpha_d \\
    0 & 0    &0                 & \ldots & 0  & 1
  \end{array}
\right)
$$}
For $n\in \mathbb{N}$, if ${\bf x}(n)=(x_i^k(n))_{1\le k\le d, 1\le
i\le d-k+1}\in \mathbb{R}^{d(d+1)/2}$ satisfies $\M({\bf
x}(n))=A^n$, then by Lemma \ref{lem-8-ite} and Remark
\ref{rem-8-ite},
\begin{equation}\label{thm-8-11-0}
x_i^k(n)=\tbinom{n}{k} P_k({\bf
x};i,k)=\tbinom{n}{k}x_i^1x_{i+1}^1\cdots
x_{i+k-1}^1=\tbinom{n}{k}\alpha_i \alpha_{i+1}\cdots \alpha_{i+k-1}
\end{equation}
for $1\le k\le d$ and $1\le i\le d-k+1$.

Now we define $f_i^1(n)=\lceil x_i^1(n)\rceil=\lceil n\alpha_i
\rceil$ for $1\le i\le d$ and inductively for $k=2,3,\cdots,d$
define
\begin{equation}\label{sec8-dtgs-1}
f_i^k(n)=\bigg\lceil x_i^k(n)-\sum
\limits_{j=1}^{k-1}x_i^{k-j}(n)f_{i+k-j}^j(n)  \bigg\rceil
\end{equation}
for $1\le i \le d-k+1$. Then we define
$$z_i^1(n)=x_i^1(n)-f_i^1(n)$$ for $1\le i\le d$ and inductively for
$k=2,3,\cdots,d$ define
\begin{equation}\label{sec8-thm-eq-1-old}
z_i^k(n)=x_i^k(n)-\Big( \sum
\limits_{j=1}^{k-1}x_i^{k-j}(n)f_{i+k-j}^j(n) \Big)-f_i^k(n)
\end{equation}
for $1\le i \le d-k+1$.

It is clear that $z_i^k(n)\in GP_k$ for $1\le k\le d$ and $1\le i\le
d-k+1$. First, we have

\medskip
\noindent{\bf Claim:} $P(n;\alpha_1,\alpha_2,\cdots,\alpha_d)
\simeq_d z_1^d(n)$.

Since the proof of the Claim is long, the readers find the proof in
the following subsection. Now we are going to show $z_1^d(n) \in
{\widetilde {GP}}_d$.

Let $X=\G_d/\Gamma$ be endowed with the metric $\rho$ in Lemma
\ref{matix-metric} and $T$ be the nilrotation induced by $A\in\G_d$,
i.e. $B\Gamma\mapsto AB\Gamma$ for $B\in\G_d$. Since $\G_d$ is a
$d$-step nilpotent Lie group and $\Gamma$ is a uniform subgroup of
$\G_d$, $(X,T)$ is a $d$-step nilsystem. Let $x_0=\Gamma\in X$ and
$Z$ be the closure of the orbit $\text{orb}(x_0,T)$ of $e$ in $X$.
Then $(Z,T)$ is a minimal $d$-step nilsystem. We consider $\rho$ as
a metric on $Z$.

For a given $\eta>0$ choose $\delta>0$ such that
$e^{\delta+\delta^2+\cdots+\delta^d}-1<\min\{\frac{1}{2},\eta\}$.
Put $$U=\{ z\in Z: \rho (z,x_0)<\delta\}$$ and $$S=\{ n\in
\mathbb{N}: \rho (A^n\Gamma, \Gamma)<\delta\}=\{ n\in
\mathbb{Z}:T^nx_0 \in U\}.$$ Then $S\in \mathcal{F}_{d,0}$ since
$(Z,T)$ is a minimal $d$-step nilsystem. In the following we are
going to show that
$$\{ m\in \mathbb{Z}: z_1^d(m)\ (\text{mod} \ \Z)\in (-\eta,\eta) \}\supset S.$$
This clearly implies that $\{ m\in \mathbb{Z}: z_1^d(m)\ (\text{mod}
\ \Z)\in (-\eta,\eta) \}\in \mathcal{F}_{d,0}$ since $S\in
\mathcal{F}_{d,0}$. As $\eta>0$ is arbitrary, we conclude that
$z_1^d(n) \in {\widetilde {GP}}_d$.

Given $n\in S$, one has $\rho(A^n\Gamma,\Gamma)<\delta$. Since
$\rho$ is right-invariant and $\Gamma$ is a group, there exists
$B_n^{-1}\in \Gamma$ such that $\rho (A^n,B_n^{-1})<\delta$. Take
${\bf h}(n)=(-h_i^k(n))_{1\le k\le d, 1\le i\le d-k+1}\in
\mathbb{Z}^{d(d+1)/2}$ with $\M({\bf h}(n))=B_n$. By
\eqref{matrix-key},
\begin{equation}\label{me-eq-222}
\|A^nB_n-I\|\le
e^{\delta+\delta^2+\cdots+\delta^d}-1<\min\Big\{\frac{1}{2},\eta\Big\}.
\end{equation}
Let ${\bf y}(n)=(y_i^k(n))_{1\le k\le d, 1\le i\le d-k+1}\in
\mathbb{R}^{d(d+1)/2}$ such that
$$\M({\bf y}(n))=A^nB_n=\M({\bf x}(n))\M({\bf h}(n)).$$
By \eqref{sec8-4-eq-1}
\begin{equation}\label{sec8-thm-eq-1}
y_i^k(n)=x_i^k(n)-\Big( \sum
\limits_{j=1}^{k-1}x_i^{k-j}(n)h_{i+k-j}^j(n) \Big)-h_i^k(n)
\end{equation}
for $1\le k \le d$ and $1\le i \le d-k+1$. Thus
\begin{equation}\label{sec8-thm-eq-2}
|y_i^k(n)|<\min\{\frac{1}{2},\eta\}
\end{equation}
 for $1\le k \le d$ and $1\le i \le d-k+1$ by \eqref{me-eq-222}.
Hence $h_i^1(n)=\lceil x_i^1(n)\rceil=\lceil n\alpha_i \rceil$ for
$1\le i\le d$ and
\begin{equation}\label{sec8-dtgs-2}
h_i^k(n)=\bigg \lceil x_i^k(n)- \sum
\limits_{j=1}^{k-1}x_i^{k-j}(n)h_{i+k-j}^j(n) \bigg \rceil
\end{equation}
for $2\le k \le d$ and $1\le i \le d-k+1$.

Since $h_i^1(n)=\lceil n\alpha_i \rceil=f_i^1(n)$ for $1\le i\le d$,
one has $h_i^k(n)=f_i^k(n)$ for $2\le k \le d$ and $1\le i \le
d-k+1$ by \eqref{sec8-dtgs-1} and \eqref{sec8-dtgs-2}. Moreover by
\eqref{sec8-thm-eq-1-old} and \eqref{sec8-thm-eq-1}, we know
$z_i^k(n)=y_i^k(n)$ for $2\le k \le d$ and $1\le i \le d-k+1$.
Combining this with \eqref{sec8-thm-eq-2},
$|z_i^k(n)|<\min\{\frac{1}{2},\eta\}$ for $1\le k \le d$ and $1\le i
\le d-k+1$. Particularly, $|z_1^d(n)|<\eta$. Thus $$n\in \{ m\in
\mathbb{Z}: z_1^d(m)\ (\text{mod} \ \Z)\in (-\eta,\eta)\},$$ which
implies that $ \{ m\in \mathbb{Z}: z_1^d(m)\ (\text{mod} \ \Z)\in
(-\eta,\eta)\}\supset S$. That is, $z_1^d(n) \in {\widetilde
{GP}}_d$.

Finally using the Claim and the fact that $z_1^d(n) \in {\widetilde
{GP}}_d$ we have $P(n;\alpha_1,\alpha_2,\cdots,\alpha_d)\in
{\widetilde {GP}}_d$ by Lemma \ref{lem-8-equi-new}. This ends the
proof, i.e. we have proved $\mathcal{F}_{d,0}\supset
\mathcal{F}_{GP_d}$.

\subsection{Proof of the Claim} Let $$u_i^k(n)=z_i^k(n)+f_i^k(n)=x_i^k(n)-\sum
\limits_{j=1}^{k-1}x_i^{k-j}(n)f_{i+k-j}^j(n) $$ for
 $1\le k\le d$
and $1\le i\le d-k+1$. Then $$f_i^k(n)=\lceil u_i^k(n) \rceil $$ for
 $1\le k\le d$
and $1\le i\le d-k+1$.

We define $U(n;j_1)=\frac{n^{j_1}}{j_1!}\prod_{r=1}^{j_1}\alpha_r$
for $1\le j_1\le d$. Then inductively for $\ell=2,3,\cdots,d$ we
define
\begin{align*}
U(n;j_1,j_2,\cdots,j_\ell)&=(U(n;j_1,\cdots,j_{\ell-1})-\lceil
U(n;j_1,\cdots,j_{\ell-1})\rceil )
\frac{n^{j_\ell}}{j_\ell!}\prod_{r=1}^{j_\ell}
\alpha_{\sum_{t=1}^{\ell-1}j_t+r}\\
&=(U(n;j_1,\cdots,j_{\ell-1})-\lceil
U(n;j_1,\cdots,j_{\ell-1})\rceil
)L(\frac{n^{j_\ell}}{j_\ell!}\prod_{r_{\ell}=1}^{j_\ell}
\alpha_{\sum_{t=1}^{\ell-1}j_t+r_{\ell}})
\end{align*}
for $j_1,j_2,\cdots,j_\ell\ge 1$ and $j_1+\cdots+j_\ell\le d$ (see
\eqref{eq-de-L} for the definition of $L$).

Next, $U(n;d)=\frac{n^{d}}{d!}\prod
\limits_{r=1}^{d}\alpha_r=L(\frac{n^{d}}{d!}\prod
\limits_{r=1}^{d}\alpha_r)$ and for $2\le \ell \le d$,
$j_1,j_2,\cdots,j_\ell\in \mathbb{N}$ with
$j_1+j_2+\cdots+j_\ell=d$, by Lemma \ref{lem8-13}(1)

\begin{align*}
U(n;j_1,j_2,\cdots,j_\ell)&=(U(n;j_1,\cdots,j_{\ell-1})-\lceil
U(n;j_1,\cdots,j_{\ell-1}) \rceil)
L(\frac{n^{j_\ell}}{j_\ell!}\prod_{r_{\ell}=1}^{j_\ell}
\alpha_{\sum_{t=1}^{\ell-1}j_t+r_{\ell}})\\&\simeq_d
U(n;j_1,\cdots,j_{\ell-1}) \lceil
L(\frac{n^{j_\ell}}{j_\ell!}\prod_{r_{\ell}=1}^{j_\ell}
\alpha_{\sum_{t=1}^{\ell-1}j_t+r_{\ell}})\rceil
\end{align*}
which is
\begin{align*}
&=(U(n;j_1,\cdots,j_{\ell-2})-\lceil U(n;j_1,\cdots,j_{\ell-2}) \rceil)\times \\
&\hskip 4cm
L(\frac{n^{j_{\ell-1}}}{j_{\ell-1}!}\prod_{r_{\ell-1}=1}^{j_{\ell-1}}\alpha_{\sum_{t=1}^{\ell-2}j_t+r_{\ell-1}},
\frac{n^{j_\ell}}{j_\ell!}\prod_{r_{\ell}=1}^{j_\ell}
\alpha_{\sum_{t=1}^{\ell-1}j_t+r_{\ell}})\\
&\simeq_d U(n;j_1,\cdots,j_{\ell-2}) \lceil
L(\frac{n^{j_{\ell-1}}}{j_{\ell-1}!}\prod_{r_{\ell-1}=1}^{j_{\ell-1}}\alpha_{\sum_{t=1}^{\ell-2}j_t+r_{\ell-1}},
\frac{n^{j_\ell}}{j_\ell!}\prod_{r_{\ell}=1}^{j_\ell}
\alpha_{\sum_{t=1}^{\ell-1}j_t+r_{\ell}})\rceil.
\end{align*}
Continuing the above argument we have
$$U(n;j_1,j_2,\cdots,j_\ell)\simeq_d L\Big(
\frac{n^{j_1}}{j_1!}\prod \limits_{r_1=1}^{j_1}\alpha_{r_1},\,
\frac{n^{j_2}}{j_2!}\prod \limits_{r_2=1}^{j_2}\alpha_{j_1+r_2},\,
\cdots,\, \frac{n^{j_\ell}}{j_\ell!}\prod
\limits_{r_\ell=1}^{j_\ell}\alpha_{\sum_{t=1}^{\ell-1}j_t+r_\ell}
\Big).$$

That is, for $1\le \ell \le d$, $j_1,j_2,\cdots,j_\ell\in
\mathbb{N}$ with $j_1+j_2+\cdots+j_\ell=d$,

\begin{equation}\label{eq-u=l}
U(n;j_1,j_2,\cdots,j_\ell) \simeq_d L\Big(
\frac{n^{j_1}}{j_1!}\prod \limits_{r_1=1}^{j_1}\alpha_{r_1},\,
\frac{n^{j_2}}{j_2!}\prod \limits_{r_2=1}^{j_2}\alpha_{j_1+r_2},\,
\cdots,\, \frac{n^{j_\ell}}{j_\ell!}\prod
\limits_{r_\ell=1}^{j_\ell}\alpha_{\sum_{t=1}^{\ell-1}j_t+r_\ell}
\Big).
\end{equation}

Thus using (\ref{eq-u=l}) we have
\begin{equation}\label{eq-8-14}
P(n;\alpha_1,\alpha_2,\cdots,\alpha_d)\simeq_d\sum
\limits_{\ell=1}^d \sum \limits_{j_1,\cdots j_\ell\in
\mathbb{N}\atop j_1+\cdots+j_\ell=d}(-1)^{\ell-1}
U(n;j_1,j_2,\cdots,j_{\ell}).
\end{equation}

Next using Lemma \ref{lem8-13}(1), for any $j_1,j_2,\cdots,j_\ell\in
\mathbb{N}$ with $j_1+j_2,\cdots+j_{\ell}\le d-1$, we have
$U(n;j_1,\cdots,j_\ell)f_{1+\sum_{t=1}^{\ell}
j_t}^{d-\sum_{t=1}^\ell j_t}(n)$ is equal to
\begin{align*}
&U(n;j_1,\cdots,j_\ell) \lceil u_{1+\sum_{t=1}^{\ell}j_t}^{d-\sum_{t=1}^\ell j_t} (n) \rceil\\
&\simeq_d \Big( U(n;j_1,\cdots,j_\ell)-\lceil
U(n;j_1,\cdots,j_\ell)\rceil \Big)
u_{1+\sum_{t=1}^{\ell} j_t}^{d-\sum_{t=1}^\ell j_t} (n)\\
&=\Big( U(n;j_1,\cdots,j_\ell)-\lceil U(n;j_1,\cdots,j_\ell)\rceil \Big) \times \\
&\hskip3.0cm \Big( x_{1+\sum_{t=1}^{\ell} j_t}^{d-\sum_{t=1}^\ell
j_t}(n)-\sum \limits_{j_{\ell+1}=1}^{d-(\sum_{t=1}^{\ell} j_t)-1}
x_{1+\sum_{t=1}^{\ell} j_t}^{j_{\ell+1}}(n) f_{1+\sum_{t=1}^{\ell+1}
j_t}^{d-\sum_{t=1}^{\ell+1} j_t} (n) \Big)
\end{align*}
which is equal to
\begin{align*}
&\Big( U(n;j_1,j_2,\cdots,j_\ell)-\lceil
U(n;j_1,j_2,\cdots,j_\ell)\rceil \Big) \times\\ &\hskip0.6cm \Bigg(
\tbinom{n}{d-\sum \limits_{t=1}^\ell j_t} \prod
\limits_{r_{\ell+1}=1}^{d-\sum \limits_{t=1}^{\ell} j_t}
\alpha_{\sum \limits_{t=1}^\ell
j_t+r_{\ell+1}}-\sum_{j_{\ell+1}=1}^{d-\sum \limits_{t=1}^{\ell+1}
j_t-1} \tbinom{n}{j_{\ell+1}} \prod
\limits_{r_{\ell+1}=1}^{j_{\ell+1}}\alpha_{\sum \limits_{t=1}^\ell
j_t+r_{\ell+1}} f_{1+\sum \limits_{t=1}^{\ell+1}j_t}^{d-\sum
\limits_{t=1}^{\ell+1}j_t}(n) \Bigg)
\end{align*}
which is
\begin{align*}
&\simeq_d \Big( U(n;j_1,j_2,\cdots,j_\ell)-\lceil U(n;j_1,j_2,\cdots,j_\ell)\rceil \Big) \times \\
&\hskip0.6cm \Bigg( \frac{n^{d-\sum \limits_{t=1}^\ell j_t}}{(d-\sum
\limits_{t=1}^\ell j_t)!} \prod \limits_{r_{\ell+1}=1}^{d-\sum
\limits_{t=1}^{\ell} j_t} \alpha_{\sum \limits_{t=1}^\ell
j_t+r_{\ell+1}}-\sum_{j_{\ell+1}=1}^{d-\sum \limits_{t=1}^{\ell+1}
j_t-1} \frac{n^{j_{\ell+1}}}{j_{\ell+1}!} \prod
\limits_{r_{\ell+1}=1}^{j_{\ell+1}}\alpha_{\sum \limits_{t=1}^\ell
j_t+r_{\ell+1}} f_{1+\sum \limits_{t=1}^{\ell+1}j_t}^{d-\sum
\limits_{t=1}^{\ell+1}j_t}(n)
\Bigg)\\
&=U(n;j_1,\cdots,j_\ell,d-\sum \limits_{t=1}^\ell
j_t)-\sum_{j_{\ell+1}=1}^{d-\sum \limits_{t=1}^{\ell+1} j_t-1}
U(n;j_1,\cdots,j_\ell,j_{\ell+1})f_{1+\sum
\limits_{t=1}^{\ell+1}j_t}^{d-\sum \limits_{t=1}^{\ell+1}j_t}(n).
\end{align*}
Using the fact and Lemma \ref{lem8-13}(1), we have
\begin{align*}
z_1^d(n)&\simeq_d u_1^d(n)=x_1^d(n)- \sum
\limits_{j_1=1}^{d-1}x_1^{j_1}(n)f_{1+j_1}^{d-j_1}(n) \\
&=\tbinom{n}{d}\alpha_1 \alpha_2\cdots \alpha_d- \sum
\limits_{j_1=1}^{d-1}\tbinom{n}{j_1}\alpha_1 \alpha_2\cdots \alpha_{j_1} f_{1+j_1}^{d-j_1}(n) \\
&\simeq_d U(n;d)-\sum
\limits_{j_1=1}^{d-1}U(n;j_1) f_{1+j_1}^{d-j_1}(n) \\
&\simeq_d U(n;d)-\Big( \sum
\limits_{j_1=1}^{d-1}(U(n;j_1,d-j_1)-\sum \limits_{j_2=1}^{d-j_1-1}
U(n;j_1,j_2) f_{1+j_1+j_2}^{d-(j_1+j_2)}(n)) \Big).
\end{align*}

Continuing this argument we obtain
\begin{align*}
z_1^d(n)&\simeq_d \sum \limits_{\ell=1}^d \sum
\limits_{j_1,\cdots,j_\ell\in \mathbb{N}\atop j_1+\cdots+j_{\ell}}
(-1)^{\ell-1} U(n;j_1,\cdots,j_\ell).
\end{align*}

Combining this with \eqref{eq-8-14}, we have proved the Claim.

\section{Proof of Theorem C}
\label{section-proof3}

In this section we will prove Theorem C. That is, we will show that
for $d\in\N$ and $F\in \F_{GP_d}$, there exist a minimal $d$-step
nilsystem $(X,T)$ and a nonempty open set $U$ such that
$$F\supset \{n\in\Z: U\cap T^{-n}U\cap \ldots\cap T^{-dn}U\neq \emptyset \}.$$


Let us explain the idea of the proof of Theorem C. Put
$\NN_d=\{B\subseteq \Z:$ there are a minimal $d$-step nilsystem
$(X,T)$ and an open non-empty set $U$ of $X$ with $B\supset
\{n\in\Z: \bigcap_{i=0}^dT^{-in}U\neq \emptyset\}\}$. Similar to the
proof of Theorem B(2) we first show that
$\mathcal{F}_{GP_d}\subseteq \NN_d$ if and only if $\{ n\in
\mathbb{Z}: ||P(n;\alpha_1,\cdots,\alpha_d)||<\epsilon\}\in \NN_d$
for any $\alpha_1,\cdots,\alpha_d\in \mathbb{R}$ and $\epsilon>0$.
We choose $(X,T)$ as the closure of the orbit of $\Gamma$ in
$\G_d/\Gamma$ (the nilrotation is induced by a matrix $A\in \G_d$),
define $U\subset X$ depending on a given $\ep>0$, put $S=\{n\in\Z:
\bigcap_{i=0}^dT^{-in}U\neq \emptyset\}$; and consider the most
right-corner entry $z_1^d(m)$ in $A^{nm}BC_m$ with $B\in\G_d$ and
$C_m\in\Gamma$ for a given $n\in S$ with $1\le m\le d$. We finish
the proof by showing $S\subset \{ n\in \mathbb{Z}:
||P(n;\alpha_1,\cdots,\alpha_d)||<\epsilon\}$ which implies that $\{
n\in \mathbb{Z}: ||P(n;\alpha_1,\cdots,\alpha_d)||<\epsilon\}\in
\NN_d$.

\subsection{The ordinary polynomial case}

To illustrate the idea of the proof of Theorem C, we first consider
the situation when the generalized polynomials are the ordinary
ones. That is, we want to explain if $p(n)$ is a polynomial of
degree $d$ with $p(0)=0$ and $\ep>0$, how we can find a $d$-step
nilsystem $(X,T)$, and a nonempty open set $U\subset X$ such that
\begin{equation}\label{polynomial}
\{n\in \Z: p(n)\ (\text{mod}\ \Z)\in (-\ep,\ep)\}\supset \{n\in\Z: U
\cap T^{-n}U \cap \ldots\cap T^{-dn}U\neq \emptyset \}.
\end{equation}
To do this define $T_{\alpha,d}:\T^d\lra \T^d$ by
$$T_{\alpha,d}(\theta_1,\theta_2,\ldots,\theta_d)=(\theta_1+\alpha, \theta_2+
\theta_1,\theta_3+ \theta_2, \ldots,\theta_d+ \theta_{d-1}),$$ where
$\alpha\in \mathbb{R}$. A simple computation yields that
\begin{align}\label{computation}
T_{\alpha,d}^n(\theta_1,\ldots,\theta_d)=(\theta_1+n\alpha,
n\theta_1+\theta_2+ \frac{1}{2}n(n-1)\alpha,\ldots,
\sum_{i=0}^d\tbinom{n}{d-i} \theta_i),
\end{align} where
$\theta_0=\alpha$, $n\in \mathbb{Z}$ and $\tbinom{n}{0}=1$,
$\tbinom{n}{i}:=\frac{\prod_{j=0}^{i-1} (n-j)}{i!}$ for
$i=1,2,\cdots,d$.

\medskip
We now prove (\ref{polynomial}) by induction. The case when $d=1$ is
easy, and we assume that for each polynomial of degree $\le d-1$
(\ref{polynomial}) holds. Now let $p(n)=\sum_{i=1}^d\alpha_in^i$
with $\alpha_i\in\R$. By induction for each $1\le i\le d-1$ there is
an $i$-step nilsystem $(X_i,T_i)$ and an open non-empty subset $U_i$
of $X_i$ such that
$$\{n\in\Z: \alpha_i n^i \ ({\rm mod}\ \Z)\in (-\frac{\ep}{d},\frac{\ep}{d})\}\supset
\{n\in\Z: U_i\cap T_i^{-n}U_i\cap \ldots\cap T_i^{-dn}U_i\neq
\emptyset \}.$$

By the Vandermonde's formula, we know
$$ \left(
  \begin{array}{ccccc}
    1 & 2 & 3 & \ldots & d \\
    1 & 2^2 & 3^2& \ldots & d^2 \\
    \vdots & \vdots & \vdots & \vdots & \vdots \\
    1 & 2^{d-1} & 3^{d-1} & \ldots & d^{d-1} \\
    1 & 2^d & 3^d & \ldots & d^d \\
  \end{array}
\right)$$ is a non-singular matrix. Hence there are integers $\ll
_1, \ll_2,\ldots, \ll_d$ and $\ll \in \N$ such that the following
equation holds:
$$
\left(
  \begin{array}{ccccc}
    1 & 2 & 3 & \ldots & d \\
    1 & 2^2 & 3^2& \ldots & d^2 \\
    \vdots & \vdots & \vdots & \vdots & \vdots \\
    1 & 2^{d-1} & 3^{d-1} & \ldots & d^{d-1} \\
    1 & 2^d & 3^d & \ldots & d^d \\
  \end{array}
\right)\left(
         \begin{array}{c}
           \ll_1 \\
            \ll_2 \\
           \vdots \\
            \ll_{d-1} \\
            \ll_d \\
         \end{array}
       \right)=\left(
                 \begin{array}{c}
                   0 \\
                   0 \\
                   \vdots \\
                   0 \\
                   \ll \\
                 \end{array}
               \right)
$$
That is,
\begin{equation}\label{a1}
\begin{split}
      & \sum_{m=1}^d \ll_m m^j=\ll_1+\ll_22^j+\ldots+ \ll_dd^j=0, \ 1\le j\le d-1;\\
       & \sum_{m=1}^d \ll_mm^d=\ll_1+\ll_22^d+\ldots+\ll_d d^d=\ll.
\end{split}
\end{equation}
Now let $T_d=T_{\frac{\alpha_d}{\lambda},d}$ and $Y_d=\T^d$. Let
$K_d=d!\sum_{i=1}^d|\lambda_i|$, $\ep_1>0$ with $K_d\ep_1<\ep/d$ and
$U_d=(-\ep_1,\ep_1)^d$.

It is easy to see that if $n\in \{n\in\Z: U_d\cap T_d^{-n}U_d\cap
\ldots\cap T_d^{-dn}U_d\neq \emptyset \}$ then we know that there is
$(\theta_1,\ldots,\theta_d)\in U_d$ such that
$T_d^{in}(\theta_1,\ldots,\theta_d)\in U_d$ for each $1\le i\le d$.
Thus, by (\ref{computation}) considering the last coordinate we ge
that {\footnotesize
\begin{align*}
\tbinom{n}{d}\theta_0+\tbinom{n}{d-1}\theta_1+\ldots+\tbinom{n}{0}\theta_d
\ (\text{\rm mod} \ \Z)&\in
(-\ep_1,\ep_1)\\
\tbinom{2n}{d}\theta_0+\tbinom{2n}{d-1}\theta_1+\ldots+\tbinom{2n}{0}\theta_d
\ (\text{\rm mod} \ \Z)&\in
(-\ep_1,\ep_1)\\
\ldots \ldots \ldots \ldots \ \ \ \ & \ldots\\
\tbinom{dn}{d}\theta_0+\tbinom{dn}{d-1}\theta_1+\ldots+\tbinom{dn}{0}\theta_d
\ (\text{\rm mod} \ \Z)&\in (-\ep_1,\ep_1),
\end{align*}}
where $\theta_0=\frac{\alpha_d}{\lambda}$. Multiplying
$\tbinom{in}{d}\theta_0+\tbinom{in}{d-1}\theta_1+\ldots+\tbinom{in}{0}\theta_d$
by $\lambda_id!$ and summing over $i=1,\ldots,d$ we get that
$$\sum_{j=1}^d\lambda_jd!\sum_{i=0}^d\tbinom{jn}{d-i}\theta_i=\alpha_d n^d
 \ (\text{\rm mod} \ \Z)\in (-K_d\ep_1,K_d\ep_1)\subset (-\ep/d,\ep/d).$$

Choose $x_i\in U_i$ for $1\le i\le d$. Let
$x=(x_1,x_2,\ldots,x_d)\in X_1\times \ldots\times X_d$ and $X$ be
the orbit $x$ under $T=T_1\times T_2\ldots\times T_d$. Then $(X,T)$
is a $d$-step nilsystem. If we let $U=(U_1\times U_2\times
\ldots\times U_d)\cap X$, then we have (\ref{polynomial}).

\medskip
By the property of nilsystems and the discussion above it is easy to
see
\begin{rem}\label{corr} Let $k\in\N$, $q_i(x)$ be a polynomial of degree $d$ with
$q_i(0)=0$ and $\ep_i>0$ for $1\le i\le k$. Then there are a
$d$-step nilsystem $(X,T,\mu)$ and $B\subset X$ with $\mu(B)>0$ such
that
$$\bigcap_{i=1}^k\{n\in \Z: ||q_i(n)||<\ep_i \}\supset \{n\in\Z: \mu(B\cap
T^{-n}B\cap \ldots\cap T^{-dn}B)>0\}$$
\end{rem}


\subsection{Some preparation}

For $d\in \N$, define
\begin{eqnarray*}
\NN_d&=&\{B\subseteq \Z: \text{there are a minimal $d$-step
nilsystem $(X,T)$ and an open} \\ & & \text{non-empty set $U$ of $X$
with $B\supset \{n\in\Z: U\cap T^{-n}U\cap \ldots\cap T^{-dn}U\neq
\emptyset \}$.} \}
\end{eqnarray*}

Hence Theorem C is equivalent to $$\F_{GP_d}\subseteq \NN_d.$$

\begin{lem}\label{lem-a1}
For each $d\in \N$, $\NN_d$ is a filter.
\end{lem}

\begin{proof}
Let $B_1, B_2\in \NN_d$. To show $\NN_d$ is a filter, it suffices to
show $B_1\cap B_2\in \NN_d$. By definition, there exist minimal
$d$-step nilsystems $(X_i,T_i), i=1,2$, and a nonempty open set
$U_i$ such that
$$B_i\supset \{n\in\Z: U_i\cap T_i^{-n}U_i\cap
\ldots\cap T_i^{-dn}U_i\neq \emptyset \}.$$ Taking any minimal point
$x=(x_1,x_2)\in X_1\times X_2$, let $X=\overline{{\rm orb}(x,T)}$,
where $T=T_1\times T_2$. Note that $(X,T)$ is also a minimal
$d$-step nilsystem.

Since $(X_i, T_i), i=1,2$, are minimal, there are $k_i\in \N$ such
that $x_i\in T_i^{-k_i}U_i$, $i=1,2$. Let $U=(T_1^{-k_1}U_1\times
T_2^{-k_2}U_2 )\cap X$, then $U$ is an open set of $X$. Note that
{\footnotesize
\begin{equation*}
\begin{split}&\hskip0.6cm \{n\in\Z: U\cap T^{-n}U\cap \ldots\cap T^{-dn}U\neq \emptyset \}
       \\&=\bigcap_{i=1,2}\{n\in\Z: T_i^{-k_i}U_i\cap T_i^{-(k_i+n)}U_i\cap \ldots
       \cap T_i^{-(k_i+dn)}U_i\neq \emptyset \}\\
       &=\bigcap_{i=1,2} \{n\in\Z: U_i\cap T_i^{-n}U_i\cap \ldots\cap T_i^{-dn}U_i\neq \emptyset \}
\end{split}
\end{equation*}}
Hence $$B_1\cap B_2\supset \{n\in\Z: U\cap T^{-n}U\cap \ldots\cap
T^{-dn}U\neq \emptyset \}.$$ That is, $B_1\cap B_2\in \NN_d$ and
$\NN_d $ is a filter.
\end{proof}

\begin{de}
For $r\in \N$, define
\begin{equation*}
    \widehat{GP}_r=\{p(n)\in GP_r: \{n\in \Z: p(n) \ ({\rm mod}\ \Z) \in (-\ep,\ep)\}\in \NN_r, \forall
    \ep>0\}.
\end{equation*}
\end{de}

\begin{rem}\label{rem-a1}
It is clear that for $p(n)\in GP_r$, $p(n)\in \widehat{GP}_r$ if and
only if $-p(n)\in \widehat{GP}_r$. Since $\NN_r$ is a filter, if
$p_1(n),p_2(n),\cdots,p_k(n)\in \widehat{GP}_r$ then
$$p_1(n)+p_2(n)+\cdots+p_k(n)\in \widehat{GP}_r.$$ Moreover by the definition
of $\widehat{GP}_r$, we know that $\mathcal{F}_{GP_r}\subset \NN_r$
if and only if $\widehat{GP}_r=GP_r$.
\end{rem}

Since we will use induction to show Theorem C, thus we need to
obtain some results under the following assumption, that is for some
$d\ge 2$,
\begin{equation}\label{induction-d-1}
\F_{GP_{d-1}}\subseteq \NN_{d-1}.
\end{equation}

\begin{lem}\label{lem-a2}
Let $p(n),q(n)\in GP_d$ with $p(n)\simeq_d q(n)$. Under the
assumption \eqref{induction-d-1}, $p(n)\in \widehat{GP}_d$ if and
only if $q(n)\in \widehat{GP}_d$.
\end{lem}

\begin{proof}
It follows from Lemma \ref{lem-8-11}, $\NN_d$ being a filter and
$\F_{GP_{d-1}}\subseteq \NN_{d-1}\subseteq \NN_d$.
\end{proof}

\begin{thm} \label{thm-a1}
Under the assumption \eqref{induction-d-1}, the following properties
are equivalent:
\begin{enumerate}
\item  $\mathcal{F}_{GP_d}\subseteq \NN_d$.

\item $P(n;\alpha_1,\alpha_2,\cdots,\alpha_d)\in \widehat{GP}_d$ for any $\alpha_1,\alpha_2,\cdots,\alpha_d\in \mathbb{R}$, that is
$$\{ n\in \mathbb{Z}: P(n;\alpha_1,\alpha_2,\cdots,\alpha_d) \ (\text{\rm mod} \ \Z) \in (-\epsilon, \epsilon)\}\in \NN_d$$ for any
$\alpha_1,\alpha_2,\cdots,\alpha_d\in \mathbb{R}$ and $\epsilon>0$.

\item $\text{SGP}_d\subset \widehat{GP}_d$.
\end{enumerate}
\end{thm}

\begin{proof}
The proof is similar to that of Theorem \ref{lem-trans}.
\end{proof}

\subsection{Proofs of Theorem C}

Now we prove $\F_{GP_d}\subseteq \NN_d$ by induction on $d$. When
$d=1$, since $\F_{GP_1}=\F_{SGP_1}$ and $\NN_d$ is a filer, it is
sufficient to show that: for any $p(n)=an\in SGP_1$ and $\ep>0$, we
have
\begin{equation}\label{}
    \{n\in \Z: p(n) \ ({\rm mod}\ \Z) \in (-\ep, \ep)\}\in \NN_1.
\end{equation}
This is easy to be verified.

Now we assume that for $d\ge 2$, $\F_{GP_{d-1}}\subseteq \NN_{d-1}$,
i.e. (\ref{induction-d-1}) holds.
Then it follows from Theorem \ref{thm-a1} that under the assumption
\eqref{induction-d-1},  to show $\mathcal{F}_{GP_d}\subseteq \NN_d$,
it is sufficient to show that
\begin{equation*}
    P(n;\b_1,\b_2,\ldots,\b_d)\in \widehat{GP}_d,
\end{equation*} for any $\beta_1,\beta_2,\cdots,\beta_d\in \mathbb{R}$.

Fix $\b_1,\b_2,\ldots,\b_d\in \R$. We divide the remainder of the
proof into two steps.

\medskip
\noindent{\bf Step 1.} We are going to show
\begin{equation*}
P(n;\b_1,\b_2,\cdots,\b_d)\simeq_d \sum \limits_{\ell=1}^d \sum
\limits_{j_1,\cdots j_\ell\in \mathbb{N}\atop
j_1+\cdots+j_\ell=d}(-1)^{\ell-1} \ll U(n;j_1,j_2,\cdots,j_{\ell}),
\end{equation*}
where as in the proof of Theorem B, we define
\begin{equation}\label{}
U(n;j_1)=\frac{n^{j_1}}{j_1!}\prod_{r=1}^{j_1}\alpha_r, \ 1\le
j_1\le d.
\end{equation}
And inductively for $\ell=2,3,\cdots,d$ define
\begin{align*}
U(n;j_1,j_2,\cdots,j_\ell) &=(U(n;j_1,\cdots,j_{\ell-1})-\lceil
U(n;j_1,\cdots,j_{\ell-1})\rceil )
\frac{n^{j_\ell}}{j_\ell!}\prod_{r=1}^{j_\ell}
\alpha_{\sum_{t=1}^{\ell-1}j_t+r}\\
&=(U(n;j_1,\cdots,j_{\ell-1})-\lceil
U(n;j_1,\cdots,j_{\ell-1})\rceil
)L(\frac{n^{j_\ell}}{j_\ell!}\prod_{r_{\ell}=1}^{j_\ell}
\alpha_{\sum_{t=1}^{\ell-1}j_t+r_{\ell}})
\end{align*}
for $j_1,j_2,\cdots,j_\ell\ge 1$ and $j_1+\cdots+j_\ell\le d$ (see
\eqref{eq-de-L} for the definition of $L$).

\medskip

In fact, let $\ll _1, \ll_2,\ldots, \ll_d\in \mathbb{Z}$ and $\ll
\in \N$ satisfying \eqref{a1}. Put
$$\a_1=\b_1/\ll, \a_2=\b_2, \a_3=\b_3,\ldots, \a_d=\b_d.$$ Then
\begin{equation*}
    P(n;\b_1,\b_2,\ldots,\b_d)=\ll P(n;\a_1,\a_2,\ldots,\a_d).
\end{equation*}

Note that in proof of Theorem B we have
\begin{equation}\label{}
P(n;\alpha_1,\alpha_2,\cdots,\alpha_d)\simeq_d \sum
\limits_{\ell=1}^d \sum \limits_{j_1,\cdots j_\ell\in
\mathbb{N}\atop j_1+\cdots+j_\ell=d}(-1)^{\ell-1}
U(n;j_1,j_2,\cdots,j_{\ell}).
\end{equation}
Since $\ll$ is an integer, we have
\begin{equation*}
\ll P(n;\alpha_1,\alpha_2,\cdots,\alpha_d)\simeq_d \sum
\limits_{\ell=1}^d \sum \limits_{j_1,\cdots j_\ell\in
\mathbb{N}\atop j_1+\cdots+j_\ell=d}(-1)^{\ell-1} \ll
U(n;j_1,j_2,\cdots,j_{\ell}).
\end{equation*}
That is,
\begin{equation*}
P(n;\b_1,\b_2,\cdots,\b_d)\simeq_d \sum \limits_{\ell=1}^d \sum
\limits_{j_1,\cdots j_\ell\in \mathbb{N}\atop
j_1+\cdots+j_\ell=d}(-1)^{\ell-1} \ll U(n;j_1,j_2,\cdots,j_{\ell}).
\end{equation*}
Hence, by Lemma \ref{lem-a2}, to show $P(n;\b_1,\b_2,\cdots,\b_d)\in
\widehat{GP}_d$, it suffices to show
\begin{equation}\label{a2}
\sum \limits_{\ell=1}^d \sum \limits_{j_1,\cdots j_\ell\in
\mathbb{N}\atop j_1+\cdots+j_\ell=d}(-1)^{\ell-1} \ll
U(n;j_1,j_2,\cdots,j_{\ell})\in \widehat{GP}_d.
\end{equation}

Now choose ${\bf x}=(x_i^k)_{1\le k\le d, 1\le i\le d-k+1}\in
\mathbb{R}^{d(d+1)/2}$ with $x_i^1=\alpha_i$ for $i=1,2,\cdots,d$
and $x_i^k=0$ for $2\le k\le d$ and $1\le i\le d-k+1$. Let
{\footnotesize
$$
A=\M({\bf x})=\left(
  \begin{array}{cccccccc}
    1 & \alpha_1 & 0                &\ldots & 0 &0\\
    0 & 1     & \alpha_2            &\ldots &  0&0\\
    \vdots & \vdots &   \vdots & \vdots &\vdots\\
    0 & 0    &0                    & \ldots & \alpha_{d-1} &0 \\
    0 & 0    &0                    & \ldots & 1 &\alpha_d \\
    0 & 0    &0                   & \ldots & 0  & 1
  \end{array}
\right)
$$}

For $n\in \mathbb{N}$, if ${\bf x}(n)=(x_i^k(n))_{1\le k\le d, 1\le
i\le d-k+1}\in \mathbb{R}^{d(d+1)/2}$ satisfies $\M({\bf
x}(n))=A^n$, then by Lemma \ref{lem-8-ite} and Remark
\ref{rem-8-ite},
\begin{equation}\label{a3}
x_i^k(n)=\tbinom{n}{k}\alpha_i \alpha_{i+1}\cdots \alpha_{i+k-1}
\end{equation}
for $1\le k\le d$ and $1\le i\le d-k+1$.

\medskip

Let $X=\G_d/\Gamma$ be endowed with the metric $\rho$ in Lemma
\ref{matix-metric} and $T$ be the nilrotation induced by $A\in\G_d$,
i.e. $B\Gamma\mapsto AB\Gamma$ for $B\in\G_d$. Since $\G_d$ is a
$d$-step nilpotent Lie group and $\Gamma$ is a uniform subgroup of
$\G_d$, $(X,T)$ is a $d$-step nilsystem. Let $x_0=\Gamma\in X$ and
$Z$ be the closure of the orbit $\text{orb}(x_0,T)$ of $e$ in $X$.
Then $(Z,T)$ is a minimal $d$-step nilsystem. We consider $\rho$ as
a metric on $Z$.

\medskip
\noindent{\bf Step 2.} For any $\ep>0$, we are going to show there
is a nonempty open set $U$ of $Z$ such that
\begin{equation}\label{a0}
\begin{split}
& \ \{n\in \Z: \sum \limits_{\ell=1}^d \sum \limits_{j_1,\cdots
j_\ell\in \mathbb{N}\atop j_1+\cdots+j_\ell=d}(-1)^{\ell-1} \ll
U(n;j_1,j_2,\cdots,j_{\ell}) \ ({\rm mod}\ \Z)\in (-\ep,\ep)\}\\
&\supseteq \{n\in\Z: U\cap T^{-n}U\cap \ldots\cap T^{-dn}U\neq
\emptyset\}.
\end{split}
\end{equation}
That means $\sum \limits_{\ell=1}^d \sum \limits_{j_1,\cdots
j_\ell\in \mathbb{N}\atop j_1+\cdots+j_\ell=d}(-1)^{\ell-1} \ll
U(n;j_1,j_2,\cdots,j_{\ell})\in \widehat{GP}_d$.

\medskip

Fix an $\ep>0$. Take $\displaystyle
\ep_2=\min\Big\{\frac{\ep}{2K(\sum_{i=0}^{d-1}d^i)}, \frac
14\Big\}$, where $\displaystyle
K=\sum_{m=1}^d|\lambda_m|\Big(\sum_{t=0}^dm^t\Big)$, and let
$\ep_1>0$ be small enough such that $e^{\ep_1+\ep_1^2+\ldots
+\ep_1^d}-1<\ep_2$. Let
\begin{equation*}
U=\{z\in Z: \rho(z,x_0)<\ep_1\}=\{c\Gamma\in Z: \rho(c\Gamma,
\Gamma)<\ep_1\}.
\end{equation*}
and let
\begin{equation*}
S=\{n\in \Z: U\cap T^{-n}U\cap \ldots \cap T^{-dn}U\neq \emptyset
\}.
\end{equation*}

 Now we show that
\begin{equation*}
   S\subseteq \Big\{n\in \Z: \sum \limits_{\ell=1}^d \sum \limits_{j_1,\cdots
j_\ell\in \mathbb{N}\atop j_1+\cdots+j_\ell=d}(-1)^{\ell-1} \ll
U(n;j_1,j_2,\cdots,j_{\ell}) \ ({\rm mod}\ \Z)\in (-\ep,\ep)\Big\}
\end{equation*}
Let $n\in S$. Then $U\cap T^{-n}U\cap \ldots \cap T^{-dn}U\neq
\emptyset$. Hence there is some $B\in \G_d$ with
\begin{equation*}
    B\Gamma \in U\cap T^{-n}U\cap \ldots \cap T^{-dn}U.
\end{equation*}
Thus
\begin{equation*}
\rho(A^{mn}B\Gamma, \Gamma)<\ep_1, \ m=0,1,2,\ldots, d-1.
\end{equation*}
Since $\rho$ is right translation invariant, we may assume that
$\rho(B, I)<\ep_1$, where $I$ is the $(d+1)\times (d+1)$ identity
matrix.

For each $m\in \{1,2,\ldots, d\}$, since $\rho(A^{mn}B\Gamma,
\Gamma)<\ep_1$ there is some $C_m\in \Gamma$ such that
\begin{equation}\label{a5-1}
    \rho(A^{mn}BC_m,I)<\ep_1.
\end{equation}
Let $A^{mn}BC_m=\M({\bf z}(m))$, where ${\bf z}(m)=(z_i^k(m))_{1\le
k\le d, 1\le i\le d-k+1}\in \mathbb{R}^{d(d+1)/2}$. Then From
\eqref{a5-1}, we have $||A^{mn}BC_m-I||<\ep_2$ by Lemma
\ref{lem-matrix}, thus
$$|z_i^k(m)|<\ep_2, \quad 1\le k\le d, 1\le
i\le d-k+1. $$

On the one hand, since $|z_1^d(m)|<\ep_2$, we have
\begin{equation}\label{a8}
    \sum_{m=1}^d \ll_m z_1^d(m)\in (-K\ep_2, K\ep_2).
\end{equation}

On the other hand, we have

\medskip
\begin{align}\label{longproof}
\begin{split}
\sum_{m=1}^d\ll_mz_1^d(m) \approx&
\Big(\sum_{l=1}^d(-1)^{l-1}\sum_{\substack{j_1,j_2,\ldots,j_l\in \N\\
j_1+j_2+\ldots +j_l=d}} \ll U(n;j_1,j_2,\ldots,j_l)\Big)\\
&+\vartriangle\Big((d+d^2+\ldots+d^{d-1})(2K\ep_2)\Big).
\end{split}
\end{align}
Note that for $a,b\in \R$ and $\d>0$, $a\thickapprox b+
\vartriangle(\d)$ means that $a-b \ ({\rm mod}\ \Z)\in (- \d,\d)$.

\medskip
Since the proof of (\ref{longproof}) is long, we put it after
Theorem C. Now we continue the proof. By (\ref{longproof}) and
(\ref{a8}), we have {\footnotesize
\begin{equation*}
\begin{split}
\sum_{l=1}^d(-1)^{l-1}\sum_{\substack{j_1,j_2,\ldots,j_l\in \N\\
j_1+j_2+\ldots +j_l=d}}\ll U(n;j_1,j_2,\ldots,j_l) \ ({\rm mod}\
\Z)&\in\Big(-M(2K\ep_2),M(2K\ep_2)\Big)\subseteq (-\ep,\ep),
\end{split}
\end{equation*}}
where $M=1+d+\ldots+d^{d-1}$. This means that
\begin{equation*}
    n\in \Big\{q\in \Z:\sum_{l=1}^d\sum_{\substack{j_1,j_2,\ldots,j_l\in \N\\
j_1+j_2+\ldots +j_l=d}}(-1)^{l-1}\ll U(q;j_1,j_2,\ldots,j_l) \ ({\rm
mod}\ \Z)\in (-\ep,\ep)\Big \}.
\end{equation*}
Hence
\begin{equation*}
   S\subseteq  \Big\{q\in \Z:\sum_{l=1}^d\sum_{\substack{j_1,j_2,\ldots,j_l\in \N\\
j_1+j_2+\ldots +j_l=d}}(-1)^{l-1}\ll U(q;j_1,j_2,\ldots,j_l) \ ({\rm
mod}\ \Z)\in (-\ep,\ep)\Big \}.
\end{equation*}
Thus we have proved (\ref{a0}) which means $\sum \limits_{\ell=1}^d
\sum \limits_{j_1,\cdots j_\ell\in \mathbb{N}\atop
j_1+\cdots+j_\ell=d}(-1)^{\ell-1} \ll
U(n;j_1,j_2,\cdots,j_{\ell})\in \widehat{GP}_d$. The proof of
Theorem C is now finished.

\subsection{Proof of (\ref{longproof})} Since $\rho(B,I)<\ep_1$, by Lemma \ref{lem-matrix},
\begin{equation}\label{a4}
    ||B-I||<e^{\ep_1+\ep_1^2+\ldots
+\ep_1^d}-1<\ep_2<1/2.
\end{equation}

Denote $B=\M({\bf y})$, where ${\bf y}=(y_i^k)_{1\le k\le d, 1\le
i\le d-k+1}\in \mathbb{R}^{d(d+1)/2}$. From (\ref{a4}),
$$|y_i^k|<\ep_2, \quad 1\le k\le d,\ 1\le
i\le d-k+1. $$ For $m=1,2,\cdots,m$, Recall that $C_m\in \Gamma$
satisfies
\begin{equation}\label{a5}
    \rho(A^{mn}BC_m,I)<\ep_1.
\end{equation}
Denote $C_m=\M({\bf h}(m))$, where ${\bf h}(m)=(-h_i^k(m))_{1\le
k\le d, 1\le i\le d-k+1}\in \mathbb{Z}^{d(d+1)/2}$. From (\ref{a5}),
we have $$||A^{mn}BC_m-I||<\ep_2, \ m=1,2 ,\ldots, d.$$ Let
$A^{mn}B=\M({\bf w}(m))$, where ${\bf w}(m)=(w_i^k(m))_{1\le k\le d,
1\le i\le d-k+1}\in \mathbb{R}^{d(d+1)/2}$. Then
\begin{equation}\label{a6}
\begin{split}
w_i^k(m)&=x_i^k(mn)+\Big( \sum
\limits_{j=1}^{k-1}x_i^{j}(mn)y_{i+j}^{k-j} \Big)+y_i^k\\
&=\tbinom{mn}{k}\alpha_i \alpha_{i+1}\cdots
\alpha_{i+k-1}+\sum_{j=1}^{k-1}\tbinom{mn}{j}\alpha_i
\alpha_{i+1}\cdots \alpha_{i+j-1}y_{i+j}^{k-j}+y_i^k\\
&\triangleq \frac{(mn)^k}{k!}\a_i\ldots \a_{i+k-1}+\sum_{j=1}^{k-1}
m^j a_i^k(j)+a_i^k(0),
\end{split}
\end{equation}
where $m=1,2,\ldots, d$, $a_i^k(j)$ does not depend on $m$ and
$|a_i^k(0)|=|y_i^k|<\ep_2$.

Recall that ${\bf z}(m)=(z_i^k(m))_{1\le k\le d, 1\le i\le d-k+1}\in
\mathbb{R}^{d(d+1)/2}$ satisfies $A^{mn}BC_m=\M({\bf z}(m))$. Hence
\begin{equation}\label{addye}
z_i^k(m)= w_i^k(m)-\Big( \sum \limits_{j=1}^{k-1} w_i^{j}(m)
h_{i+j}^{k-j}(m) \Big) -h_i^k(m).
\end{equation}
From $||A^{mn}BC_m-I||<\ep_2$, we have
$$|z_i^k(m)|<\ep_2, \quad 1\le k\le d, 1\le
i\le d-k+1. $$ Note that $h_i^k(m)\in \Z$, and we have
$$h_i^k(m)=\Big \lceil w_i^k(m)- \sum
\limits_{j=1}^{k-1}w_i^{j}(m)h_{i+j}^{k-j}(m) \Big
\rceil.$$ Let
$$u_i^k(m)=w_i^k(m)- \sum
\limits_{j=1}^{k-1}w_i^{j}(m)h_{i+j}^{k-j}(m) .$$ Then
\begin{equation*}
    |u_i^k(m)-h_i^k(m)|=|z_i^k(m)|<\ep_2<1/2.
\end{equation*}

\medskip

Recall that for $a,b\in \R$ and $\d>0$, $a\thickapprox b+
\vartriangle(\d)$ means $a-b \ ({\rm mod}\ \Z)\in (- \d,\d)$.

\bigskip

\noindent {\bf Claim:} {\em Let $1\le r\le d-1$ and $v_r(0),
v_r(1),\ldots,v_r(r)\in \R$. Then for each $1\le r_1\le d-r-1$ and
$1\le j\le r_1+r$, there exist $v_{r,r_1}(j)\in \R$ such that
\begin{enumerate}
\item
\begin{equation*}
\begin{split}
&\sum_{m=1}^d  \ll_m \Big(\sum_{t=0}^r m^t
v_r(t)\Big)h_{1+r}^{d-r}(m) \thickapprox \ll(v_r(r)-\lceil
v_r(r)\rceil)\frac{n^{d-r}}{(d-r)!}\a_{1+r}\ldots
\a_d \\
&\hskip3.cm-\sum_{r_1=1}^{d-r-1}\sum_{m=1}^d\ll_m\Big(\sum_{t=1}^{r_1+r}m^tv_{r,r_1}(t)\Big)h_{1+r+r_1}^{d-r-r_1}(m)+
\vartriangle(2K\ep_2)
\end{split}
\end{equation*}

\item $\displaystyle v_{r,r_1}(r+r_1)=\Big( v_r-\lceil v_r\rceil\Big)\frac{n^{r_1}}{r_1!}
\a_{r+1}\ldots \a_{r+r_1}$ for all $1\le r_1\le d-r-1$.
\end{enumerate}}

\begin{proof}[Proof of Claim]
Since $|u_{1+r}^{d-r}(m)-h_{1+r}^{d-r}(m)|<\ep_2$, we have
\begin{equation*}
\begin{split}
\Big | \sum_{m=1}^d \ll_m\Big(\sum_{t=0}^rm^t(v_r(r)-\lceil
v_r(r)\rceil)\Big) \Big| &\le \sum_{m=1}^d
|\ll_m|\Big(\sum_{t=0}^rm^t\Big)\\
&\le \Big(\sum_{m=1}^d |\ll_m|\Big)\Big(\sum_{t=0}^rm^t\Big)=K.
\end{split}
\end{equation*}
Hence
\begin{equation}\label{a7}
\begin{split}&\sum_{m=1}^d \ll_m \Big(\sum_{t=0}^r
m^tv_r(t)\Big)h_{1+r}^{d-r}(m)\\ &\thickapprox  \sum_{m=1}^d \ll_m
\Big(\sum_{t=0}^r m^t (
v_r(t)-\lceil v_r(t)\rceil)\Big)h_{1+r}^{d-r}(m)\\
&\thickapprox  \sum_{m=1}^d \ll_m \Big(\sum_{t=0}^r m^t (
v_r(t)-\lceil v_r(t)\rceil
)\Big)u_{1+r}^{d-r}(m)+\vartriangle(K\ep_2).
\end{split}
\end{equation}

Then we have
\begin{equation*}
\begin{split}
&\sum_{m=1}^d \ll_m \Big(\sum_{t=0}^r m^t ( v_r(t)-\lceil
v_r(t)\rceil )\Big)u_{1+r}^{d-r}(m)\\
&\sum_{m=1}^d \ll_m \Big(\sum_{t=0}^r m^t ( v_r(t)-\lceil
v_r(t)\rceil)\Big)\Big(w_{1+r}^{d-r}(m)- \sum
\limits_{r_1=1}^{d-r-1}w_{1+r}^{r_1}(m)h_{1+r+r_1}^{d-r-r_1}(m)
\Big).
\end{split}
\end{equation*}

From (\ref{a6}) we have
\begin{equation*}
\begin{split}
&\sum_{m=1}^d \ll_m \Big(\sum_{t=0}^r m^t ( v_r(t)-\lceil
v_r(t)\rceil)\Big)w_{1+r}^{d-r}(m)\\ &= \sum_{m=1}^d \ll_m
\Big(\sum_{t=0}^r m^t ( v_r(t)-\lceil v_r(t)\rceil)\Big)
\Big(\frac{(mn)^{d-r}}{(d-r)!}\a_{1+r}\ldots
\a_{d}+\sum_{j=0}^{d-r-1} m^j a_{1+r}^{d-r}(j)\Big)\\
&= \sum_{m=1}^d \ll_m m^d \frac{n^{d-r}}{(d-r)!}\a_{1+r}\ldots
\a_{d} \Big( v_r(t)-\lceil v_r(t)\rceil \Big)\\
&\hskip4.cm+ \sum_{h=1}^{d-1}\sum_{m=1}^d \ll_m m^h \Bigg(\sum_{0\le
t\le r\atop{0\le j\le d-r-1\atop{t+j=h}}}( v_r(t)-\lceil
v_r(t)\rceil )a_{1+r}^{d-r}(j)\Bigg)\\
&\hskip6.cm+\sum_{m=1}^d \ll_m \Big( v_r(0)-\lceil
v_r(0)\rceil\Big)a_{1+r}^{d-r}(0),
\end{split}
\end{equation*}
and so $\sum_{m=1}^d \ll_m \Big(\sum_{t=0}^r m^t (
v_r(t)-\lceil v_r(t)\rceil)\Big)w_{1+r}^{d-r}(m)$ is equal to

\begin{equation*}
\begin{split}
&  \ll \frac{n^{d-r}}{(d-r)!}\a_{1+r}\ldots \a_{d} (
v_r(t)-\lceil v_r(t)\rceil)+(\sum_{m=1}^d \ll_m )
( v_r(0)-\lceil v_r(0)\rceil)y_{1+r}^{d-r}\\
&\approx  \ll \frac{n^{d-r}}{(d-r)!}\a_{1+r}\ldots \a_{d} (
v_r(t)-\lceil v_r(t)\rceil )+ \vartriangle (K\ep_2).
\end{split}
\end{equation*}
The last equation follows from
$$\Big |(\sum_{m=1}^d \ll_m)
( v_r(0)-\lceil v_r(0)\rceil)y_{1+r}^{d-r}\Big |\le
\sum_{m=1}^d |\ll_m|\ep_2<K\ep_2.$$ Then for $1\le r_1\le d-r-1$, by
(\ref{a6}), we have

\begin{equation*}
\begin{split}
& \sum_{m=1}^d\ll_m \Big(\sum_{t=0}^r m^t ( v_r(t)-\lceil
v_r(t)\rceil)\Big)w_{1+r}^{r_1}(m)h_{1+r+r_1}^{d-r-r_1}(m) \\
&= \sum_{m=1}^d \ll_m \Big(\sum_{t=0}^r m^t ( v_r(t)-\lceil
v_r(t)\rceil )\Big) \Big(\frac{(mn)^{r_1}}{(r_1)!}\a_{1+r}\ldots
\a_{r+r_1}+\sum_{j=0}^{r_1-1} m^j a_{1+r}^{r_1}(j)\Big)h_{1+r+r_1}^{d-r-r_1}(m)\\
&= \sum_{m=1}^d \ll_m \Big( m^{r+r_1}
\frac{n^{r_1}}{(r_1)!}\a_{1+r}\ldots \a_{r+r_1} ( v_r(t)-\lceil
v_r(t)\rceil)+ \\ &\hskip 5cm \sum_{h=0}^{r+r_1-1}m^h
\big(\sum_{0\le t\le r\atop{0\le j\le r_1-1\atop{t+j=h}}}\big(
v_r(t)-\lceil v_r(t)\rceil \big ) a_{1+r}^{r_1}(j)\big) \Big)
h_{1+r+r_1}^{d-r-r_1}(m).
\end{split}
\end{equation*}
Let
\begin{equation*}
    v_{r,r_1}(h)=\sum_{0\le t\le
r\atop{0\le j\le r_1-1\atop{t+j=h}}}\big( v_r(t)-\lceil v_r(t)\rceil
\big ) a_{1+r}^{r_1}(j), \ 0\le h\le r+r_1-1;
\end{equation*}
\begin{equation*}
    v_{r,r_1}(r+r_1)=\frac{(n)^{r_1}}{(r_1)!}\a_{1+r}\ldots \a_{r+r_1} \left(
v_r(t)-\lceil v_r(t)\rceil\right);
\end{equation*}
\begin{equation*}
    v_{r,r_1}(0)=\left(v_r(0)-\lceil
    v_r(0)\rceil\right)a_{1+r}^{r_1}(0)=\left(v_r(0)-\lceil
    v_r(0)\rceil\right)y_{1+r}^{r_1}(0).
\end{equation*}
It is easy to see that $|v_{r,r_1}(0)|<\ep_2$. Then
\begin{equation*}
\begin{split}
& \sum_{m=1}^d \ll_m \Big(\sum_{t=0}^r m^t ( v_r(t)-\lceil
v_r(t)\rceil)\Big)w_{1+r}^{r_1}(m)h_{1+r+r_1}^{d-r-r_1}(m) \\
&= \sum_{m=1}^d \ll_m \Big(\sum_{t=0}^{r+r_1} m^t v_{r,r_1}(t) \Big)
h_{1+r+r_1}^{d-r-r_1}(m).
\end{split}
\end{equation*}

To sum up, we have {\footnotesize
\begin{equation*}
\begin{split}
&\sum_{m=1}^d \ll_m \Big(\sum_{t=0}^r m^t \left( v_r(t)-\lceil
v_r(t)\rceil\right)\Big)u_{1+r}^{d-r}(m) \\
&\approx \ll \frac{(n)^{d-r}}{(d-r)!}\a_{1+r}\ldots \a_{d} (
v_r(t)-\lceil v_r(t)\rceil
)-\sum_{r_1=1}^{d-r-1}\Bigg(\sum_{m=1}^d\ll_m\Big(\sum_{t=0}^{r_1+r}m^tv_{r,r_1}(t)\Big)
h_{1+r+r_1}^{d-(r+r_1)}(m)\Bigg)\\
&\hskip11.5cm +\vartriangle(K\ep_2).
\end{split}
\end{equation*}}
Together with ( \ref{a7}), we have {\footnotesize
\begin{equation*}
\begin{split}
\sum_{m=1}^d \ll_m \Big(\sum_{t=0}^r m^tv_r(t)\Big)h_{1+r}^{d-r}(m)
&\approx \ll \frac{(n)^{d-r}}{(d-r)!}\a_{1+r}\ldots \a_{d} \Big(
v_r(t)-\lceil v_r(t)\rceil \Big)\\
&-\sum_{r_1=1}^{d-r-1}\Bigg(\sum_{m=1}^d\ll_m\Big(\sum_{t=0}^{r_1+r}m^tv_{r,r_1}(t)\Big)
h_{1+r+r_1}^{d-(r+r_1)}(m)\Bigg)+\vartriangle(2K\ep_2).
\end{split}
\end{equation*}}
The proof of the claim  is completed. \end{proof}

We will use the claim repeatedly. First using (\ref{addye}) we have
\begin{equation*}
    \sum_{m=1}^d\ll_mz_1^d(m)\approx \sum_{m=1}^d \ll_m
    \bigg (w_1^d(m)-\sum_{j_1=1}^{d-1} w_1^{j_1}(m) h_{1+j_1}^{d-j_1}(m)\bigg).
\end{equation*}
By (\ref{a6}), we have
\begin{equation*}
\begin{split}
\sum_{m=1}^d \ll_mw_1^d(m)& \approx \sum_{m=1}^d \ll_m m^d
\frac{n^d}{d!}\a_1\ldots \a_d+\sum_{m=1}^d \ll_my_1^d\\
& \approx \ll \frac{n^d}{d!}\a_1\ldots \a_d +\vartriangle(K\ep_2).
\end{split}
\end{equation*}
Using this, (\ref{a6}) and the claim, we have {\footnotesize
\begin{equation*}
\begin{split} \sum_{m=1}^d \ll_mz_1^d(m)&
\approx \ll \frac{n^d}{d!}\a_1\ldots
\a_d-\sum_{m=1}^d\ll_m\sum_{j_1=1}^{d-1}\Big(m^{j_1}\frac{n^{j_1}}{j_1!}\a_1\ldots
\a_{j_1}+\sum_{t=0}^{j_1-1}m^ta_1^{j_1}(t)\Big)h_{1+j_1}^{d-j_1}(m)+\vartriangle(K\ep_2)\\
&\approx \ll \frac{n^d}{d!}\a_1\ldots
\a_d-\left(\sum_{j_1=1}^{d-1}\ll
\frac{n^{d-j_1}}{(d-j_1)!}\a_{1+j_1}\ldots\a_d\Big(\frac{n^{j_1}}{j_1!}\a_1\ldots\a_{j_1}-
\lceil \frac{n^{j_1}}{j_1!}\a_1\ldots\a_{j_1} \rceil\Big)\right)\\
& +
\sum_{j_1=1}^{d-1}\sum_{j_2=1}^{d-j_1-1}\Bigg(\sum_{m=1}^d\ll_m\bigg(m^{j_1+j_2}
\Big(\frac{n^{j_1}}{j_1!}\a_1\ldots\a_{j_1}- \lceil
\frac{n^{j_1}}{j_1!}\a_1\ldots\a_{j_1}
\rceil\Big)\frac{n^{j_2}}{j_2!}\a_{1+j_1}\ldots
\a_{j_1+j_2}\\
& +
\sum_{t=0}^{j_1+j_2-1}m^tv_{j_1,j_2}(t)\bigg)h_{1+j_1+j_2}^{d-(j_1+j_2)}(m)\Bigg)+\vartriangle(\big(2(d-1)K+K\big)\ep_2).
\end{split}
\end{equation*}}
Note that here we use $v_{j_1}(t)=a_1^{j_1}(t), t=0,1,\ldots, j_1-1$
and $v_{j_1}(j_1)=\frac{n^{j_1}}{j_1!}\a_1\ldots\a_{j_1}$.

Recall the definition of $U(\cdot)$:
\begin{equation*}
    \frac{n^d}{d!}\a_1\ldots\a_d=U(n;d),
\end{equation*}
\begin{equation*}
    \Big(\frac{n^{j_1}}{j_1!}\a_1\ldots\a_{j_1}-
\lceil \frac{n^{j_1}}{j_1!}\a_1\ldots\a_{j_1}
\rceil\Big)\frac{n^{j_2}}{j_2!}\a_{1+j_1}\ldots
\a_{j_1+j_2}=U(n;j_1,j_2).
\end{equation*}
Substituting these in the above equation, we have {\footnotesize
\begin{equation*}
\begin{split}
\sum_{m=1}^d\ll_mz_1^d(m)& \approx \ll U(n;d)-\sum_{j_1=1}^{d-1}\ll
U(n; j_1,d-j_1)\\
&\hskip0.5cm
+\sum_{j_1=1}^{d-1}\sum_{j_2=1}^{d-j_1-1}\Bigg(\sum_{m=1}^d\ll_m\bigg(m^{j_1+j_2}U(n;j_1,j_2)
+\sum_{t=0}^{j_1+j_2-1}m^tv_{j_1,j_2}(t)\bigg)h_{1+j_1+j_2}^{d-(j_1+j_2)}(m)\Bigg)\\
&\hskip0.5cm +\vartriangle(2d K\ep_2)
\end{split}
\end{equation*}}
Using the claim again, we have: {\footnotesize
\begin{equation*}
\begin{split}
\sum_{m=1}^d\ll_mz_1^d(m)& \approx \ll U(n;d)-\sum_{j_1=1}^{d-1}\ll
U(n; j_1,j_2)+ \sum_{j_1=1}^{d-1}\sum_{j_2=1}^{d-j_1-1}\ll
U(n;j_1,j_2,d-j_1-j_2)
\\
&-\sum_{j_1=1}^{d-1}\sum_{j_2=1}^{d-j_1-1}\sum_{j_3=1}^{d-(j_1+j_2)-1}\Bigg(\sum_{m=1}^d\ll_m
\bigg(m^{j_1+j_2+j_3}U(n;j_1,j_2,j_3) +\\
& \sum_{t=0}^{j_1+j_2+j_3-1}
m^tv_{j_1,j_2,j_3}(t)\bigg)h_{1+j_1+j_2+j_3}^{d-(j_1+j_2+j_3)}(m)\Bigg)+\vartriangle(2dK\ep_2+2d^2K\ep_2).
\end{split}
\end{equation*}}

Inductively, we have $\sum_{m=1}^d\ll_mz_1^d(m)$ {\footnotesize
\begin{equation*}
\begin{split}
& \approx
\Big(\sum_{l=1}^d(-1)^{l-1}\sum_{\substack{j_1,j_2,\ldots,j_l\in \N\\
j_1+j_2+\ldots +j_l=d}}\ll U(n;j_1,j_2,\ldots,j_l)\Big)
+\vartriangle(2dK\ep_2+2d^2K\ep_2+\ldots+2d^{d-1}K\ep_2)\\
&\approx
\Big(\sum_{l=1}^d(-1)^{l-1}\sum_{\substack{j_1,j_2,\ldots,j_l\in \N\\
j_1+j_2+\ldots +j_l=d}} \ll U(n;j_1,j_2,\ldots,j_l)\Big)
+\vartriangle\Big((d+d^2+\ldots+d^{d-1})(2K\ep_2)\Big).
\end{split}
\end{equation*}}
The proof of (\ref{longproof}) is now finished.

\section{Applications}\label{section-8}

Our main results can be applied to get results in the theory of
dynamical systems. As the limitation of the length of the paper,
here we only state the results and the detailed proofs will appear
in a forthcoming paper by the same authors.


\subsection{$d$-step almost automorpy}
The notion of almost automorphy was first introduced by Bochner in
1955 in a work of differential geometry \cite{Bochner55, Bochner62}.
Veech showed that each almost automorphic minimal system is an
almost one-to-one extension of a compact metric abelian group
rotation \cite{V65}. Let $(X,T)$ be a minimal system and $d\in \N$.
$(X,T)$ is called a {\em $d$-step almost automorphic} system if it
is an almost one-to-one extension of a $d$-step nilsystem. Let
$\pi:X\lra Y$ be the almost one-to-one extension with $Y$ being a
$d$-step nilsystem. A point $x\in X$ is {\em $d$-step almost
automorphic} if $\pi^{-1}\pi(x)=\{x\}$.

\medskip

Using the main results of this paper, we show that $\F_{Poi_d}$,
$\F_{Bir_d}$ and $\F_{d,0}$ can be used to characterize $d$-step
almost automorphy, i.e. in some sense, $\F_{Poi_d}$ and $\F_{d,0}^*$
can not be distinguished ``dynamically''. Similar results can be
found in the next subsections.

\begin{thm}\cite{HSY}
Let $(X,T)$ be a minimal system, $x\in X$ and $d\in\N$. Then the
following statements are equivalent
\begin{enumerate}
\item $x$ is $d$-step almost automorphic.

\item $N(x,V)\in \F_{Poi_d}^*$ for each
neighborhood $V$ of $x$.

\item $N(x,V)\in \F_{Boi_d}^*$ for each
neighborhood $V$ of $x$.

\item $N(x,V)\in \F_{d,0}$ for each
neighborhood $V$ of $x$.

\end{enumerate}
\end{thm}


\subsection{Regionally proximal relation of order $d$}

\subsubsection{Regionally proximal relation}
Regionally proximal relation plays a very import role in the theory
of topological dynamics. 
It is the main tool to characterize the equicontinuous structure
relation $S_{eq}(X)$ of a system $(X, T)$; i.e. to find the smallest
closed invariant equivalence relation $R(X)$ on $(X, T)$ such that
$(X/ R(X), T)$ is equicontinuous. Veech \cite{V68} gave the first
proof of the fact that the regionally proximal relation is an
equivalence one. Also he showed that Poincar\'e sets can be used to
characterize regionally proximal relation.

\subsubsection{Regionally proximal relation of order $d$} In
\cite{HK05} Host and Kra defined a $d$-step nilfactor for each
ergodic system, see also \cite{Z}. To get a similar factor in
topological dynamics Host, Maass and Kra introduced the notion of
regionally proximal relation of higher order.

\begin{de}\cite{HM, HKM}
Let $(X, T)$ be a system and let $d\ge 1$ be an integer. A pair $(x,
y) \in X\times X$ is said to be {\em regionally proximal of order
$d$} if for any $\d  > 0$, there exist $x', y'\in X$ and a vector
${\bf n} = (n_1,\ldots , n_d)\in\Z^d$ and $\ep\in \{0,1\}^d$ such
that $d(x, x') < \d, d(y, y') <\d$, and $$\rho(T^{{\bf n}\cdot
\ep}x', T^{{\bf n}\cdot \ep}y') < \d\ \text{for any $\ep\neq
(0,0,\ldots,0)$},$$ where ${\bf n}\cdot
\ep=n_1\ep_1+\ldots+n_d\ep_d$.

The set of regionally proximal pairs of order $d$ is denoted by
$\RP^{[d]}(X)$, which is called {\em the regionally proximal
relation of order $d$}.
\end{de}

It is easy to see that $\RP^{[d]}(X)$ is a closed and invariant
relation for all $d\in \N$. When $d=1$, $\RP^{[d]}(X)$ is nothing
but the classical regionally proximal relation. In \cite{HKM}, for
distal minimal systems the authors showed that $\RP^{[d]}(X)$ is a
closed invariant equivalence relation, and the quotient of $X$ under
this relation is its maximal $d$-step nilfactor. Recently, these
results are shown to be true for general minimal systems by Shao-Ye
\cite{SY}. We remark that a point is $d$-step almost automorphic if
and only if $\RP^{[d]}[x]=\{x\}$. Moreover, we have the following
theorem whose proofs are based on the results built in this paper.

\begin{thm}\label{RP-d}\cite{HSY}
Let $(X,T)$ be a minimal system and $d\in \N$. The following
statements are equivalent:

\begin{enumerate}

\item $(x,y)\in \RP^{[d]}(X)$

\item
$N(x,U)\in \F_{Poi_d}$ for each neighborhood $U$ of $y$.

\item $N(x,U)\in \F_{Bir_d}$ for each
neighborhood $U$ of $y$.

\item $N(x,U)\in \F_{d,0}^*$ for each
neighborhood $U$ of $y$.
\end{enumerate}
\end{thm}

We remark that for $d=1$ the above theorem is known, see for example
\cite{V68, BG, HLY}.

\subsection{Nil$_d$ Bohr$_0$ sets and sets $SG_d(P)$}

In this article to study Nil$_d$ Bohr$_0$ sets, we use generalized
polynomials. In \cite{HK10}, Host and Kra introduced an interesting
set ($SG_d$ set) to study the Nil$_d$ Bohr$_0$ sets. Here we state
some results and give some questions. First we recall definitions
introduced by Host and Kra in \cite{HK10}.

\subsubsection{Sets $SG_d(P)$ and $SG_d^*$}

Let $d\ge 0$ be an integer and let $P=\{p_i\}_i$ be a (finite or
infinite) sequence in $\N$. The {\em set of sums with gaps of length
less than $d$} of $P$ is the set $SG_d(P)$ of all integers of the
form
$$\ep_1p_1+\ep_2p_2+\ldots +\ep_np_n$$ where $n\ge 1$ is an integer,
$\ep_i\in \{0,1\}$ for $1\le i\le n$, the $\ep_i$ are not all equal
to $0$, and the blocks of consecutive $0$'s between two $1$ have
length less than $d$. A subset $A \subseteq \N$ is an $SG^*_d$-set
if $A \cap SG_d(P)\neq \emptyset$ for every infinite sequence $P$ in
$\N$.

Note that in this definition, $P$ is a sequence and not a subset of
$\N$. For example, if $P=\{p_1,p_2,\ldots\}$, then $SG_1(P)$ is the
set of all sums $p_m+p_{m+1}+\ldots +p_n$ of consecutive elements of
$P$, and thus it coincides with the set $\D(S)$ where $S=\{p_1,
p_1+p_2, p_1+p_2+p_3,\ldots\}$. Therefore $SG^*_1$-sets are the same
as $\D^*$-sets.

In \cite{HK10} a notion called  {\em strongly piecewise-$F$},
written PW- $\F$ was introduced and the following proposition was
proved

\begin{prop}
Every $SG_d^*$-set is a PW-Nil$_d$ Bohr$_0$-set.
\end{prop}

Host and Kra asked the following question.
\begin{ques}\label{HK-question} Is every Nil$_d$ Bohr$_0$-set an
$SG_d^*$-set?
\end{ques}

Though we can not answer this question, we show that they can not be
distinguished dynamically (see Theorems \ref{RP-d} and
\ref{thm-SGd}). Using Theorem B, Question \ref{HK-question} can be
reformulated in the following way:

\begin{ques} Let $d\in\N$ and $S$ be an SG$_d$-set. Is it true that for
any $k\in\N$, any $P_1,\ldots, P_k\in \F_{SGP_d}$ and any $\ep_i>0$,
there is $n\in S$ such that
$$P_i(n)(\text{mod}\ \Z)\in (-\ep_i,\ep_i)$$
for all $i=1,\ldots,k$?
\end{ques}

We remark that since a $d$-step nilsystem is distal, the above
question has an affirmative answer for any IP-set.

\subsubsection{$SG_d$ sets and regionally proximal relation of order $d$}

We can use $SG_d$ sets to characterize regionally proximal relation
of order $d$. For each $d\in \N$, denote by $SG_d$ the collection of
all sets $SG_d(P)$, and $\F_{SG_d}$ the family generated by $SG_d$.

\begin{thm}\cite{HSY}\label{thm-SGd}
Let $(X,T)$ be a minimal system. Then for any $d\in\N$,
$(x,y)\in\RP^{[d]}$ if and only if $N(x,U)\in \F_{SG_d}$ for each
neighborhood $U$ of $y$.
\end{thm}

A direct corollary of Theorem \ref{thm-SGd} is: let $(X,T)$ be a
minimal system, $x\in X$, and $d\in \N$. If $x$ is
$\F^*_{SG_d}$-recurrent, then it is $d$-step almost automorphic.
Since $\F_{SG_d}$ does not have the Ramsey property \cite{HSY}, we
do not know if the converse of the above corollary holds.


\end{document}